\documentclass{article}
\usepackage{amsmath, amssymb, amsthm, graphicx, bm, url,color,
  cite,comment}
   
\usepackage[all]{xy}

\usepackage{tikz}
\usetikzlibrary{matrix,arrows,calc,intersections,fit}
\usepackage{tikz-cd}
\usepackage{circuitikz}

\DeclareMathOperator{\Mod}{mod}\DeclareMathOperator{\supp}{supp}\DeclareMathOperator{\obj}{obj}
\DeclareMathOperator{\grad}{grad} 
 \DeclareMathOperator{\diam}{diam}
 \DeclareMathOperator{\tr}{tr}
 \DeclareMathOperator{\di}{div}
 
\DeclareMathOperator{\im}{im}

\DeclareMathOperator{\Tw}{Tw}
\DeclareMathOperator{\sym}{sym}

\DeclareMathOperator{\fct}{fct}
\DeclareMathOperator{\rf}{rf}\DeclareMathOperator{\vg}{vg}

\theoremstyle{plain}
\newtheorem{thm}{Theorem}[section]
\newtheorem{prop}[thm]{Proposition}
\newtheorem{lem}[thm]{Lemma}

\newtheorem{cor}[thm]{Corollary}
\theoremstyle{definition}

\newtheorem{dfn}[thm]{Definition}

\newtheorem{pr}[thm]{Property}

\newtheorem{exam}[thm]{Example}
\newtheorem{rem}[thm]{Remark}
\newtheorem{hp}[thm]{Hypothesis}

\numberwithin{equation}{section}
\numberwithin{figure}{section}

\def\wh{\widehat}
\def\e#1\e{\begin{equation}#1\end{equation}}
\def\iz#1\iz{\begin{itemize}#1\end{itemize}}
\def\ea#1\ea{\begin{align}#1\end{align}}
\def\eq{\eqref}
\def\l{\label}
\def\0{\hspace{0pt}}

\def\Uh_#1{\,\widehat{\!U}_{\!#1}}

\def\fm{{\mathfrak m}}
\def\fn{{\mathfrak n}}

\def\F{{\mathbb F}}

\def\id{{\rm id}}

\def\SL{\mathop{\rm SL}}
\def\GL{\mathop{\rm GL}}

\def\Im{\mathop{\rm Im}}

\def\End{\mathop{\rm End}}

\def\vol{\mathop{\rm vol}\nolimits}

\def\ge{\geqslant}
\def\le{\leqslant\nobreak}

\def\B_#1{\,\ov{\!B_{#1}}}
\def\Ba_#1{\,\ov{\!B_{#1}\!}\,}

\def\cA{{\mathbin{\cal A}}}
\def\cB{{\mathbin{\cal B}}}
\def\cC{{\mathbin{\cal C}}}

\def\cE{{\mathbin{\cal E}}}
\def\cF{{\mathbin{\cal F}}}
\def\cG{{\mathbin{\cal G}}}
\def\cH{{\mathbin{\cal H}}}

\def\cM{{\mathbin{\cal M}}}
\def\Mbar{\overline{\cM}}

\def\cP{{\mathbin{\cal P}}}

\def\={\equiv}

\def\cW{{\mathbin{\cal W}}}
\def\cX{{\mathbin{\cal X}}}

\def\C{{\mathbin{\mathbb C}}}\def\N{{\mathbb N}}
\def\CP{{\mathbin{\mathbb{C}}}P}
\def\RP{{\mathbin{\mathbb{R}}}P}

\def\Q{{\mathbin{\mathbb Q}}}
\def\R{{\mathbin{\mathbb R}}}
\def\Z{{\mathbin{\mathbb Z}}}
\def\al{\alpha}
\def\be{\beta}
\def\ga{\gamma}
\def\de{\delta}
\def\et{\eta}
\def\io{\iota}
\def\ep{\epsilon}
\def\la{\lambda}
\def\ka{\kappa}
\def\th{\theta}
\def\ze{\zeta}

\def\si{\sigma}
\def\om{\omega}
\def\ph{\phi}
\def\Ph{\Phi}
\def\ps{\psi}
\def\Ps{\Psi}
\def\r{\rho}
\def\ta{\tau}
\def\De{\Delta}
\def\La{\Lambda}
\def\Si{\Sigma}

\def\Th{\Theta}
\def\Om{\Omega}
\def\Ga{\Gamma}

\def\bw{\bigwedge}

\def\d{{\partial}}

\def\ts{\textstyle}

\def\w{\wedge}
\def\-{\setminus}

\def\op{\oplus}
\def\bop{\bigoplus}
\def\ot{\otimes}
\def\bot{\bigotimes}
\def\ov{\overline}

\def\iy{\infty}
\def\es{\emptyset}

\def\t{\times}
\def\cm{\circ}

\def\hf{{\frac{1}{2}}}

\def\nb{\nabla}
\def\sb{\subseteq}
\def\k{{\mathbf{k}}}\def\b{{\mathbf{b}}}

\def\nc{{\rm nc}}
\def\v{{\mathbf{v}}}\def\dirlim{\displaystyle\varinjlim}
\def\ex{{   {\rm ex}    }}\def\ev{{\rm ev}}
 \def\unr{{\rm unr}}\def\tame{{\rm tame}}
\title{Nearby Special Lagrangians}
\author{Mohammed Abouzaid and Yohsuke Imagi}
\begin{document}
\maketitle
\begin{abstract}
Let $X$ be a Calabi--Yau manifold and $Q\subset X$ a closed connected embedded special Lagrangian; closed Lagrangians mean compact Lagrangian submanifolds without boundary. We prove that if the fundamental group $\pi_1Q$ is abelian then there exists a Weinstein neighbourhood of $Q\subset X$ in which every closed irreducibly immersed special Lagrangian with unobstructed Floer cohomology is $C^1$ close to $Q.$ We prove also that if $\pi_1Q$ is virtually solvable then for every positive integer $R$ there exists a Weinstein neighbourhood of $Q\subset X$ in which every closed irreducibly immersed special Lagrangian of degree $\le R$ and with unobstructed Floer cohomology is unbranched; that is, the projection $L\to Q$ is a covering map. We prove a stronger statement when $\pi_1Q$ is finite and a weaker statement when $\pi_1Q$ has no non-abelian free subgroups. The $\pi_1Q$ conditions, the Floer cohomology condition and the special Lagrangian condition are all essential as we show by counterexamples.
\end{abstract}
\setcounter{tocdepth}{1}
\tableofcontents
\section{Introduction}
We begin by defining the basic notions we will use. 
\begin{dfn}\l{dfn: sLag}
A {\it Calabi--Yau} manifold is a complex manifold $Y$ of complex dimension $n,$ equipped with a K\"ahler form and a nowhere-vanishing holomorphic $(n,0)$ form $\Om.$ Joyce \cite{J5} and other authors call this an {\it almost} Calabi--Yau manifold to indicate that this need not be Ricci-flat. We do more symplectic topology in the present paper and therefore leave out the word almost so as not to be confused with almost complex structures.

A {\it Lagrangian} in the Calabi--Yau manifold $Y$ is a Lagrangian immersion $\io:\hat L\to Y$ whose intersection points are transverse double points. We refer only to the image $L:=\io(\hat L)\subset X$ when we want to save notation. We say that $\io:\hat L\to Y$ is {\it closed} if $\hat L$ is compact without boundary, {\it embedded} if $\io$ is an embedding of manifolds, {\it irreducibly immersed} if $\hat L$ is connected, and {\it special of phase $\th\in\R$} if $e^{-i\th/2}\io^*\Om$ is a nowhere-vanishing real $n$-form on $\hat L.$ We then orient $\hat L$ by $e^{-i\th/2}\io^*\Om.$ Plainly by a special Lagrangian we mean a special Lagrangian of phase $0\in\R.$ 
\end{dfn}

Let $Y$ be a Calabi--Yau manifold and $Q\subset Y$ a closed embedded special Lagrangian. McLean \cite{ML} determines those closed special Lagrangians which are $C^1$ close to $Q\subset Y.$ We study the $C^0$ version of this problem; that is, we seek a neighbourhood of $Q\subset X$ in which we can determine as many closed special Lagrangians as possible. 

There are some circumstances in which no non-trivial $C^0$ nearby special Lagrangians exist. Tsai and Wang \cite{TW0,TW} prove indeed that if the K\"ahler metric on $Y$ is Ricci-flat and the induced metric on $Q$ Ricci-positive then there exists a neighbourhood of $Q\subset Y$ in which every closed special Lagrangian has image equal to $Q.$ 

There are other circumstances in which many $C^0$ nearby special Lagrangians exist. For instance, let $Y$ be a hyperK\"ahler $4$-manifold and $Q$ a compact Riemann surface of genus $>1,$ embedded in $Y$ as a complex Lagrangian with respect to one of the three complex structures; then the hyperK\"ahler rotations of Hitchin spectral curves are $C^0$ nearby special Lagrangians. We make in Example \ref{higher-dim ex} a higher-dimensional version of this, using special Lagrangian surgeries. Donaldson and He \cite{Don,He,He2} construct more examples of branched special Lagrangians which are locally modelled upon the product of $S^1$ and a $2$-fold spectral curve. These special Lagrangians are all branched as we define now. 

\begin{dfn}\l{dfn: br}
Let $Y$ be a Calabi--Yau manifold and $Q\subset Y$ a closed connected embedded special Lagrangian. A {\it Weinstein neighbourhood} of $Q\subset Y$ means a neighbourhood $X$ which may be embedded, as a compact disc sub-bundle, into the cotangent bundle $T^*Q$ with its standard symplectic form. For a closed irreducibly immersed special Lagrangian $\io:\hat L\subset X$ we define its {\it degree} to be that of the composite $\hat L\to  Q$ of the immersion $\io:\hat L\to X$ and the projection $X\sb T^*Q\to Q;$ the degree is an integer because special Lagrangians are oriented. We say that $\io:\hat L\to X$ is {\it unbranched} if the map $\hat L\to Q$ is a finite cover (of degree equal to what we have just defined). This is equivalent to saying that there exists a finite cover $P\to Q$ such that the pre-image of $L:=\io(\hat L)\subset X$ under the induced map $T^*P\to T^*Q$ is a union of $C^1$ perturbations of $P$ (if $L$ is embedded, this union is disjoint); here a $C^1$ perturbation means the graph over $P$ of some closed $1$-form. We say that $L$ is {\it branched} if $L$ is not unbranched.
\end{dfn}

We prove that no branching occurs under the hypotheses which are weaker than that of Tsai and Wang, except that we can treat only those Lagrangians with unobstructed Floer cohomology. We make a brief definition of the key words about this, referring to \S\ref{sec: Floer} for the actual contents.
\begin{dfn}\l{dfn: HF}
Let $Y$ be a Calabi--Yau manifold, $Q\subset Y$ a closed connected embedded special Lagrangian, and $X$ a Weinstein neighbourhood of $Q\subset Y.$ Let $\La$ be a Novikov field in the sense of Definition \ref{dfn: Novikov}. As $X$ is an exact symplectic manifold we can use a result of Fukaya, Oh, Ohta and Ono \cite{FOOO-Z} by which we can allow $\La$ to have characteristic $p>0.$ Let $L\subset X$ be a special Lagrangian with immersion map $\hat L\to X.$ By a {\it filtered local system} of $L$ we mean a finite-rank non-zero local system on $\hat L$ of valued $\La$-vector spaces in the sense of Definition \ref{dfn: val vect}. We say that $L$ has {\it unobstructed Floer cohomology,} or for short, {\it $HF^*$ unobstructed} if there exists a filtered local system $E$ of $L$ and a bounding cochain of $(L,E)$ in the sense of Definition \ref{dfn: bound coch}. The zero-section $Q\subset X$ for instance has $HF^*$ unobstructed with any choice of local system on this because it has essentially no holomorphic discs. The same holds for its $C^1$ perturbations too. 
\end{dfn}

We state now the main results of the paper.
\begin{thm}\l{thm: main}
Let $Y$ be a Calabi--Yau manifold, $Q\subset Y$ a closed connected embedded special Lagrangian, and $X$ a Weinstein neighbourhood of $Q\subset Y;$ then the following hold.
\iz
\item[\bf(i)] Let the fundamental group $\pi_1Q$ be finite and let $L\subset X$ be a closed irreducibly immersed special Lagrangian of phase $\th\in\R$ with $HF^*$ unobstructed; then $\th\in2\pi\Z$ and $L=Q.$
\item[\bf(ii)]
If $\pi_1Q$ is virtually abelian (that is, having a finite-index abelian subgroup) then there exists a neighbourhood $U\sb X$ of $Q$ such that every closed irreducibly immersed special Lagrangian $L\subset X$ with $HF^*$ unobstructed and contained in $U$ is unbranched; moreover, if $\pi_1Q$ is (strictly) abelian then $L$ is a $C^1$ perturbation of $Q.$
\item[\bf(iii)]
If $\pi_1Q$ is virtually solvable then for every $R>0$ there exists a neighbourhood $U\sb X$ of $Q$ such that the following holds: let $L\subset X$ be a closed irreducibly immersed special Lagrangian of degree $\le R,$ contained in $U,$ and which has $HF^*$ unobstructed with respect to some filtered local system of rank $\le R;$ then $L$ is unbranched.
\item[\bf(iv)]
If $\pi_1Q$ has no non-abelian free subgroups then for every $R>0$ there exists a neighbourhood $U\sb X$ of $Q$ such that the following holds: let $L\subset X$ be a closed irreducibly immersed special Lagrangian of degree $\le R,$ contained in $U,$ and which has $HF^*$ unobstructed with respect to some filtered local system of rank $\le R$ over some characteristic-zero Novikov field; then $L$ is unbranched.
\iz
\end{thm}
If $\pi_1Q$ is finite as in (i) above then the other hypotheses on $\pi_1Q$ in (ii)--(iv) hold too except the strict abelian condition in the latter part of (ii). But the conclusion of (i) is still stronger in two respects: we do not have to shrink the Weinstein neighbourhood $X,$ and we can include special Lagrangians of any phase. By Myers' theorem the hypothesis that $\pi_1Q$ should be finite follows from that of Tsai and Wang who suppose that $Q$ has positive Ricci curvature. 

We give now a sketch of the proof of Theorem \ref{thm: main}. The Lagrangian $L$ has $HF^*$ unobstructed and defines accordingly, as we shall see in \S\ref{sec: Floer}, an object $\b$ of the Fukaya category $\cF(X).$ This is a strict $A_\iy$ category whose hom spaces are valued vector spaces over some Novikov field $\La.$ We recall in \S\ref{sec: Nov} that the algebraic closure of $\La$ is itself a Novikov field. As the bounding cochain of $\b$ may still be defined with coefficients in the larger field, we can suppose that the original Novikov field $\La$ is already algebraically closed.

We embed $\cF(X)$ into the wrapped Fukaya category $\cW(X)$ which include objects supported on compact Lagrangians with boundary. Fix a point $q\in Q$ and by an abuse of notation identify the non-compact fibre $T^*_qQ\subset T^*Q$ with the compact fibre in $X$ over the same point $q.$ This defines then an object of $\cW(X)$ which we denote by the same symbol $T^*_qQ.$ 

The hom spaces of $\cF(X)$ and $\cW(X)$ are not only filtered but also $\Z$-graded $\La$-vector spaces. We denote by $HF^*(T_q^*Q,\b)$ the hom space from $T^*_qQ$ to $\b.$ Since the Lagrangian $L$ is special it follows that $HF^*(T_q^*Q,\b)$ is non-zero and supported in a single degree; the non-zero property will be proved in Corollary \ref{cor: special non-zero}, and the single degree property in Lemma \ref{lem: single degree}. 

We prove then that $\cF(X)$ satisfies the split-generation and other properties. These are known for exact Lagrangians. But no non-trivial special Lagrangian $C^1$ perturbations of $Q$ are exact, so we do include non-exact Lagrangians in $\cF(X)$ and show that the split-generation and other statements are still valid. These imply then that $HF^*(T_q^*Q,\b)$ is a certain representation of the fundamental group $\pi_1Q$ and that this representation determines the object $\b$ up to isomorphism in the cohomology category $H^0\cF(X).$ 

Suppose now that $\pi_1Q$ is strictly abelian as in the last part of Theorem \ref{thm: main} (ii). Since $\La$ is algebraically closed we get then a one-dimensional sub-representation of $HF^*(T_q^*Q,\b).$ This may be realized by a special Lagrangian $C^1$ perturbation of $Q$ with some local system on it. Then by a Thomas--Yau theorem, Proposition \ref{prop: TY} (i), we can show that the original special Lagrangian $L$ agrees with the $C^1$ perturbation of $Q.$ 

Theorem \ref{thm: main} (i) is proved in \S\ref{sec: proof main i}, which is simpler than that of (ii). The chief reason is that the algebraic problem about the $\pi_1Q$ representation will be trivial after we pass to the universal cover of $Q.$

Theorem \ref{thm: main} (iii) is proved in the same way as (ii) is proved, but the representation theory argument will be more complex. The new part will be the use of Lie--Kolchin's theorem and we shall therefore be restricting the representation to a finite-index subgroup of $\pi_1Q.$ Denote by $P\to Q$ the corresponding finite cover. Using the upper bound $R$ given in (iii) we prove in \S\ref{sec: finite McLean} that the Green operator over $P$ (a right inverse to the linearized operator in McLean's deformation theory) is bounded above by a constant independent of the original $L$ and $\b.$ This implies that the $C^1$ perturbation argument will work in $T^*P$ after pulling back the relevant objects in $T^*Q.$ The result is that $L$ is the immersed image of some $C^1$ perturbation of the zero-section $P\subset T^*P.$ So $L$ is unbranched as we have to prove.

The proof of Theorem \ref{thm: main} (iv) is nearly the same as that of (iii). The new part will be the use of Tits' alternative theorem. We shall need the characteristic to be zero for an entirely algebraic reason which we explain in \S\ref{sec: discuss}.



The hypotheses on $\pi_1Q$ in (ii)--(iv) are essential to Theorem \ref{thm: main} (ii)--(iv). Suppose for instance that $Q$ is a compact Riemann surface of genus $>1.$ Then $\pi_1Q$ has non-abelian free subgroups, satisfying none of the hypotheses in (ii)--(iv). Feix \cite{Fx} and Kaledin \cite{Kal} construct a hyperK\"ahler metric on a neighbourhood of the zero-section $Q\subset T^*Q,$ in which we can accordingly define the hyperK\"ahler rotations of spectral curves. These have $HF^*$ unobstructed with respect to any local systems over any Novikov field, because they are of real dimension $2$ (for which trivial bounding cochains exist). 


The strict abelian condition (rather than the virtually abelian condition) is essential to the last part of Theorem \ref{thm: main} (ii). We prove in Corollary \ref{flat} that there are multi-valued graphs over a certain flat manifold $Q.$ This $Q$ has a non-trivial finite cover by $T^n$ and the fundamental group $\pi_1Q$ is accordingly not abelian. 

We show that Lagrangians being special is indispensable to Theorem \ref{thm: main} (ii). We prove in Example \ref{ex: non-special} that there exists an example of $Q$ with $\pi_1Q\cong\Z$ and branched $C^0$ nearby non-special Lagrangians with $HF^*$ unobstructed. We prove in Corollary \ref{cor: non-special} that these are not special even in the weak categorical sense.

We prove in Corollary \ref{O} and in Example \ref{ob ex} that there exist branched $C^0$ nearby special Lagrangians with $HF^*$ obstructed. We verify directly that these properties hold except the $HF^*$ obstructed condition. So if they had $HF^*$ unobstructed, Theorem \ref{thm: main} would be false. Hence it follows that they have $HF^*$ obstructed. There is a further consequence, Corollary \ref{O2}, which is entirely about Morse $1$-forms on $Q,$ having nothing directly to do with special Lagrangians.

We begin in \S\ref{sec: Floer} with the more formal treatment of Fukaya categories. In \S\ref{sec: proof} we prove Theorem \ref{thm: main}. In \S\ref{sec: analytical} we recall Tsai--Wang's result and give the examples related to it. In \S\ref{sec: branch} we give the branched examples and prove the relevant results.

{\bf Acknowledgements.}
This paper is based upon the discussions we had in Fall 2016 at the Institute for Advanced Study in Princeton. We are both grateful to David Treumann for pointing out that spectral curves are the typical counter example to our results, to Dominic Joyce for helping us to find the examples with $HF^*$ obstructed, and to the referees for pointing out the mistakes in the earlier versions of the paper and helping us to improve the presentation. The second author would also like to thank Mark Haskins and Martin Li for interesting conversations which inspired him to begin this work.

The first author was supported by the Erik Ellentuck Fellowship, the IAS Fund of Math, NSF grants DMS-1308179,  DMS-1609148 and DMS-1564172, DMS-2103805 and by the Simons Foundation through its ``Homological Mirror Symmetry'' Collaboration grant.
The second author was supported by the WPI project at the Kavli Institute for Physics and Mathematics of Universe, by the JSPS grants 16K17587, 18J00075 and 21K13788, and by the NSFC grant 11950410501.

\section{Fukaya Categories}\l{sec: Floer}
In this section we explain the definition and properties of the Fukaya categories concerned with the proof of Theorem \ref{thm: main}. We use pseudo-holomorphic curve moduli spaces to define them. As we include non-exact Lagrangians we cannot follow the direct perturbation method of Seidel \cite{Seid3}. We follow instead the virtual perturbation method of Akaho--Joyce \cite{AJ} and Fukaya, Oh, Ohta and Ono \cite{FOOO}. Our treatment is however slightly different from theirs. On the one hand, ours is simpler in that we shall need to define Floer cohomology groups only as $\La$-vector spaces rather than as $\La^0$ modules as in \cite{AJ,FOOO}. The latter has more information but we shall not need it. On the other hand, for our theorem (Theorem \ref{thm: main}) to hold, we need to include filtered local systems (in the sense of Definition \ref{dfn: loc}) in the Fukaya categories. Although this is not explicitly done in \cite{AJ,FOOO} it is a straightforward modification.

Also, as we are concerned with wrapped Fukaya categories we need to include countably many Lagrangians (in \cite{AJ,FOOO} the authors deal with single or finitely many Lagrangians). We do this by the same induction process as in \cite{AJ,FOOO}. Recall that we can deal only with finitely many kinds of Kuranishi spaces in doing the virtual counts \cite[\S7.2.3]{FOOO} and the results are accordingly a finite approximation of what we want. These are called gapped $A_N$ categories in our simplified notation, Definition \ref{dfn: gap} (or gapped $A_{NK}$ algebras in their notation \cite{AJ,FOOO}). By the obstruction theory arguments of \cite{AJ,FOOO} we can pass to the limit called gapped $A_\iy$ categories. In this method the issue \cite[\S7.2.3]{FOOO} above arises every time we have infinitely many things. We do have a new parameter which is the number of Lagrangians. To deal with this we modify the obstruction theory arguments of \cite{AJ,FOOO}, which we do in Theorem \ref{thm: ob1}.

We thus obtain the gapped $A_\iy$ categories we shall need. (These are curved $A_\iy$ categories, including the $\fm^0$ terms. We make them into strict $A_\iy$ categories by adding bounding cochains, following the references \cite{FOOO,AJ}.) We define then the wrapped Fukaya categories by the localization method, which has been used already in the literature \cite{GPS} (the idea is due to Paul Seidel and the first author). The advantage of this is that we can make the definition only from the algebraic properties of Hamiltonian continuations, without going back to the study of pseudo-holomorphic curve moduli spaces.

We prove the split-generation theorem, Theorem \ref{thm: gen}, by a method of Paul Seidel and the first author (taken from their unpublished work). This is close to the proof of Fukaya, Seidel and Smith \cite{FSS1} for exact Lagrangians. They use Lefschetz fibrations, Seidel long exact sequences, and equivariant Fukaya categories; and we replace the last ingredient by the use of Viterbo restriction functors. When working with non-exact Lagrangians, the first two ingredients have already appeared in the literature, so we focus on the discussion of Viterbo restriction functors in the non-exact setting. After providing the argument for this generalisation, the rest of the proof is straightforward.

There is another method \cite{Ab1} of proving the split-generation theorem, using symplectic cohomology (for time-one periodic Hamiltonian orbits). One should then prove that it has the same relation as in the exact case to the algebraic invariants obtained from the Fukaya category.
One could do this in closed symplectic manifolds (as will appear in \cite{AFOOO}) and for monotone Lagrangians (as is done already in \cite{RitterSmith}) but not yet in the circumstances above.

We begin in \S\ref{sec: Nov} by defining Novikov fields and explaining their properties we shall need. In \S\ref{sec: gap} we define the gapped $A_\iy$ categories and provide the obstruction theory arguments. Apart from the issue mentioned above, the major difference from the references \cite{FOOO,AJ} is that we extend the definition of gapped $A_\iy$ algebras; see Remark \ref{rem: difference from FOOO}. 

The gapped structure will be used only for the process of defining the curved $A_\iy$ categories, which we shall in practice be able to treat as {\it filtered} $A_\iy$ categories (in the sense of Definition \ref{dfn: gap}) forgetting the gap conditions.

In \S\S\ref{sec: Fuk} and \ref{sec: Fuk2} we define the curved $A_\iy$ category $\cC(X)$ of a symplectic manifold $X.$
In \S\ref{sec: bounding} we define the strict $A_\iy$ categories $\cF_\nc(X),\cF(X)$ including the bounding cochains.
In \S\ref{sec: TY} we recall the Thomas--Yau theorems we shall need.
In \S\ref{sec: Ham} we explain the effect of symplectic and Hamiltonian diffeomorphisms upon the Fukaya categories.
In \S\ref{sec: loc} we introduce the wrapped Fukaya category $\cW(X)$ by the localisation method. 
In \S\ref{sec: ex} we define the exact Fukaya category $\cF_\ex(X)$ and show that its cohomology category $H\cF_\ex(X)$ is equivalent to that in the literature. 

We turn in \S\ref{sec: cotangent} to the proof of results in cotangent bundles. We prove two lemmas we shall need to deal with $C^1$ perturbations of $Q\subset X$ in the proof of Theorem \ref{thm: main} (ii)--(iv). In \S\ref{sec:gener-fukaya-categ} we prove the split-generation theorem including non-exact Lagrangians. In \S\ref{sec: Yoneda} we deal with the Yoneda functors obtained from the split-generation theorem.

\subsection{Novikov Fields}\l{sec: Nov}
We begin by recalling the definition of non-Archimedean valued fields. We call them valued fields for short. 
\begin{dfn}
By a {\it valued field} we mean the pair $(K,\v)$ of a field $K$ and a function $\v: K\to(-\iy,\iy],$ called a valuation, such that $\v^{-1}(+\iy)=\{0\}$ and for $a,b\in K$ we have $\v(ab)=\v a+\v b$ and $\v(a+b)\ge \min\{\v a,\v b\}.$
By a valued field {\it extension} of $(K,\v)$ we mean a valued field $(K',\v)$ such that $K$ is a sub-field of $K'$ and the two valuations agree on $K.$ We use the same symbol $\v$ for different valuations.

Every valued field $(K,\v)$ has a sub-ring $K^0\subset K$ defined by $K^0:=\{a\in K:\v a\ge0\},$ which is a local ring with maximal ideal $K^+:=\{a\in  K:\v a>0\};$ in other words, $a\in  K^0$ is a unit if and only if $\v a=0.$

Let $(  K,\v)$ be a valued field and define a function $|\,\,|:  K\to[0,\iy)$ by $|a|:=e^{-\v a}$ for $a\in   K.$ Then for $a,b\in   K$ the {\it strong} triangle inequality $|a+b|\le\max\{|a|,|b|\}$ holds and there is a distance function ${\rm d}:  K\times   K\to[0,\iy)$ defined by $(a,b)\mapsto|a-b|,$ making $  K$ into a metric space $(  K,{\rm d}).$
We call $(  K,\v)$ a {\it complete} valued field if $(  K,{\rm d})$ is complete.
\end{dfn}

We generalize the definition of Novikov fields in the references \cite{AJ,FOOO}. 
\begin{dfn}\l{dfn: Novikov}
Define the {\it minimum Novikov field} $(\La^{\min},\v)$ to be the Novikov field in the literature \cite{AJ,FOOO} without the Maslov index factor $e$ in their notation. We begin by recalling this. Let $\k$ be a field and $\k[\R]$ its group ring. Denote by $T\in\k[\R]$ the formal variable corresponding to $1\in \R$ so that every element $a\in\k[\R]$ may be written as $\sum_{\ga\in\R}a_\ga T^\ga,\ a_\ga\in\k$ with $\supp a:=\{\ga\in\R:a_\ga\ne0\}$ a {\it finite} subset of $\R.$ 
Define a function $\v:\k[\R]\to(-\iy,\iy]$ by $a\mapsto\min(\supp a).$ Then the function $|\,\,|:\k[\R]\to[0,\iy)$ defined by $|a|:=e^{-\v a}$ for $a\in \La$ satisfies the strong triangle inequalities, making $\k[\R]$ into a metric space; and $\La^{\min}$ is its completion. The convolution products extend to the completion, making $\La^{\min}$ into a ring. This is in fact a field because of the inversion formula $(1-a)^{-1}=1+a+a^2+\dots$ for $a\in\k[\R]$ with $\v a>0.$ The function $\v$ extends to the completion, making it into a complete valued field $(\La^{\min},\v).$
We denote an element of $\La^{\min}$ by an infinite sum $\sum_{\ga\in\R}a_\ga T^\ga,\ a_\ga\in\k,$ with $\supp a:=\{\ga\in\R:a_\ga\ne0\}$ a {\it discrete} subset of $\R.$

By {\it a Novikov field} we mean a complete valued field extension of $(\La^{\min},\v).$
We call $\k$ the {\it ground} field of the Novikov field.
\end{dfn}

We prove now that algebraically closed Novikov fields exist for arbitrary characteristic. We give two proofs: one is Krasner's lemma, Lemma \ref{lem: Krasn}, a general fact about valued fields, and the other is an explicit construction in Proposition \ref{prop: La} (ii). 

Let $(K,\v)$ be a valued field. Define a subring $K^0\sb K$ defined by
$K^0:=\{a\in K:\v a \ge0\}.$ This has a unique maximal ideal $K^+:=\{a\in K:\v a>0\}$ and the residue field
$\rf K:=K^0/ K^+.$
The {\it value group} $\vg K\sb\R$ is a subgroup of $\R$ defined by
$\vg K:= \v(K\-\{0\})\sb\R.$
We recall the following basic facts:
\begin{prop}\l{prop: ram}
Let $K$ be a complete valued field, and $L/K$ an algebraic extension;
then the valuation of $K$ extends uniquely to $L$ so that:
\iz
\item[\bf(a)] if $L/K$ has finite degree, the valuation on $L$ is complete; and
\item[\bf(b)] if $L/K$ has infinite degree, the valuation of $L$ is strictly not complete. 
\iz
Also, if $L/K$ is finite then
\e\l{fund ineq} [L:K]\ge [\rf L:\rf K][\vg L:\vg K].\e
More precisely, there exist two intermediate fields of $L/K,$
the maximal unramified field $L^\unr$ and the maximal tamely-ramified field $L^\tame,$ such that: 
\iz
\item[\bf(i)]  $K\sb L^\tame\sb L^\unr\sb L;$
\item[\bf(ii)] $\rf K\sb \rf L^\unr=\rf L^\tame \sb \rf L;$ 
\item[\bf(iii)] $\vg K=\vg  L^\unr\sb \vg L^\tame\sb \vg L,$  $ [\vg L^\tame:\vg L^\unr]=[L^\tame:L^\unr];$
\item[\bf(iv)] if $\rf K$ has characteristic $p=0$ then $L=L^\tame,$ and if $p>0$ then $[L:L^\tame]$ is a power of $p.$
\iz
\end{prop}
\begin{proof}
We consult Neukirch \cite{Neuk};
the referred statements in what follows are all from Chapter II of this book.

The unique extension is by Theorem 4.8.
By Proposition 4.9 the extended valuation is equivalent to the $\ell^\iy$ norm on a {\it direct sum} of $K,$
whose elements have only finitely many non-zero components even if they are infinite dimensional. 
So they are complete if and only if finite-dimensional. 
 
The inequality \eq{fund ineq} is from Proposition 6.8.
The field $L^\unr$ is defined in Definition 7.4, and $L^\tame$ in Definition 7.10 but under the hypothesis that $\rf K$ has positive characteristic;
if $\rf K$ has characteristic $0$ we define $L^\tame:=L.$
The parts (i), (ii) and (iv) are then obvious.
The part (iii) follows from Proposition 7.7; this is stated again in the positive-characteristic case,
but its proof applies also to the characteristic-zero case.
\end{proof}

Here the extension $L/L^\tame$ is called the {\it wild} part of $L/K,$
which is in general difficult to study.
But it has degree a power of $p;$ and for {\it Galois} extensions of such a degree, there is a classification result:
\begin{prop}[Artin--Schreier]\l{prop: Artin--Schreier}
Let $K$ be any field of characteristic $p.$
Then for every $a\in K$ either:
\iz
\item[\bf(i)]
the polynomial $f_a=f_a(x):=x^p-x-a\in K[x]$ has a root $b\in K,$
in which case its roots $b,b+1,\cdots,b+p-1$ are all in $K;$ or
\item[\bf(ii)]
$f_a\in K[x]$ is irreducible, in which case $K[x]/f_a$
is a degree-$p$ Galois extension of $K.$
\iz
Conversely:
\iz
\item[\bf(iii)] every degree-$p$ Galois extension of $K$ is of the form $K[x]/f_a$ for some $a\in K;$ or more generally
\item[\bf(iv)] every finite Galois extension of $K$ of degree equal to a power of $p$ is a repeated extension of this kind.
\iz
\end{prop}
\begin{proof}
For (i)--(iii) consult for instance Lang \cite[Chapter VI, Theorem 6.4]{Lang};
and then for (iv) use the following fact \cite[Chapter I, Corollary 6.6]{Lang}:
for every finite group $G$ of order a power of $p$ there exists a normal series
\e G=G_0\supseteq G_1\supseteq \cdots \supseteq G_n=\{1\}\e
with $[G_i:G_{i-1}]=p$ for every $i.$
\end{proof}

Finally we recall:
\begin{lem}[Krasner]\l{lem: Krasn}
Every valued field $K$ has a complete algebraically-closed extension.
\end{lem}
\begin{proof}
Let $L$ be the algebraic closure of $K.$
The valuation extends uniquely to $L$ by the first assertion in Proposition \ref{prop: ram},
because $L$ is the union of finite subextensions of $K.$
The completion of this valued field $L$ is algebraically closed, by Krasner's lemma \cite[Chapter II, \S6, Exercise 2]{Neuk}.
\end{proof}

We turn now to the second method. 
Let $\k$ be a field. In Definition \ref{dfn: Novikov} we have defined the minimum Novikov field $\La^{\min}$ with ground field $\k.$
We define the {\it maximum} Novikov field by
\e
\l{La2}
\La^{\max}:=\left\{a=\sum_{\ga\in\R}a_\ga T^\ga: a_\ga\in \k\text{ and }\supp a\sb\R\text{ well-ordered}\right\}
\e
where $\supp a:=\{\ga\in\R:a_\ga\ne0\}.$
Recall that $\supp a$ being well ordered means that every non-empty subset has a minimum. 
In particular, $\supp a$ is {\it half discrete} in the sense that
for every $\ga\in\supp a$ 
there exists $\ep>0$ such that
\e\supp a\cap(\ga,\ga+\ep)=\es;\e
because we can take $\ga+\ep$ to be the minimum of $\supp a\cap(\ga,\iy).$
The sums and products in $\La^{\max}$ are defined in the natural way: for $a,b\in\La^{\max}$ we can define
\e
\l{a+b} a+b:=\sum_{\ga\in \supp a\cup\supp b}(a_\ga+b_\ga)T^\ga
\e
because $\supp a\cup\supp b$ is well ordered; and
\e
\l{ab} ab:=\sum_\ga\left(\sum_{\al+\be=\ga}a_\al b_\be\right) T^\ga
\e
because $\sum_{\al+\be=\ga}$ is finite, which follows since $\supp a,\supp b$ are half discrete. 

The valuation $\v:\La^{\max}\to(-\iy,+\iy]$ is defined by setting $\v0:=+\iy$ and by taking in \eq{La2} the least $\ga\in\supp a.$
Then $(\La^{\max},\v)$ is a complete valued field.
We have
\e\rf \La^{\min}\cong\rf\La^{\max}\cong\k\text{ and }\vg \La^{\min}\cong\vg\La^{\max}=\R.\e
We prove now the following proposition. Part (ii) with $\k$ algebraically closed is the result we have promised to prove. Recall that a field $\La$ is a {\it perfect field} if either this has characteristic $p=0$ or the ring homomorphism $\La\to \La$ defined by $x\mapsto x^p$ is surjective.
\begin{prop}\l{prop: La}
\iz
\item[\bf(i)]
$\La^{\min},\La^{\max}$ are both perfect fields.
\item[\bf(ii)]
$\La^{\max}$ is maximal in Krull's sense {\rm\cite{Krull};} that is, $\La^{\max}$ has no non-trivial valued field extension $K$ with $\rf K\cong\k$ and $\vg K=\R.$
In particular, if $\k$ is algebraically closed, so is $\La^{\max}.$ 
\item[\bf(iii)]
If $\k$ is algebraically closed of characteristic $0,$ so is $\La^{\min}.$
\\
If $\k$ has positive characteristic $p$ then $\La^{\min}$ is not algebraically closed,
and its algebraic closure is strictly not complete.
\iz
\end{prop}
\begin{rem}
These results do not assert that $\La^{\max}$ is the algebraic closure of $\La^{\min}.$ For rational-power series, the algebraic closure is described in Kedlaya \cite{Kedl}.
\end{rem}

\begin{proof}[ Proof of Proposition $\ref{prop: La}$]
For (i) 
let $\La$ be either $\La^{\min}$ or $\La^{\max}.$ We suppose $p>0$ and prove that for every $a\in\La$ there exists $b\in\La$ with $b^p=a.$
For $a=0$ we can take $a=0.$
For $a\ne0$ we can write $a=\sum_{\ga\in\supp a} a_\ga T^\ga,$ $a_\ga\in\k.$ For each $\ga\in\supp a$ choose a root $b_\ga\in \k$ of the polynomial $x^p-a_\ga\in \k[x],$ which exists because $\k$ is algebraically closed.
Put $b:= \sum_{p\ga\in\supp a}b_{\ga/p} T^\ga.$ As in \eq{ab} the binomial expansion is valid so that 
\e b^p=  \sum_{\substack{p\de_1,\dots,p\de_p\in \supp a\\ \de_1+\dots+\de_p=\ga}}b_{\de_1}\cdots b_{\de_p} T^{\ga}=\sum_{\substack{p\de_1,\dots,p\de_p\in\supp a\\ \de_1=\dots=\de_p=\ga/p}}b_{\de_1}\cdots b_{\de_p} T^{\ga}=\sum_{\ga\in\supp a}b_{\ga/p}^p T^{\ga}=a\e
where the second equality follows since $\k$ has characteristic $p>0.$ Thus $b$ is an element we want.

The former part of (ii) is an old result \cite{Kap,Krull}.
For the latter, if $\k$ is algebraically closed, then $\rf\La\cong \k$ is algebraically maximal and $\vg\La=\R$ is already maximal in $\R;$
so $\La$ has no non-trivial finite extension, or equivalently $\La$ is algebraically closed. 

The former part of (iii) is proved by Fukaya, Oh, Ohta and Ono \cite[Lemma A.1]{FOOO2} but we give another proof, using Proposition \ref{prop: ram} with $K=\La^{\min}.$ Suppose therefore that $L/K$ is a finite extension, and note that:
\iz
\item
since $\rf K$ is algebraically closed it follows that the inclusions in (ii) are all equalities, so $K=L^\unr;$
\item
since $\vg K=\R$ it follows that the inclusions in (iii) are all equalities, so $L^\unr=L^\tame;$ and
\item
$L=L^\tame$ because $\rf K$ has characteristic zero.
\iz
Consequently, $K=L;$ that is, $K=\La$ has no non-trivial finite extension, or equivalently, $K$ is algebraically closed.  

For the latter part of (iii) let $\ga<0$ and consider the Artin--Schreier polynomial
\e\l{ArtSch} x^p-x-T^\ga\in\La^{\min}[x].\e
If this had a root of the form $x=\sum_{n=1}^\iy x_nT^{\ga_n}$ then by induction on $n$ it must be of the form 
\e\l{root}\sum_{n=1}^\iy T^{-\ga/p^n}\in\La^{\max}\-\La^{\min};\e
so $\La^{\min}$ is not algebraically closed.

We show next that the algebraic closure of $\La^{\min}$ has infinite degree over $\La^{\min}.$
By an {\it Artin--Schreier} element over $K:=\La^{\min}$ we shall mean a root of an Artin--Schreier polynomial over $K.$
Let $L/K$ be the field extension which contains all the Artin--Schreier elements over $K.$
We prove the following three facts in this order: 
\iz
\item[\bf(I)]
if $a_0,a_1,a_2\in K$ with $a_0=a_1+a_2$ and if $\al_1,\al_2,\al_3\in L$ are respectively roots of the  
polynomials $x^p-x-a_i\in K[x],$ $i=0,1,2,$ then 
$K(\al_0)\sb K(\al_1,\al_2);$
\item[\bf(II)]
if $\al$ is an Artin--Schreier element over $K$ then
$ K(\al)\sb K(\be;\de_1,\cdots,\de_n)$
where $\be$ is a root of the polynomial $x^p-x-b\in K[x]$ for some $b\in K$ with $\v b\ge0,$
and every $\de_i$ a root of the polynomial $x^p-x-c_iT^{\ga_i}\in K[x]$ for some $c_i\in\k$ and $\ga_i<0;$ and
\item[\bf(III)]
if the algebraic closure of $\La^{\min}$ had {\it finite} degree over $\La^{\min},$
we should have
$L=K(\be_1,\cdots,\be_m;\de_1,\cdots,\de_n)$
where $\be_1,\cdots,\be_m$ are such as $\be$ in (II) and $\de_1,\cdots,\de_n$ are under the same condition as in (II) but that $n$ may of course be different. 
\iz
For (I) note that $\al_1+\al_2$ is a root of the polynomial $x^p-x-a_0\in K[x],$ so $\al_0=\al_1+\al_2+c$ for some $c\in\{0,1,\cdots,p-1\},$
which implies $K(\al_0)\sb K(\al_1,\al_2)$ as we want.
For (II) write $\al$ as a root of the polynomial $x^p-x-a\in K[x],$ $a\in K=\La^{\min},$ and recall that $a$ has by definition only finitely many terms of negative powers with respect to $T;$ that is, $a=b+\sum c_iT^{\ga_i},$ $\v b\ge0$ and $\ga_i<0.$
Applying to this the part (I) repeatedly, we see that $K(\al)\sb K(\be;\de_1,\cdots,\de_n)$ as we want.
For (III) it is clear from definition that
$L\supseteq K(\be_1,\cdots,\be_m;\de_1,\cdots,\de_n).$
On the other hand, by our hypothesis $L/K$ is a finite extension and so is obtained from $K$ by adding finite Artin--Schreier elements.
Hence using (II) repeatedly we get the other inclusion we want.

Now in (III) write $\de_i$ as a root of the polynomial $x^p-x-c_iT^{\ga_i}\in K[x]$ for some $c_i\in\k$ and $\ga_i<0.$
Then $\de_i=\sum_{n=1}^\iy c_i^{1/p^n}T^{\ga_i/p^n}$ for some $c_i^{1/p^n}\in\ov\k,$ the algebraic closure of $\k,$ so
\e\l{Art4} L\sb\left\{\sum a_\ga T^\ga:a_\ga\in\ov\k\text{ and if $\ga<0$ then $\ga\in \Q(p;\ga_1,\cdots,\ga_n)$}\right\}.\e
But as $\R/\Q$ is an infinite extension there is some $\ep<0$ that is not in $\Q(p;\ga_1,\cdots,\ga_n).$
Then by \eq{Art4} the Artin--Schreier element $\sum_{n=1}^\iy T^{\ep/p^n},$ a root of the polynomial $x^p-x-T^{\ep}\in K[x],$
is not in $L,$ which contradicts the definition of $L.$

So the algebraic closure of $\La^{\min}$ has infinite degree over $\La^{\min}$ and by Proposition \ref{prop: ram} (b) it is strictly not complete.  
\end{proof}
\begin{rem}
The r\^ole of Artin--Schreier polynomials is essential to the proof above.
In fact, if $\k$ is algebraically closed then every finite {\it normal} extension $L/\La^{\min}$ is a repeated Artin--Schreier extension, as we show now.
In the notation of Proposition \ref{prop: ram} with $K=\La^{\min},$ we have again $K=L^\tame.$
On the other hand, by (i) above, $L/L^\tame$ is automatically Galois.
To this we can apply Proposition \ref{prop: Artin--Schreier} so $L/\La^{\min}$ is a repeated Artin--Schreier extension as we want.
\end{rem}

\subsection{Gapped $A_\iy$ Categories}\l{sec: gap}
We begin by defining the notion of valued vector spaces. As in Definition \ref{dfn: HF} we will use these to define filtered local systems of Lagrangians. 
\begin{dfn}\l{dfn: val vect}
Let $(\La,\v)$ be a Novikov field. By a {\it valued $\La$-vector space} we mean the pair $(E,\v)$ of a $\La$-vector space $E$ and a function $\v:E\to(-\iy,+\iy],$ called also the {\it valuation,} such that
$\v^{-1}(+\iy)=\{0\};$ for $x,y\in E$ we have $\v(x+y) \ge \min\{\v x,\v y\};$ and 
for $a\in\La,\ x\in E$ we have $\v(a x)=\v a+\v x$ where we use both $\v:\La\to(-\iy,\iy]$ and $\v:E\to(-\iy,\iy].$

Define a function $|\,\,|:E\to[0,\iy)$ by $|x|:=e^{-\v x}$ for $x\in E.$ Then the strong triangle inequality $|x+y|\le\max\{|x|,|y|\}$ holds for $x,y \in E;$ and in particular, there is a distance function ${\rm d}:E\times E\to[0,\iy)$ defined by $(x,y)\mapsto|x-y|,$ making $E$ into a metric space $(E,{\rm d}).$
We call $(E,\v)$ a {\it complete} valued vector space if $(E,{\rm d})$ is a complete metric space.
The valued vector space structure of $(E,\v)$ extends to the metric space completion of $(E,{\rm d}),$ making it into a complete valued vector space. We call this the {\it completion} of $(E,\v).$

Let $(E,\v),(F,\v)$ be valued $\La$-vector spaces and $\al:E\to F$ a $\La$-linear map which is continuous with respect to the metric space structures above.
Then there exists $\v\al\in(-\iy,\iy]$ by $\v\al:=\inf_{x\in E} (\v (\al x)-\v x);$ that is, $|\al|:=e^{-\v\al}$ is the operator norm of the bounded operator $\al:(E,|\,\,|)\to(F,|\,\,|).$
We say that $\al$ is {\it filtered} if $\v\al\ge0.$
We call it a {\it filtered isomorphism} if $\al$ is a linear isomorphism with $\al,\al^{-1}$ filtered. The latter is equivalent to $\v\al=\v(\al^{-1})=0$ or to $|\al|=|\al^{-1}|=1$ because $|\al\be|\le|\al||\be|$ for linear maps $\al,\be$ between valued vector spaces. 

Let $(E_1,\v),\dots,(E_d,\v),(F,\v)$ be valued $\La$-vector spaces and $\al:E_1\times\dots\times E_k\to F$ a $\La$ multi-linear map.
Define then $\v\al\in[-\iy,\iy]$ by
\e \v\al:=\ts\inf_{(x_1,\dots,x_k)\in E_1\times\dots\times E_k} (\v[\al (x_1,\dots,x_k)]-\v x_1-\dots-\v x_k).\e
We say that $\al$ is {\it filtered} if $\v\al\ge0.$
We denote by $\hom(E_1,\dots,E_k;F)$ the set of filtered $\La$ multi-linear maps from $E_1\times\dots\times E_k$ to $F.$
The pair $(\hom(E_1,\dots,E_k;F),\v)$ has the obvious structure of a valued $\La$-vector space.
If $(F,\v)$ is complete then so is $(\hom(E_1,\dots,E_k;F),\v).$

Let $(E_i,\v)_{i\in I}$ be a family of complete valued $\La$-vector spaces. Then there exists a valuation $\bigoplus_{i\in I} E_i\to(-\iy,\iy]$ defined by $(x_i)_{i\in I}\mapsto \min_{i\in I}\v x_i.$
The {\it complete direct sum} $(\wh\bigoplus_{i\in I}E_i,\v)$ is the completion of the direct sum $(\bop_{i\in I}E_i,\v).$
\end{dfn}

We give a basic example of finite-dimensional valued $\La$-vector space.
\begin{exam}\l{ex: ell inf}
Let $n>0$ be an integer and $(E,\v)$ an $n$-dimensional valued $\La$-vector space. We give $E$ another valuation. Let $\bm e:=\{e_1,\dots,e_n\}\subset E$ be a basis with $\v e_1=\dots=\v e_n=0,$ which exists because $T^\ga\in\La$ for every $\ga\in\R.$
Then $E$ has a valuation $\v_{\bm e}:E\to(-\iy,\iy]$ defined by $x=x_1e_1+\dots+x_n e_n\mapsto\min\{\v x_1,\dots,\v x_n\}$ for $x_1,\dots,x_n\in\La.$ Note that by the strong triangle inequality we have $\v x\ge \v_{\bm e} x$ for every $x\in E.$ We call $\v$ the {\it $\ell^\iy$ valuation with respect to $\bm e$} if the equality $\v x=\v_{\bm e}x$ holds for every $x\in E.$ 

More generally, when the basis $\bm e$ is not specified, we call $\v$ plainly an {\it $\ell^\iy$ valuation} if it is so with respect some basis $\bm e$ of $E$ with $\v e_1=\dots=\v e_n=0.$ Recall from Definition \ref{dfn: Novikov} that we have the subring $\La^0\subset\La.$ Regard $E$ as an $\La^0$ module and define a $\La^0$ submodule $E^0\subset E$ by $E^0:=\{e\in E:\v e\ge 0\}.$ Then $\v$ is an $\ell^\iy$ valuation if and only if $E^0$ is a free $\La^0$ module. More precisely, given a finite set $\bm e\subset E^0$ with $\v e_1=\dots=\v e_n=0,$ we have $\v=\v_{\bm e}$ if and only if $E^0$ is the free $\La^0$ module with basis $\bm e.$ 

We give two other equivalent conditions. Note that $\La^+E^0:=\{ae: a\in\La^+,\ e\in E^0\}$ is a $\La^0$ submodule of $E^0$ and the quotient $\bar E:=E^0/\La^+E^0$ has the vector space structure over $\La^0/\La^+=:\k.$
We prove that the following three conditions are equivalent: {\bf(i)} $E^0$ is a rank-$n$ free $\La^0$ module; {\bf(ii)} $E^0$ is a finitely generated $\La^0$ module; and {\bf(iii)} the $\k$-vector space $E^0/\La^+E^0$ has dimension $\ge n.$
Clearly (i) implies (ii). We prove that (ii) implies (iii). Take $e_1,\dots,e_p\in E^0$ such that their images $\bar e_1,\dots,\bar e_p\in E^0/\La^+E^0$ are a $\k$-basis. Then as $E^0$ is finitely generated we can use Nakayama's lemma \cite[Chapter X, Lemma 4.3]{Lang} (for non-Noetherian local rings) which implies $E^0=\La^0e_1+\dots+\La^0e_p.$ So $E=\La e_1+\dots+\La e_p$ and $p\ge n,$ which is the condition (iii).
We prove now that (iii) implies (i). Take $e_1,\dots,e_n\in E^0$ such that their images $\bar e_1,\dots,\bar e_n\in E^0/\La^+E^0$ are linearly independent over $\k.$
Then $F:=\La^0e_1+\dots+\La^0e_n\sb E^0$ and we show that the equality holds.
Otherwise there is some $x\in E^0\-F.$
By re-ordering $e_1,\dots,e_n$ we can write $x=a_1T^{\ga_1}e_1+\dots+a_mT^{\ga_m}e_m$ where $m\le n,$ $a_1,\dots,a_m\in E\-\{0\}$ with $\v a_1=\dots=\v a_m=0,$ $\ga_1=\dots=\ga_l<\ga_{l+1}\le\dots\le\ga_m.$
Then
\[T^{-\ga_1}x=a_1e_1+\dots+a_le_l+a_{l+1}T^{\ga_{l+1}-\ga_1}+\dots+a_mT^{\ga_m-\ga_1}e_m\]
and $T^{-\ga_1}x,a_{l+1}T^{\ga_{l+1}-\ga_1}+\dots+a_mT^{\ga_m-\ga_1}e_m$ have positive valuation. So $\bar a_1\bar e_1+\dots
+\bar a_l\bar e_l=0$ in $\bar E.$ But none of $\bar a_1,\dots,\bar a_l$ is zero, which contradicts that $\bar e_1,\dots,\bar e_n$ are linearly independent. Thus (i)--(iii) are equivalent.
\end{exam}

\begin{rem}\l{rem: ell inf}
Consider the category whose objects are $n$-dimensional valued $\La$-vector spaces, whose morphisms are filtered homomorphisms.  
Then there is an equivalence from the category of rank-$n$ free $\La^0$ modules to that of $n$-dimensional valued vector spaces with $\ell^\iy$ valuations. The function between the object sets is defined by $F\mapsto F\otimes_{\La^0}\La.$ Each morphism $\ph\in\hom(F,F')$ is mapped to $\ph\otimes\id.$

When $n=1$ every valuation is the $\ell^\iy$ valuation. So the category of rank-one free $\La^0$ modules is equivalent to that of one-dimensional valued vector spaces. 
\end{rem}
We define then the filtered local systems. 
\begin{dfn}\l{dfn: loc}
By a {\it filtered} local system over a topological space we mean one of non-zero finite-dimensional valued vector spaces whose parallel transports are filtered homomorphisms of valued vector spaces.
\end{dfn}

We make now the definitions about discrete sub-monoids of $[0,\iy).$ These will be concerned with the areas of pseudo-holomorphic curves with Lagrangian boundary conditions.
\begin{dfn}
Denote by $\N$ the set of non-negative integers.
Let $\Ga\subset[0,\iy)$ be a discrete {\it sub-monoid,} that is, a discrete subset containing $0$ and closed under addition. 
By a {\it decomposition} of $(k,\ga)\in(\N\t\Ga)\-\{(0,0)\}$ we mean $l\in\N\-\{0\},$ $a=(a_0,\dots,a_l)\in\N^{\times (l+1)}$ with $0=a_0\le\dots\le a_l=k,$ and $\al=(\al_0,\dots,\al_l)\in\Ga^{\times (l+1)}$ with $\al_0+\dots+\al_l=\ga,$ such that $(l,\al_0),(a_1-a_0,\al_1),\dots,(a_l-a_{l-1},\al_l)\ne(0,0).$
By a {\it reduced} decomposition of $(k,\ga)$ we mean $p,q\in\N$ and $\al_0,\al_1\in\Ga$ with $0\le p\le q\le k,$ $\al_0+\al_1=\ga$ and $(k+p-q+1,\al_0),(q-p,\al_1)\in(\N\times\Ga)\-\{(0,0)\};$
see Equation \eq{A_iy} below for a typical formula where this notion is used.
Every decomposition $(l,a,\al)$ with $a_0=0,\dots,a_p=p$ and $a_{p+1}=q,\dots, a_l=q+l-p-1$ for some $p,q$ defines indeed the reduced decomposition $(p,q;\al_0,\al_1).$

Choose a total order $\le$ on $(\N\t\Ga)\-\{(0,0)\}$ such that there is an ordered set isomorphism $((\N\t\Ga)\-\{(0,0\},\le)\cong(\N,\le),$ which we denote by $(k,\ga)\mapsto [k,\ga],$ and if $(l,a,\al)$ are as above a decomposition of $(k,\ga)\in(\N\times\Ga)\-\{(0,0)\}$ then 
\e\l{kga} \max\{[l,\al_0],[a_1-a_0,\al_1],\dots,[a_l-a_{l-1},\al_l]\}\le [k,\ga]\e
with equality if and only if at least $l$ of $(l,\al_0),(a_1-a_0,\al_1),\dots,(a_l-a_{l-1},\al_l)$ are equal to $(1,0).$
Setting $(l,\al_0)=(k,\ga)$ and $(a_1-a_0,\al_1),\dots,(a_l-a_{l-1},\al_l)=(1,0)$ we get $(1,0)\le(k,\ga);$ that is, $(1,0)$ is the least element of $(\N\t\Ga)\-\{(0,0\}.$ In other words, $[k,\ga]=0$ if and only if $(k,\ga)=(1,0).$

We choose the order $\le$ for {\it every} discrete sub-monoid $\Ga\subset[0,\iy)$ so that if $\Ga,\De\subset[0,\iy)$ are two discrete sub-monoids with $\Ga\sb\De$ then the induced map $(\N\t\Ga)\-\{(0,0\}\to(\N\t\De)\-\{(0,0\}$ will preserve the orders. 
Here is an explicit way of doing this. Set $[0]:=0$ and for $\ga\in\Ga\-\{0\}$ set
\e\l{split Gamma}
[\ga]:=\max\{n\in\N:\ga=\ga_1+\cdots+\ga_n;\ \ga_1,\cdots,\ga_n\in\Ga\-\{0\}\}.\e
The maximum makes sense because $\Ga\sb[0,\iy)$ is discrete and has the least positive element.
It is clear that $[\ga+\de]\ge[\ga]+[\de]$ for $\ga,\de\in\Ga.$ For $(k,\ga),(l,\de)\in(\N\times\Ga)\-\{(0,0)\}$ define $(k,\ga)\le (l,\de)$ by either $k+[\ga]+\ga<l+[\de]+\de$ or $k+[\ga]+\ga=l+[\de]+\de$ and $\ga\le\de.$ Then we can check all the conditions above.
\end{dfn}

Let $(\La,\v)$ be a Novikov field, which we shall fix in this section.
We define the notion of gapped $A_N$ categories, $N\in\N\sqcup\{\iy\},$ over $(\La,\v).$
The morphisms spaces are $\Z$-graded.
The $A_\iy$ structures are {\it curved,} including the $\fm^0$ terms.  
\begin{dfn}\l{dfn: gap}
We begin with a {\it filtered} curved $A_\iy$ category $(\cA,\fm)$ which consists of an object set $\obj\cA,$ a family $(\cA_{XY},\v)_{X,Y\in\obj\cA}$ of $\Z$-graded complete valued $\La$-vector spaces, and a family
\e 
(\fm^k=\fm^k_{X_0\dots X_k}\in \hom^1(\cA_{X_0X_1}[1],\dots,\cA_{X_{k-1}X_k}[1];\cA_{X_0X_k}[1])),
\e
$k\in\N$ and $X_0,\dots,X_k\in\obj\cA,$ such that 
\e\l{A_iy}
0=\sum_{p,q}(-1)^\#\fm^{k+p-q+1}(x_1,\dots,x_p,\fm^{q-p}(x_{p+1},\dots,x_q),x_{q+1},\dots,x_k)
\e
where 
$x_1\in \cA_{X_0X_1}[1],\dots, x_k\in\cA_{X_{k-1}X_k}[1];$ $\#:=\deg x_1+\dots+\deg x_p;$ and $p,q$ are any integers with $0\le p\le q\le k.$
For $k=0$ we have $\fm^0_X\in \cA^1_{XX}$ for every $X\in\obj\cA,$ and \eq{A_iy} means $\fm^1_{XX}\fm^0_X=0\in\cA^2_{XX}.$

Suppose now given $N\in\N\sqcup\{\iy\}.$
By a {\it gapped} $A_N$ category $(\cA,\fm)$ we mean an object set $\obj\cA,$ a family $(\cA_{XY},\v)_{X,Y\in\obj\cA}$ of $\Z$-graded complete valued $\La$-vector spaces, a discrete sub-monoid $\Ga\subset[0,\iy)$ and a family
$(\fm^k_\ga=\fm^{k\ga}_{X_0\dots X_k}\in\hom^1(\cA_{X_0X_1}[1],\dots,\cA_{X_{k-1}X_k}[1];\cA_{X_0X_k}[1])),$
$(k,\ga)\in(\N\times\Ga)\-\{(0,0)\}$ with $[k,\ga]\le N$ and  $X_0,\dots,X_k\in\obj\cA,$ such that
\e\l{gapped A_N}
0=\sum_{p,q;\al,\be}(-1)^\#\fm_\al^{k+p-q+1}(x_1,\dots,x_a,\fm_\be^{q-p}(x_{p+1},\dots,x_q),x_{q+1},\dots,x_k)
\e
where $x_1,\dots,x_k,\#,p,q$ are as in \eq{A_iy} and the sum is taken over reduced decompositions $(p,q;\al,\be)$s of $(k,\ga).$ Here the hom space
\e\hom^1(\cA_{X_0X_1}[1],\dots,\cA_{X_{k-1}X_k}[1];\cA_{X_0X_k}[1]))\e
consists of {\it filtered} vector space homomorphisms (as in Definition \ref{dfn: val vect}); the superscript $1$ of $\hom^1$ refers to the degree, which has nothing to do with the filtrations. In \eq{gapped A_N} the operators $\fm_\al^{k+p-q+1},\fm_\be^{q-p}$ are well-defined because $\max\{[q-p,\al],[k+p-q+1,\be]\}\le N$ by \eq{kga}.

When $N=\iy$ the gapped $A_\iy$ category $(\cA,\fm)$ defines a filtered $A_\iy$ category of the same object set, the same morphism spaces, and the $\fm^k$ operators defined by $\fm^k_{X_0\dots X_k}:=\sum T^\ga \fm^{k\ga}_{X_0\dots X_k}$ where $\ga\in\Ga\-\{0\}$ for $k=0$ and $\ga\in\Ga$ for $k>0.$
The sum is well-defined because $\hom^1(\cA_{X_0X_1}[1],\dots,\cA_{X_{k-1}X_k}[1];\cA_{X_0X_k}[1])$ is complete.
\end{dfn}
\begin{rem}\l{rem: difference from FOOO}
The definition of $\hom^1(\cA_{X_0X_1}[1],\dots,\cA_{X_{k-1}X_k}[1];\cA_{X_0X_k}[1]))$ implies that every $\fm^{k\ga}_{X_0\dots X_k}$ is a {\it filtered} homomorphism between valued vector spaces.
This will be crucial to defining the Fukaya categories with filtered local systems.
In the references \cite{AJ,FOOO} they define $\fm^{k\ga}_{X_0\dots X_k}$ to be linear over the residue field $\La^0/\La^+,$ determined {\it uniquely} from the sum $\fm^k_{X_0\dots X_k};$ this is not true in our circumstances.

They also work with gapped $A_N$ {\it algebras} which are gapped $A_N$ categories with a single object; or they are expressible as the pair $(A,\fm)$ of a $\Z$-graded complete valued $\La$-vector space $A$ and a family $\fm=(\fm^k_\ga\in\hom^1(A,\dots,A;A)),$ $(k,\ga)\in(\N\times\Ga)\-\{(0,0)\}$ with $[k,\ga]\le N,$ satisfying the same equations. 
If $(\cA,\fm)$ is a gapped $A_N$ category then $\wh\bop_{X,Y\in\obj\cA}\cA_{XY},$ the complete direct sum, has a gapped $A_N$ algebra structure $\fm$ defined in the obvious way. Conversely, let $(A,\fm)$ be a gapped $A_N$ algebra with $A=\wh\bop_{X,Y\in\cX}\cA_{XY}$ such that the induced map $\fm^k_\ga:A_{X_1Y_1}\times\dots\times A_{X_kY_k}\to A$ is zero unless $Y_1=X_2,\dots,Y_{k-1}=X_k;$ and if this is the case then $\fm^k_\ga$ has image in $A_{X_1Y_k}.$ Then $\cX,(\cA_{XY})_{X,Y\in\cX}$ and $\fm$ define a gapped $A_N$ category.
\end{rem}

We define the truncations of gapped $A_N$ functors, $N\in\N\sqcup\{\iy\},$ and the full subcategories of gapped $A_N$ functors.
\begin{dfn}
Let $M,N\in\N\sqcup\{\iy\}$ with $M\le N$ and $(\cA,\fm)$ a gapped $A_N$ category.
The $A_M$ {\it truncation} of $(\cA,\fm)$ means the same object set $\obj\cA,$ the same morphism spaces $(\cA_{XY})_{X,Y\in\obj\cA}$ and the subfamily $(\fm^{k\ga}_{X_0\dots X_k})_{[k,\ga]\le M};$ these define a gapped $A_M$ category. 
By a {\it full subcategory} $(\cB,\fm)$ of $(\cA,\fm)$ we mean a subset $\obj\cB\sb\obj\cA,$ the subfamily $(\cA_{XY})_{X,Y\in\obj\cB}$ and the subfamily $(\fm^{k\ga}_{X_0\dots X_k})_{X_0,\dots,X_k\in\obj\cB};$ these define a gapped $A_N$ category.
\end{dfn}

We define the gapped $A_N$ functors, $N\in\N\sqcup\{\iy\}.$
\begin{dfn}
Let $\cA,\cB$ be two filtered $A_\iy$ categories.
By a {\it filtered $A_\iy$ functor} from $\cA$ to $\cB$ we mean a function $f:\cA\to\cB$ between the object sets and a family
\e (f^k=f^k_{X_0\dots X_k}\in\hom^0(\cA_{X_0X_1}[1],\dots,\cA_{X_{k-1}X_k}[1];\cA_{fX_0fX_k}[1])),\e
$k\in\N$ and $X_0,\dots,X_k\in\obj\cA,$ such that for every $X\in\obj\cA$ the element $f^0_X\in\cA_{fXfX}^0$ has positive valuation, that is,  $\v f^0_X>0;$ and for $X_0,\dots,X_k\in\obj\cA$ with $k\in\N$ we have
\e\l{nf}\begin{split}
\sum_{p,q}(-1)^\#f^{k+p-q+1}(x_1,\dots,x_p,\fm^{q-p}(x_{p+1},\dots,x_q),x_{q+1},\dots,x_k)\\
=\sum_{a_1,\dots,a_l}\fn^l(f^{a_1}(x_1,\dots,x_{a_1}),\dots,f^{a_l-a_{l-1}}(x_{a_{l-1}+1},\dots,x_{a_l})).
\end{split}\e
where $x_1,\dots,x_k,\#,p,q$ are as in \eq{A_iy} and $a_1,\dots,a_l$ are any integers with $0\le a_1\le\dots\le a_l=k.$
The right-hand side of \eq{nf} converges because the $f^0$ terms have positive valuation.

Suppose now that $(\cA,\fm),(\cB,\fn)$ are gapped $A_N$ categories, $N\in\N\sqcup\{\iy\}.$
Note that any two discrete sub-monoids of $[0,\iy)$ are contained in some single discrete sub-monoid of $[0,\iy).$
By a {\it gapped $A_N$ functor} from $\cA$ to $\cB$ we mean a function $f:\obj\cA\to\obj\cB$ between the object sets; a discrete sub-monoid $\Ga\subset[0,\iy)$ with respect to which $\cA,\cB$ are gapped; and a family
\e(f^k_\ga=f^{k\ga}_{X_0\dots X_k}) \in\hom^0(\cA_{X_0X_1}[1],\dots,\cA_{X_{k-1}X_k}[1];\cB_{fX_0fX_k}[1]))\e
$(k,\ga)\in(\N\t\Ga)\-\{(0,0)\}$ with $[k,\ga]\le N$ and $X_0,\dots,X_k\in\obj\cA,$ such that
\e\l{A_N funct}
\begin{split}
\sum_{p,q;\al,\be}(-1)^\#f_\al^{k+p-q+1}(x_1,\dots,x_p,\fm_\be^{q-p}(x_{p+1},\dots,x_q),x_{q+1},\dots,x_k)\\
=\sum_{l;a_0,\dots,a_p;\al_0,\dots,\al_p}\fn_{\al_0}^l(f_{\al_1}^{a_1}(x_1,\dots,x_{a_1}),\dots,f_{\al_l}^{a_l-a_{l-1}}(x_{a_{l-1}+1},\dots,x_{a_l})
\end{split}\e
where $x_1,\dots, x_k,\#,p,q,\al,\be$ are as in \eq{gapped A_N} and $(l;a_0,\dots,a_p;\al_0,\dots,\al_p)$ any decomposition of $(k,\ga).$
The operators $\fn_{\al_0}^l, f_{\al_1}^{a_1},f_{\al_l}^{a_l-a_{l-1}}$ are well-defined because $[l,\al_0],[a_1,\al_1],\dots,[a_l-a_{l-1},\al_l]\le N$ by \eq{kga}.

When $N=\iy$ the gapped $A_\iy$ functor $f:(\cA,\fm)\to(\cB,\fn)$ defines a filtered $A_\iy$ functor $(f^k_{X_0\dots X_k})^{k\in\N}_{X_0,\dots,X_k\in\obj\cA}$ between the corresponding filtered $A_\iy$ categories, with $f^k_{X_0\dots X_k}:=\sum T^\ga f^{k\ga}_{X_0\dots X_k}$ where $\ga\in\Ga\-\{0\}$ for $k=0$ and $\ga\in\Ga$ for $k>0.$ The sum is well-defined because $\hom^0(\cA_{X_0X_1}[1],\dots,\cA_{X_{k-1}X_k}[1];\cB_{fX_0fX_k}[1])$ is complete. 
\end{dfn}

We define the identity functors of gapped $A_N$ categories, $N\in\N\sqcup\{\iy\}.$
\begin{dfn}
Let $\cA$ be a filtered $A_\iy$ category; we leave out $\fm$ for short. 
The {\it identity} functor $\cA\to\cA$ is defined by taking the function $\obj\cA\to\obj\cA$ to be the identity; $f^0_X=0\in \cA^0_{XX}$ for $X\in\obj\cA;$ $f^1_{XY}:\cA_{XY}\to\cA_{XY}$ to be the identity for $X,Y\in\obj \cA;$ and $f^k_{X_0\dots X_k}=0$ for $X_0,\dots,X_k\in\obj\cA$ with $k\ge 2.$

Suppose now that $\cA$ is a gapped $A_N$ category, $N\in\N\sqcup\{\iy\}.$ Then the identity functor $\cA\to\cA$ is defined by taking the function $\obj\cA\to\obj\cA$ to be the identity; $f^{10}_{XY}:\cA_{XY}\to\cA_{XY}$ to be the identity for $X,Y\in\obj \cA;$ and $f^{k\ga}_{X_0\dots X_k}=0$ for $X_0,\dots,X_k\in\obj\cA$ with $(k,\ga)\in(\N\times\Ga)\-\{(0,0),(1,0)\}.$
\end{dfn}

We define the composites of gapped $A_N$ functors, $N\in\N\sqcup\{\iy\}.$
\begin{dfn}
Let $\cA,\cB,\cC$ be filtered $A_\iy$ categories and $f:\cA\to\cB,$ $g:\cB\to\cC$ filtered $A_\iy$ functors. Then the {\it composite} $A_\iy$ functor $g\cm f:\cA\to\cC$ means the ordinary composite $\obj\cA\to \obj\cC$ and the family $(g\cm f)^k_{X_0\dots X_k},$  $X_0,\dots,X_k\in\obj\cA$ with $k\in\N,$ defined by
\[
(g\cm f)^k(x_1,\dots,x_k):=\sum_{a_1,\dots,a_l}(-1)^\#g^l(f^{a_1}(x_1,\dots,x_{a_1}),\dots,f^{a_l-a_{l-1}}(x_{a_l+1},\dots,x_{a_l}))
\]
where $x_1,\dots,x_k,\#,a_1,\dots,a_l$ are as in \eq{nf}.

Suppose now that $\cA,\cB,\cC$ are filtered $A_\iy$ categories, $N\in\N\sqcup\{\iy\};$ and $f:\cA\to\cB,$ $g:\cB\to\cC$ filtered $A_\iy$ functors.
Then the {\it composite} $A_N$ functor $g\cm f:\cA\to\cC$ means the ordinary composite $\obj\cA\to \obj\cC;$ a discrete sub-monoid $\Ga\subset[0,\iy)$ with respect to which $f,g$ are gapped; and the family $(g\cm f)^{k\ga}_{X_0\dots X_k},$ $(k,\ga)\in(\N\times\Ga)\-\{(0,0)\}$ with $[k,\ga]\le N$ and $X_0,\dots,X_k\in\obj\cA,$ defined by
\e\l{gf_k}
(g\cm f)^k_\ga(x_1,\dots,x_k):=\!\!\!\!\!\!\!\!\!\!\!\!\!\!\!\sum_{\,\,\,\,\,\,\,\,\,\,\,\,l;a_1,\dots,a_p;\al_1,\dots,\al_p\!\!\!\!\!\!\!\!\!\!\!\!\!\!\!}\!\!\!\!\!\!\!\!\!\!\!\!\!\!\!(-1)^\#g^l_{\al_0}(f_{\al_1}^{a_1}(x_1,\dots,x_{a_1}),\dots,f_{\al_p}^{a_p-a_{p-1}}(x_{a_p+1},\dots,x_{a_p}))
\e
where $x_1,\dots, x_k,\#,l,a_1,\dots,a_p,\al_1,\dots,\al_p$ are as in \eq{A_N funct}. Whether or not \eq{gf_k} holds is independent of the choice of $\Ga$ with respect to which $f,g$ are gapped; for we can first make the smallest choice of $\Ga$ and then show that $(g\cm f)^{k\ga}_{X_0\dots X_k}=0$ for $\ga$ not in this smallest $\Ga.$
\end{dfn}

We define the homotopies of gapped $A_N$ functors, $N\in\N\sqcup\{\iy\}.$
We make a rather strong hypothesis on the functions between object sets. 
\begin{dfn}\l{dfn: homotopies}
Let $(\cA,\fm),(\cB,\fn)$ be filtered $A_\iy$ categories and $f,g:(\cA,\fm)\to(\cB,\fn)$ filtered $A_\iy$ functors.
Suppose that $f,g$ induce the same function $\obj\cA\to\obj\cB$ between the object sets, which we denote by $f$ for short. By a {\it homotopy} $H:f\Rightarrow g$ we mean a family
\e (H^k=H^k_{X_0,\dots,X_k}\in \hom^{-1}(\cA_{X_0X_1}[1],\dots,\cA_{X_{k-1}X_k}[1];\cA_{fX_0fX_k}[1])),\e
$k\in\N$ and $X_0,\dots,X_k\in\obj\cA,$ such that {\bf(i)} for every $X\in\obj\cA$ the element $H^0_X\in\cA^{-1}_{fXfX}$ has positive valuation, that is, $\v H^0_X>0;$ and {\bf(ii)} for $X_0,\dots,X_k\in\obj\cA$ with $k\in\N$ we have
\begin{align*}
\sum_{p,q}(-1)^\#H^{k+p-q+1}(x_1,\dots,x_p,\fm^{q-p}(x_{p+1},\dots,x_q),x_{q+1},\dots,x_k)\\
=\sum_{a_1,\dots,a_i;b_0,\dots,b_j}\fn^{i+j+1}(f^{a_1}(x_1,\dots,x_{a_1}),\dots,f^{a_i-a_{i-1}}(x_{a_{i-1}+1},\dots,x_{a_i}),\\
H^{b_0-a_i}(x_{a_i+1},\dots,x_{b_0}),g^{b_1-b_0}(x_{b_0+1},\dots,x_{b_1}),\dots,g^{b_j-b_{j-1}}(x_{b_{j-1}+1},\dots,x_{b_j}))
\end{align*}
where $x_1,\dots,x_k,\#,p,q$ are as in \eq{nf}; and $a_1,\dots,a_i,b_0,\dots,b_j$ any integers with $0\le a_1\le\dots\le a_i\le b_0\le b_1\le\dots\le b_j=k.$ The right-hand side converges by (i) above.

Suppose now that $(\cA,\fm),(\cB,\fn)$ are gapped $A_N$ categories, $N\in\N\sqcup\{\iy\};$ and $f,g:(\cA,\fm)\to(\cB,\fn)$ gapped $A_N$ functors with the same function $f:\obj\cA\to\obj\cB$ between the object sets. 
By a {\it homotopy} $H:f\Rightarrow g$ we mean a discrete sub-monoid $\Ga\subset[0,\iy)$ with respect to which $f,g$ are both gapped and a family
\e (H^k_\ga=H^{k\ga}_{X_0,\dots,X_k}\in \hom^{-1}(\cA_{X_0X_1}[1],\dots,\cA_{X_{k-1}X_k}[1];\cA_{fX_0fX_k}[1])),\e
$(k,\ga)\in(\N\t\Ga)\-\{(0,0)\}$ with $[k,\ga]\le N$ and $X_0,\dots,X_k\in\obj\cA,$ such that
\begin{align*}
&\sum_{p,q,\al,\be}(-1)^\# H_\al^{k+p-q+1}(x_1,\dots,x_p,\fm_\be^{q-p}(x_{p+1},\dots,x_q),x_{q+1},\dots,x_k)\\
&=\sum_{i,j;a_0,\dots,a_i,b_0,\dots,b_j;\al_0,\dots\al_i,\be_0,\dots,\be_j}
\!\!\!\!\!\!\!\!\!\!\!\!\!\!\!\!\!\!\!\!\!\!\!\!\!\!\!\!\!\!
\fn_{\al_0}^{i+j+1}(f_{\al_1}^{a_1}(x_1,\dots,x_{a_1}),\dots,f_{\al_i}^{a_i-a_{i-1}}(x_{a_{i-1}+1},\dots,x_{a_i}),\\
&H_{\be_0}^{b_0-a_i}(x_{a_i+1},\dots,x_{b_0}),g_{\be_1}^{b_1-b_0}(x_{b_0+1},\dots,x_{b_1}),\dots,g_{\be_j}^{b_j-b_{j-1}}(x_{b_{j-1}+1},\dots,x_{b_j}))
\end{align*}
where $x_1,\dots,x_k,\#,p,q,\al,\be$ are as in \eq{A_N funct} and
$i+j+1,$ $(a_0,\dots,a_i,b_0,\dots,b_j),$ $(\al_0,\dots\al_i,\be_0,\dots,\be_j)$ any decomposition of $(k,\ga).$
The operators 
$\fn_{\al_0}^{i+j+1},$ $f_{\al_1}^{a_1},$ $\dots,$ $f_{\al_i}^{a_i-a_{i-1}},$ $H_{\be_0}^{b_0-a_i},$ $g_{\be_1}^{b_1-b_0},$ $\dots,$ $g_{\be_j}^{b_j-b_{j-1}}$ are well-defined because $[i+j+1,\al_0],$ $[a_1,\al_1],\dots,[a_i,\al_i],$ $[b_0-a_i,\be_0],$ $[b_1-b_0,\be_1],\dots,[b_j-b_{j-1},\be_j]$ are all $\le N$ by \eq{kga}.

When $N=\iy$ the gapped $A_\iy$ homotopy $H:f\Rightarrow g$ defines a filtered $A_\iy$ homotopy $(H^k_{X_0\dots X_k})^{k\in\N}_{X_0,\dots,X_k\in\obj\cA}$ between the corresponding filtered $A_\iy$ functors,
with $H^k_{X_0\dots X_k}:=\sum T^\ga H^{k\ga}_{X_0\dots X_k}$ where $\ga\in\Ga\-\{0\}$ for $k=0$ and $\ga\in\Ga$ for $k>0.$ The sum is well-defined because $\hom^{-1}(\cA_{X_0X_1}[1],\dots,\cA_{X_{k-1}X_k}[1];\cA_{fX_0fX_k}[1])$ is complete.
\end{dfn}
\begin{rem}\l{rem: models}
We can make an equivalent definition by using {\it models} of $\cB\times[0,1]$ as Fukaya, Oh, Ohta and Ono do \cite[\S4.2]{FOOO}. This is the result of obstruction theory arguments \cite[\S4.4]{FOOO} which extend readily to our context; the major difference is that mentioned in Remark \ref{rem: difference from FOOO}.
\end{rem}

We define the notion of being homotopic for gapped $A_N$ functors, $N\in\N\sqcup\{\iy\}.$
\begin{dfn}
Let $(\cA,\fm),(\cB,\fn)$ be gapped $A_N$ categories and $f,g:(\cA,\fm)\to(\cB,\fn)$ gapped $A_N$ functors which induce the same function between the object sets.
We say that $f,g$ are {\it homotopic} if there is a homotopy from $f$ to $g.$
\end{dfn}
The following is a straightforward extension of Fukaya, Oh, Ohta and Ono's result for gapped $A_N$ algebras.
\begin{lem}[cf.\! {\cite[Proposition 4.2.37]{FOOO}}]
Let $(\cA,\fm),(\cB,\fn)$ be gapped $A_N$ categories and fix a function $\obj\cA\to\obj\cB.$ 
Then being homotopic is an equivalence relation for those gapped $A_N$ functors from $(\cA,\fm)$ to $(\cB,\fn)$ inducing the given function $\obj\cA\to\obj\cB.$ \qed
\end{lem}

We define the homotopy equivalences of gapped $A_N$ categories, $N\in\N\sqcup\{\iy\}.$
\begin{dfn}
Let $(\cA,\fm),(\cB,\fn)$ be gapped $A_N$ categories of the same object set and $f:(\cA,\fm)\to(\cB,\fn)$ a gapped $A_N$ functor which induces the identity on the object set.
We call $f$ a {\it homotopy equivalence} if there exists a gapped $A_N$ functor $g:(\cB,\fn)\to(\cA,\fm)$ which induces the identity on the object set and such that $g\cm f,f\cm g$ are homotopic to the identity functors of $(\cA,\fm),(\cB,\fn)$ respectively.
We say that $(\cA,\fm),(\cB,\fn)$ are {\it homotopy equivalent} if there exists a homotopy equivalence $(\cA,\fm)\to(\cB,\fn).$
\end{dfn}
The following is also a straightforward extension of Fukaya, Oh, Ohta and Ono's result for gapped $A_N$ algebras.
\begin{lem}[cf.\! {\cite[Corollary 4.2.44 and Remark 7.2.71]{FOOO}}]
\l{thm: A_N Whitehead}
Being homotopy equivalent is an equivalence relation for gapped $A_N$ categories of the same object set.

Let $(\cA,\fm),(\cB,\fn)$ be gapped $A_N$ categories of the same object set and $f:(\cA,\fm)\to(\cB,\fn)$ a gapped $A_N$ functor which induces the identity on the object set.
Then $f$ is a homotopy equivalence if and only if for any $X,Y\in\obj \cA$ the co-chain map $f^{10}_{XY}:(\cA_{XY},\fm^{10}_{XY})\to (\cB_{fXfY},\fn^{10}_{XY})$ induces an isomorphism of the cohomology groups. \qed
\end{lem}

Here is the obstruction theory result for gapped $A_N$ categories which extends that of Fukaya, Oh, Ohta and Ono \cite[Theorem 7.2.72]{FOOO}.
\begin{thm}\l{thm: ob1}
Let $(\cB,\fn)$ be a gapped $A_{N+1}$ category, $N\in\N,$ and $(\cB',\fn)$ a full subcategory of the $A_N$ truncation of $(\cB,\fn).$
Let $(\cA',\fm)$ be a gapped $A_N$ category with $\obj\cA'=\obj\cB',$ and $g:(\cA',\fm)\to(\cB',\fn)$ a homotopy equivalence which induces the identity on the object set.
Then there exist a gapped $A_{N+1}$ category $(\cA,\fm)$ with $\obj\cA=\obj\cB,$ which contains $(\cA',\fm)$ as a full subcategory in its $A_N$ truncation; and a gapped $A_{N+1}$ functor $f:(\cA,\fm)\to(\cB,\fn)$ which induces the identity on the object set, whose $A_N$ truncation agrees with $g$ on $\cA'.$
\end{thm}
\begin{proof}
Recall from definition that the $A_0$ truncation of $(\cA',\fm)$ is a family of co-chain complexes.
We assign to $X,Y\in\obj\cA$ a co-chain complex $(\cA_{XY},\fm^{10}_{XY})$ and a co-chain map $f^{10}_{XY}:(\cA_{XY},\fm^{10}_{XY})\to(\cB_{XY},\fn^{10}_{XY}).$
For $X,Y\in\obj\cA'$ set $(\cA_{XY},\fm^{10}_{XY}):=(\cA_{XY}',\fm^{10}_{XY})$ and $f^{10}_{XY}:=g^{10}_{XY}.$
For $X,Y\in\obj\cA$ with $X\notin\obj\cA'$ or $Y\notin\obj\cA'$ set $(\cA_{XY},\fm^{10}_{XY}):=(\cB_{XY},\fn^{10}_{XY})$ and define $f^{10}_{XY}:\cA_{XY}\to\cB_{XY}$ to be the identity map.
In both cases the co-chain map $f^{10}_{XY}:(\cA_{XY},\fm^{10}_{XY})\to (\cB_{XY},\fn^{10}_{XY})$ induces an isomorphism of the cohomology groups.

We extend these co-chain complexes to a gapped $A_N$ category $(\cA,\fm)$ and these co-chain maps to a gapped $A_N$ functor $f^N:(\cA,\fm)\to(\cB,\fn)^N$ where $(\cB,\fn)^N$ is the $A_M$ truncation of $(\cB,\fn).$
We do this by an induction on $M\in\{0,\dots,N-1\}.$
Suppose therefore that we have defined $\fm^{l\de}_{X_0\dots X_l}$ for $(l,\de)\in(\N\times\Ga)\-\{(0,0)\}$ with $[l,\de]\le M$ and for $X_0\dots X_l\in\obj\cA.$
We extend these to a gapped $A_{M+1}$ category $(\cA,\fm)$ and a gapped $A_{M+1}$ functor $(\cA,\fm)\to(\cB,\fn)^{M+1}.$

Denote by $(k,\ga)\in(\N\times\Ga)\-\{(0,0)\}$ the element with $[k,\ga]=M+1.$
We assign to $X_0,\dots,X_k\in\obj\cA$ some filtered homomorphisms
$\fm^{k\ga}_{X_0\dots X_k},f^{k\ga}_{X_0\dots X_k}$ which satisfy the $A_{M+1}$ equations.
For $X_0,\dots,X_k\in\obj\cA'$ we use the given $\fm^{k\ga}_{X_0\dots X_k},f^{k\ga}_{X_0\dots X_k};$
and otherwise we introduce some new $\fm^{k\ga}_{X_0\dots X_k},f^{k\ga}_{X_0\dots X_k}.$
Define a degree-one $\La$-linear map $\d$ from $\hom^*(\cA_{X_0X_1}[1],\dots,\cA_{X_{k-1}X_k}[1];\cA_{X_0X_k}[1])$ to itself by 
\e\d\ph:=\fm^1_0\cm\ph-(-1)^{\deg\ph}\ph\cm(\text{the co-derivation of }\fm^1_0).\e
This is well-defined because $\v\ph\ge0$ implies $\v(\d\ph)\ge0.$ And $\d\cm\d=0.$
Define an {\it obstruction co-cycle} $({\rm o}\fm)^{k\ga}_{X_0\dots X_k}\in\hom^2(\cA_{X_0X_1}[1],\dots,\cA_{X_{k-1}X_k}[1];\cA_{X_0X_k}[1])$ by
\[
(x_1,\dots,x_k)\mapsto\sum_{p,q}(-1)^\#\fm^{k+p-q+1}_\al(x_1,\dots,x_p,\fm^{q-p}_\be(x_{p+1},\dots,x_q),x_{q+1},\dots,x_k)
\]
where $[k+p-q+1,\al]\le M$ or $[q-p,\be]\le M.$ Doing the same computation as Fukaya and others \cite[(4.5.2.1)]{FOOO} we find $\d({\rm o}\fm)^{k\ga}_{X_0\dots X_k}=0;$ that is, $({\rm o}\fm)^{k\ga}_{X_0\dots X_k}$ is certainly a co-cycle.
Their computation \cite[(4.5.2.4)]{FOOO} shows also that the $A_M$ homotopy equivalence $f:(\cA,\fm)\to(\cB,\fn)$ induces an isomorphism from the $\d$-cohomology group
$H^2(\hom(\cA_{X_0X_1}[1],\dots,\cA_{X_{k-1}X_k}[1];\cA_{X_0X_k}[1]))$ to the $\d$-cohomology group $H^2(\hom(\cB_{X_0X_1}[1],\dots,\cB_{X_{k-1}X_k}[1];\cB_{X_0X_k}[1])),$
mapping the cohomology class of $({\rm o}\fm)^{k\ga}_{X_0\dots X_k}$ to that of $({\rm o}\fn)^{k\ga}_{X_0\dots X_k}.$
But the latter vanishes for the obvious reason: $(\cB,\fn)$ is already an $A_{M+1}$ category.
Thus $({\rm o}\fm)^{k\ga}_{X_0\dots X_k}$ is a co-boundary, written as $\d\mu^{k\ga}_{X_0\dots X_k}$ for some
\e\mu^{k\ga}_{X_0\dots X_k}\in \hom^1(\cA_{X_0X_1}[1],\dots,\cA_{X_{k-1}X_k}[1];\cA_{X_0X_k}[1]).\e
Define then an obstruction co-cycle in $\hom^1(\cA_{X_0X_1}[1],\dots,\cA_{X_{k-1}X_k}[1];\cB_{X_0X_k}[1])$ by sending $(x_1,\dots,x_k)$ to
\[\begin{split}
\left[\!\!\!\!\!\!\!\!\!\!\!!\!\!\sum_{\,\,\,\,\,\,\,\,\,\,\,\,\,l;a_1,\dots,a_l;\al_1,\dots,\al_l\!\!\!\!\!\!\!\!\!\!\!\!\!\!}\!\!\!\!\!\!\!\!\!\!\!\!\fm^l_{\al_0}(f^{a_1}_{\al_1}(x_1,\dots,x_{a_1}),\dots,f^{a_l-a_{l-1}}_{\al_l}(x_{a_{l-1}},\dots,x_{a_l}))\right]+f^1_0\cm\mu^k_\ga(x_1,\dots,x_k)\\
+\sum_{p,q;\al,\be}(-1)^\# f^{k+p-q+1}_\al(x_1,\dots,x_p,\fm^{q-p}_\be(x_{p+1},\dots,x_q),x_{q+1},\dots,x_k)
\end{split}\]
where $[k+p-q+1,\al]\le M$ or $[q-p,\be]\le M;$ and $l,$ $a_1,\dots,a_l,$ $\al_1,\dots,\al_l$ are a decomposition of $(k,\ga)$ with $(l,\al_0)>(1,0)$ so that $[a_1,\al_1],\dots[a_l-a_{l-1},\al_l]\le M.$
As $f^1_0$ is a co-chain equivalence we can modify $\mu^{k\ga}_{X_0\dots X_k}$ if we need, so that the corresponding obstruction co-cycle will vanish.
So there is a gapped $A_{M+1}$ functor $f:(\cA,\fm)\to(\cB,\fn)^{M+1}$ as we want, completing the induction.

We have thus obtained a gapped $A_N$ category $(\cA,\fm)$ and a gapped $A_N$ functor $f:(\cA,\fm)\to(\cB,\fn)^N;$ that is, $\fm^{k\ga}_{X_0\dots X_k},$ $f^{k\ga}_{X_0\dots X_k}$ are given for {\it every} $(k,\ga)\in(\N\times\Ga)'\-\{(0,0)\}$ with $[k,\ga]\le N.$ Extend these to $(k,\ga)$ with $[k,\ga]=N+1,$ by repeating the obstruction theory arguments above; then the proof is complete.
\end{proof}
\begin{rem}
The first step above has no counterpart in their original proof \cite[Theorem 7.2.72]{FOOO}.
The last step is the closest to it; the difference is only that mentioned in Remark \ref{rem: difference from FOOO}.
The induction steps are similar but we must be more careful: we cannot change the given $\fm^{k\ga}_{X_0\dots X_k},$ $f^{k\ga}_{X_0\dots X_k}.$
\end{rem}

\subsection{Fukaya Categories}\l{sec: Fuk} 
We begin by recalling the definition of symplectic manifolds with contact boundary. As in Definition \ref{dfn: br} the Weinstein neighbourhoods are the typical example.
\begin{dfn}  \label{dfn:Liouv}
A compact symplectic manifold $(X,\om)$ with contact boundary is a symplectic manifold which is given a $1$-form $\la$ near $\d X\subset X$ such that $d\la=\om$ and the vector field $v_\la$ given by $v_\la\lrcorner \,\om=\la$ is outwards pointing on $\d X$. We call $\la$ a {\it Liouville $1$-form} of $(X,\om).$
 
In these circumstances there exists near $\d X\sb X$ a unique smooth function $r_\la$ with values in $(0,1)$ outside $\d X,$ with values $1$ on $\d X,$ and such that $dr_\la(v_\la) =r_\la.$ We call $r_\la$ the {\it radius} function near $\d X.$ We say that an almost complex structures $J$ is compatible with the contact boundary if it is compatible with $\om$ in the usual sense, if $\la=Jd r_\la$ near $\d X$ and if $J$ is invariant near $\d X$ under the flow of $v_\la.$

If there exists a {\it global} $1$-form $\la$ with $d\la=\om $ and satisfying the conditions above then we call $(X,\la)$ a {\it Liouville domain}. By a {\it Liouville subdomain} of a Liouville domain $(X,\la)$ we mean a compact co-dimension $0$ submanifold $Y\subset X\-\partial X$ with smooth boundary at every point of which the vector field $v_\la$ points outwards. Then $(Y,\la|_Y)$ is itself a Liouville domain. Its radius function is well defined in $Y$ near $\partial Y$ and in fact extends to a neighbourhood in $X$ of $\partial Y.$ 
\end{dfn}

We consider the following class of Lagrangians:
\begin{dfn}\l{dfn: Lag}
Let $(X,\om)$ be a compact symplectic manifold with contact boundary, and $\la$ a Liouville $1$-form of it. Let $\hat L$ be a compact manifold possibly with boundary and $\io:\hat{L}\to X$ a Lagrangian immersion. We say that $L$ is {\it generic} if its every self-intersection point is a transverse double point outside $\d X.$ We say that $L$ has {\it Legendrian} collar if $\io(\d \hat{L})\sb \d X,$ $\io^*\la=0$ near $\d \hat{L}\sb \hat{L}$ and $\io^*r_\la$ has regular value $1$ (this is automatically true if $\hat{L}$ is closed).

Suppose now that $L_0$  and $L_1$ are two Lagrangians in $X$ given by immersions $\hat L_0\to X$ and $\hat L_1\to X.$
We say that $L_0,L_1$ are a {\it generic pair} if the immersion $\hat L_0\sqcup \hat L_1$ is generic. 
\end{dfn}

We recall now the definition of $\Z$-graded symplectic manifolds which uses almost complex structures on the ambient symplectic manifold; for other equivalent definitions see for instance Seidel \cite{Seid1}.
\begin{dfn}\l{dfn: Z-grading}
Let $X$ be a symplectic manifold of dimension $2n.$ Then a {\it $\Z$-grading} of $X$ is the homotopy class of a pair $(J,\Om^{\otimes2})$ where $J$ is a compatible almost complex structure and $\Om^{\otimes2}$ a nowhere-vanishing $C^\iy$ section of the complex line bundle $(\La^{n0}X)^{\otimes 2}$ where $\La^{n0}X$ is the complex line bundle over $X$ of $(n,0)$ forms relative to $J.$ We say that $X$ is {\it $\Z$-graded} if it is given a $\Z$-grading. 
\end{dfn}
\begin{rem}
A $\Z$-grading of $(X,\om)$ exists if and only if $2c_1(X)=0\in H^2(X,\Z)$ 
where $c_1(X)$ denotes the first Chern class of $X.$
\end{rem}

We recall next the definition of $\Z$-graded Lagrangians, which is a generalization of special Lagrangians in the sense of Definition \ref{dfn: sLag}.
\begin{dfn}\l{dfn:Lag Z-grading}
Let $X$ be a $\Z$-graded symplectic manifold of dimension $2n,$ and let $(J,\Om^{\otimes2})$ represent the $\Z$-grading. It is known then that for each Lagrangian immersion $\io:\hat L\to X$ there exists a unique $C^\iy$ function $\ph:\hat L\to \R/2\pi\Z$ such that $\io^*\Om^{\ot2}=e^{i\ph}{\rm vol}^2$ where ${\rm vol}^2$ is some nowhere-vanishing section of the real line bundle $(\La^n\hat L)^{\otimes2}$ where $\La^n\hat L$ is the real line bundle over $L$ of $n$-forms. We call $\ph:\hat L\to \R/2\pi\Z$ the {\it phase function} of $\io:\hat L\to X.$

By a {\it $\Z$-grading} of $\io:\hat L\to X$ we mean a lift of $\ph:\hat L\to \R/2\pi\Z$ to the universal cover $\R\to\R/2\pi\Z\cong S^1.$ By a {\it $\Z$-graded Lagrangian} we mean the pair of $\io:\hat L\to X$ and a $\Z$-grading of it. We say that $\io:\hat L\to X$ is {\it special of phase $\ph:\hat L\to \R/2\pi\Z$} with respect to $\Om^{\ot2}$ if $\ph:\hat L\to \R/2\pi\Z$ is locally constant. In particular, a connected $\Z$-graded special Lagrangian is equivalent to the pair of $\io:\hat L\to X$ and a constant phase $\ph\in\R.$ Define an operator $[1]$ by sending the $\Z$-graded Lagrangian $(\io,\ph)$ to $(\io,\ph-2\pi)$ where $\ph$ is the lifted phase function $\hat L\to\R$ so that we can do define $\ph-2\pi:\hat L\to\R.$ This will induce the shift operator of the triangulated $A_\iy$ category $\Tw\cF(X)$ once $\cF(X)$ has been defined. 
\end{dfn}
\begin{rem}\l{rem: Lag grading}
The Lagrangian immersion admits a $\Z$-grading if and only if the group homomorphism $\ph_*:H_1(L,\Z)\to H_1(S^1,\Z)$ vanishes.
\end{rem}

We define next the Maslov indices. 
\begin{dfn}
Let $X$ be a $\Z$-graded symplectic manifold of dimension $2n,$ and let $(J,\Om^{\otimes2})$ represent the $\Z$-grading. Let $\io:\hat{L}\to L\subset X$ be a generic Lagrangian immersion. We consider the fibre product
\e L\t_XL:=\{(x,y)\in \hat{L}\t \hat{L}:\io(x)=\io(y)\}\e
which consists of the diagonal
$\{(x,x)\in \hat{L}\t \hat{L}\},$ identified with $\hat{L},$ and the self-intersection pairs in $\hat{L}.$
Suppose that $L$ is $\Z$-graded by a phase function $\ph:\hat L\to\R$ and we define then a {\it Maslov index} function
$\mu_L:L\t_XL\to \Z.$
Set $\mu_L\=0$ on the diagonal $\hat L.$
We associate to each pair $(x_1,x_2)\in L\t_XL$ an integer called the {\it Maslov index}. Put $x:=\io(x_1)=\io(x_2)\in X$ and take a $\C$-linear isomorphism $(T_xX,J|_x)\cong\C^n$ which maps $\om$ to
$\frac{i}{2}\sum_{j=1}^ndz_j\w d\bar{z}_j,$
$T_{x_2}L$ to $\R^n\sb\C^n$
and $T_{x_1}L$ to
$\{(e^{i\th_1}t_1,\cdots,e^{i\th_n}t_n)\in\C^n:t_1,\cdots,t_n\in\R\}$
for some $\th_1,\cdots,\th_n\in(0,\pi).$
These $\th_1,\cdots,\th_n$ are unique up to order so we can define
\e
\l{mu0}\mu_L(x_1,x_2):=\frac{1}{2\pi}[2\th_1+\cdots+2\th_n+\ph(x_2)-\ph(x_1)].
\e
The definition of phase functions implies readily that $\mu_L(x_1,x_2)$ is in fact an integer. 

Suppose now that $L_1,L_2\sb X$ are a generic pair of Lagrangians, $\Z$-graded by $\ph_1:L_1\to\R$ and $\ph_2:L_2\to\R$ respectively.
The fibre product $L_1\t_XL_2$ is then a finite set and we define the Maslov index function
$\mu_{L_1L_2}:L_1\t_XL_2\to \Z$ by sending $(x_1,x_2)$ to the right-hand side of \eq{mu0} above with $\ph_1(x_2),\ph_2(x_2)$ in place of $\ph(x_1),\ph(x_2)$ respectively.  
\end{dfn}
\begin{rem}
Since $\th_1,\cdots,\th_n\in(0,\pi)$ it follows by \eq{mu0} that
\e
\l{mu} \mu_L(x_1,x_2)-n<\frac{\ph(x_2)-\ph(x_1)}{2\pi}<\mu_L(x_1,x_2).
\e
In particular, if $\ph(x_2)-\ph(x_1)\in(-1,1)$ then the integer $\mu_L(x_1,x_2)$ must lie in the closed interval $[0,n];$ and if $\ph(x_1)=\ph(x_2)$ then $\mu_L(x_1,x_2)\in[1,n-1].$
\end{rem}
\begin{exam}
Take $X=\C^n$ with co-ordinates $z_1,\dots,z_n,$ $\Om=dz_1\wedge\dots\wedge dz_n$ and $L_2=\R^n\subset\C^n.$
Define a $\Z$-grading $\ph_2:L_2\to\R$ by $\ph_2\=0.$
Take $\th_1,\dots,\th_m\in (\frac\pi2,\pi)$ and $\th_{m+1},\dots,\th_n\in (0,\frac\pi2).$
For $s\in[0,1]$ define
\[L_s:=\{(e^{is(\th_1-\pi)}t_1,\dots,e^{is(\th_m-\pi)}t_m,e^{is\th_{m+1}}t_{m+1},\dots,e^{is\th_n}t_n):t_1,\dots,t_n\in\R\}\]
so $L_0=L_2.$
Define a $\Z$-grading $\ph_s:L_s\to\R$ by starting with $\ph_2\=0$ on $L_2$ and lifting the path $[0,s].$ Then   
$\ph_s:=2s[(\th_1-\pi)+\dots+(\th_m-\pi)+\th_{m+1}+\dots+\th_n]$ and in particular
$\ph_1=2(\th_1+\dots+\th_n)-2m\pi.$ The Maslov index of $(L_1,L_2)$ at the intersection point $0\in\C^n$ is now equal to 
$\frac{1}{2\pi}[2\th_1+\cdots+2\th_n+\ph_2-\ph_1]=m.$
This agrees with the Morse index of the quadratic function $(\tan \th_1)t_1^2+\dots+(\tan\th_n) t_n^2$ on $L_2.$ Our convention is thus the same as Seidel's \cite[\S2d, (v)]{Seid1}.
\end{exam}

We extend the definition of relative spin structures to countably many Lagrangians. These will equip the pseudo-holomorphic curve moduli spaces with natural orientations.
\begin{dfn}
Let $X$ be a symplectic manifold, $\hat L$ a compact manifold and $\io:\hat L\to X$ a generic Lagrangian immersion; we allow $\hat L,X$ to have boundary. By a {\it relative spin structure} of $\io:\hat L\to X$ we mean the data $(\tau,V,o,\si,s)$ where $\ta$ is a triangulation of $X,$ $V$ an oriented vector bundle over the $3$-skeleton $[X,\tau]^3$ of $(X,\tau),$ $o$ an orientation of $\hat L,$
$\si$ a triangulation of $\hat L$ such that $\io$ induces a simplicial map $(\hat L,\si)\to(X,\ta),$ and $s$ a spin structure of the bundle $\io^*V\op T\hat L$ restricted to the $2$-skeleton of $(\hat L,\si).$

Suppose now that $(L_N)_{N=0}^\iy$ is a sequence of compact manifolds which may have boundary, and $\io:\bigsqcup_{N=0}^\iy L_N\to X$ a generic Lagrangian immersion. By a {\it relative spin structure} of $\io: \bigsqcup_{N=0}^\iy L_N\to X$ we mean a sequence $(\ta_N,V_N,o_N,\si_N,s_N)_{N=0}^\iy$ such that each $(\ta_N,V_N,o_N,\si_N,s_N)$ is a relative spin structure of $\io|_{L_N}:L_N\to X,$ each $(X,\ta_N)$ is a subcomplex of $(X,\ta_{N+1}),$ and the restriction of each $V_{N+1}$ to $[X,\ta_N]^3$ agrees with $V_N.$
\end{dfn}
\begin{rem}
Suppose that $V_0=V_1=V_2=\dots=:V,$ that $V$ is defined over the whole $X$ and that each $s_N$ is a spin structure of $\io^*V\op TL_N$ over the whole $L_N.$
Then each Stiefel--Whitney class $w_2(TL_N)\in H^2(L_N,\Z/2\Z)$ is the pull-back of $w_2(V)\in H^2(X,\Z/2\Z).$ 
In this case we refer to $w_2(V)\in H^2(X,\Z/2\Z)$ as the {\it background class} of the relative spin structure.

We can in practice take $X$ to be embedded in the cotangent bundle $T^*Q$ over a closed oriented manifold $Q$ and take $V=\pi^*TQ,$ where $\pi:T^*Q\to Q$ is the projection. So the background class is $\pi^*w_2(TQ)$ and the zero-section $Q$ will be given a spin structure of $V|_Q\oplus TQ.$
\end{rem}

We introduce now the notion of branes:
\begin{dfn}\l{dfn: branes}
Let $(X,\om)$ be a compact symplectic manifold with contact boundary, and $\la$ a Liouville $1$-form of it.
By a {\it brane} on $X$ we mean a pair $(L,E)$ where
\iz
\item
$L\subset X$ is a compact $\Z$-graded generically-immersed Lagrangian with Legendrian collar; and
\item
$E$ is a filtered local system over the domain $\hat{L}$ of the immersion $\hat L\to L\subset X.$
\iz
We write only $E$ in place of $(L,E)$ when we want to save notation. 
\end{dfn}
We choose countably many branes and a Novikov field.
\begin{hp}\l{hp: branes} 
Let $(X,\om)$ be a compact symplectic manifold with contact boundary, and $\la$ a Liouville $1$-form of it.
Let $\cC(X)$ be a countable set of branes such that any two distinct Lagrangian immersions in $\cC(X)$ are a generic pair in the sense of Definition \ref{dfn: Lag}; we allow some of the Lagrangian immersions to be the same.
Fix a Novikov field $(\La,\v).$ When $\La$ has characteristic $\ne0$ we suppose that $(X,\om)$ has no non-constant pseudo-holomorphic sphere (cf.\ \cite{FOOO-Z}).
When $\La$ has characteristic $\ne2$ we suppose that $X$ and the underlying Lagrangians have a relative spin structure with respect to some order of the Lagrangians. 
\end{hp}

We introduce also the homologically perturbed Floer complexes:
\begin{dfn}\l{dfn: Floer complexes} 
Let $\cC(X)$ be as in Hypothesis \ref{hp: branes} and $(L_0,E_0),(L_1,E_1)\in \cC(X)$ two elements. 
Denote by $\pi_-:L_0\times_XL_1\to L_0$ and
$\pi_+:L_0\times_XL_1\to L_1$ the projections.  
The {\it homologically perturbed Floer complex} associated to the pair $(L_0,E_0),(L_1,E_1)$ is the $\Z$-graded finite-dimensional valued $\La$-vector space
\e\l{CF} CF^*(E_0,E_1):=H^*(L_0\t_XL_1,\hom(\pi_-^*E_0,\pi_+^*E_1))[-\mu_{L_0L_1}]\e
where $[-\mu_{L_0L_1}]$ denotes the degree shift by Maslov indices.
Since the fibres of $E_0,E_1$ are valued vector spaces it follows that so are $\hom(\pi_-^*E_0,\pi_+^*E_1)$ and $CF^*(E_0,E_1).$
\end{dfn}
\begin{rem}\l{rem: CF}
The definition \eq{CF} is actually the result of applying the filtered homological perturbation lemma \cite[Theorem A]{FOOO}
to the singular or de Rham model in the literature.
More precisely, in any construction of Lagrangian Floer cohomology groups we have to choose a chain model of the
Floer complex $CF^*(E_0,E_1)$ with $L_0=L_1.$
Akaho--Joyce \cite{AJ} and Fukaya, Oh, Ohta and Ono \cite{FOOO} both take the singular chain model.
Fukaya \cite{FukCyc,Fuk} takes the de Rham model assuming that the ground field $\k$ contains $\R.$
\end{rem}

We come now to the construction of the curved Fukaya category $\cC(X).$ This is not yet the Fukaya category $\cF(X)$ used in the Introduction.
\begin{thm}\l{thm: Fuk cat}
Let $\cC(X)$ be as in Hypothesis \ref{hp: branes} and put $n:=\hf\dim X.$
Then $\cC(X)$ has the structure of a gapped $A_\iy$ category whose morphism spaces are given by \eq{CF} and which has the following properties:
\iz
\item
the $A_\infty$ structure is curved in such a way that if $(L,E)\in\cC(X),$ if $L$ has no self-intersection point of index $2$ or $n-2,$
and if $L$ does not bound any non-constant pseudo-holomorphic discs, then $(L,E)$ has curvature $\fm^0=0;$ 
\item  for a pair $E_0,E_1\in\cC(X)$, the differential
$\fm^1:CF^*(E_0,E_1)\to CF^*(E_0,E_1)$ has leading term $\fm^1_0=0;$ and
  \item for a triple $E_0,E_1,E_2\in\cC(X)$ which are supported on the same Lagrangian, the leading term $\fm^2_0$ of the product map
\e\fm^2:CF^*(E_0,E_1)\t CF^*(E_1,E_2)\to CF^*(E_0,E_2)\e
vanishes with the following exceptions: on the diagonal component $\hat{L}$ of $L\t_XL$ this agrees with the usual cup product map
\[ H^*(\hat{L},\hom(E_0,E_1)) \otimes H^*(\hat{L},\hom(E_1,E_2)) \to H^*(\hat{L},\hom(E_0,E_2));\]
and on each self-intersection pair $(a,b)\in L\times_XL$ it is given by the map
\[ \hom(E_0|_a,E_1|_b)[-\mu_{ab}] \otimes \hom(E_1|_b,E_2|_a)[-\mu_{ba}] \to  H^n(\hat{L},\hom(E_0,E_2)).   \]
\iz
\end{thm}

\subsection{Proof of Theorem \ref{thm: Fuk cat}}\l{sec: Fuk2}
We begin by taking the singular chain model of \eq{CF} as in Remark \ref{rem: CF}:
\begin{dfn}\l{dfn: chain model}
Let $\hat L$ be a compact manifold with boundary, which we suppose oriented if $\La$ has characteristic $\ne2.$ Let $E$ be a filtered local system on $\hat L.$ The smooth singular chain complex $C_*(L,E)$ is the $\La^0$ module generated by pairs $(\si,\ep)$ where $\si$ is a smooth singular $i$-simplex in $\hat L,$ and $\ep$ a (flat) section of $\si^*E$ of non-negative valuation, modulo the relation
\e a(\si,\ep)=(\si,a\ep)\text{ for }a\in\La^0.\e
The differential $d$ on $C_*(\hat L,E)$ is defined by
\e d(\si,\ep)=\sum_{j=0}^i(-1)^j(\si\cm \de^j,\ep\cm \de^j)\e
where $\de^j$ is the $j^{\rm th}$ face inclusion map. 

We define $C_*(\hat L,\d \hat L;E):=C_*(\hat L,E)/C_*(\d \hat L,E)$ which we regard, by Poincar\'e duality (which makes sense because of the relative spin structure data), as a {\it cochain} complex. Recall from Hypothesis \ref{hp: branes} that we have the date $\cC(X).$ To $(L_i,E_i)\in\cC(X),$ $i=0,1,$ we assign
\e\l{E0E1} C_*(L_0\t_X L_1, \d(L_0\t_X L_1); \hom (\pi_-^*E_0,\pi_+^*E_1))\e
where $\pi_+,\pi_-$ are as in Definition \ref{dfn: Floer complexes}.
\end{dfn}

Let $J$ be a compatible almost complex structure on $X$ in the sense of Definition \ref{dfn:Liouv}. 
We introduce now the moduli spaces of $J$-holomorphic discs.
Let $L_0,\cdots,L_k\sb X$ be $k+1$ underlying Lagrangians in $\cC(X).$
\begin{dfn}\l{dfn: hol curves}
We denote by $\cM(L_1,\cdots,L_k;L_0)$ the compact moduli space of finite-area stable $J$-holomorphic discs in $X$ with $k+1$ boundary points, whose corresponding segments are ordered counterclockwise and lifted to $L_0,\cdots,L_k.$ More precisely, for each genus-zero prestable bordered Riemann surface $\Si$ we choose a continuous map $S^1\to \d\Si$ as Akaho--Joyce do \cite[Definition 4.1]{AJ}; this is essentially unique. We distinguish $L_1,\dots,L_k$ and $L_0$ because the evaluations maps, which we will define shortly in \eq{ev_i}, are different.
 
Denote by $\ze_0,\dots,\ze_{k+1}\in S^1$ the points corresponding to the marked points on $\d\Si;$ and by $\hat L_0,\dots,\hat L_k$ the domains of the immersed Lagrangians. For each holomorphic curve $u:\Si\to X$ we choose a continuous map $S^1\-\{\ze_0,\dots,\ze_{k+1}\}\to \hat L_0\sqcup\dots\sqcup \hat L_k$ as Akaho--Joyce \cite[Definition 4.2]{AJ}. We require that the $k+1$ connected components of $S^1\-\{\ze_0,\dots,\ze_{k+1}\}$ should map to $\hat L_0,\dots,\hat L_k$ respectively.  

For $\ga\ge0$ we denote by $\cM_\ga(L_1,\cdots,L_k;L_0)$ the subset made from $J$-holomorphic curves of area $\ga.$
\end{dfn}
 
We recall the maximum principle on Riemann surfaces with Neumann boundary conditions, under the assumption that the almost complex structure is compatible with the contact boundary:
\begin{lem}[{Abouzaid--Seidel \cite[Lemma 7.2]{AS}}]
There exists a compact set of $X\-\d X$ which contains the image of every non-constant $J$-holomorphic curve included in $\cM(L_1,\cdots,L_k;L_0).$
\qed
\end{lem}
Hence by the Gromov compactness theorem we get
\begin{cor}\l{cor: Gromov}
Every $\cM_\ga(L_1,\cdots,L_k;L_0)$ is compact. \qed
\end{cor}

Suppose now that $L_0,\cdots,L_k$ are given local systems $E_0,\cdots, E_k$ respectively.
Note that for every $i=1,\cdots,k$ there is an evaluation map
\e\l{ev_i} \ev_i:\cM(L_1,\cdots,L_k;L_0)\to L_{i-1}\t_X L_i,\e
defined by $\ev_i(u):=\lim_{\th\to+0}(u(e^{-i\th}z_i),u(e^{i\th}z_i))$ where $u$ is the lifted map $S^1\to \hat L,$ and $z_i$ the $i^{\rm th}$ marked point on $S^1.$ Define $\ev_0:\cM(L_1,\cdots,L_k;L_0)\to L_0\t_X L_k$ by $\ev_0(u):=\lim_{\th\to+0}(u(e^{i\th}z_i),u(e^{-i\th}z_i)).$
So $\cM(L_1,\cdots,L_k;L_0)$ has local systems of the form $\ev_i\pi_\pm^*E_j$ for $i,j=0,\dots,k.$
Using the parallel transport map $\ev_i^*E_i\cong \ev_{i+1}^*E_i$ we get a natural map of local systems
\e\l{loc maps} \bot_{i=1}^k\hom(\ev_i^*E_{i-1},\ev_i^*E_i)\to\hom (\ev_0^*E_0,\ev_0^*E_k).\e 

Let $\Ga\sb[0,\iy)$ be a discrete sub-monoid which contains the areas of every pseudo-holomorphic disc with boundary on the underlying Lagrangians of $\cC(X).$
The gapped $A_N$ categories in what follows will be gapped with respect to $\Ga.$
We fix an order on the countable set $\cC(X)$ and use the same order on the underlying Lagrangians of it.
Here are the next three steps:
\iz
\item[\bf (I)] For each $N\in\N$ construct a gapped $A_N$ category $\cC_N$ whose objects are the first $N$ Lagrangians and whose morphism space between $(L_0,E_0),(L_1,E_1)$ is \eq{E0E1}. We denote this morphism space by $\cC(E_0,E_1).$
As Fukaya, Oh, Ohta and Ono \cite[Remark 13.5]{FOOO-Z} point out we can work with the whole complex \eq{E0E1}, not with countably generated sub-complexes of it as in the earlier works \cite{AJ,FOOO}.
\item[\bf (II)] Construct a homotopy equivalence from $\cC_N$ to the $A_N$ truncation of $\cC_{N+1}.$
\item[\bf(III)] Lift the gapped $A_N$ structure of $\cC_N$ to the gapped $A_{N+1}$ structure.
\iz 
We do (I) as follows.
By Corollary \ref{cor: Gromov} every moduli space of the form $\cM_\ga(L_1,\cdots,L_k;L_0)$ is a finite union of Kuranishi spaces, which are oriented naturally by the relative spin structure.
Suppose that for each $i=1,\dots, k$ we are given $x_i\in \cC(E_{i-1},E_i).$
We choose a multivalued perturbation of the fibre product 
\e\l{pert} \cM_\ga(L_1,\cdots,L_k;L_0)\t_\ev(x_1\t\cdots \t x_k),\e
where $\ev:=(\ev_1,\cdots,\ev_k),$ whenever $[k,\ga]\le N.$
Here we can in fact make a single-valued perturbation when the symplectic manifold $(X,\om)$ is spherically positive in the sense of Fukaya, Oh, Ohta and Ono \cite{FOOO-Z}; this is possible mainly because of Remarks 12.6 and 12.7 of their paper. Combining this perturbation with \eq{loc maps} we get an element of $\cC(E_0,E_k)$ which we call the virtual chain for $(x_1,\dots,x_k).$ It has coefficients in the ground field $\k\subset\La$ by the result of Fukaya, Oh, Ohta and Ono \cite{FOOO-Z}.
Varying $x_1,\cdots,x_k$ we get a map
\e\l{m^k_ga} \fm^k_\ga:\prod_{i=1}^k\cC(E_{i-1},E_i)\to \cC(E_0,E_k).\e
One can prove that there is a consistent choice of the virtual chains above such that $(\fm^k_\ga)$ defines the gapped $A_N$ structure.

Step (II) is done by a standard cobordism (or bifurcation) argument. 
Step (III) is a consequence of Theorem \ref{thm: ob1}.
Once these have been done, we can define a gapped $A_\iy$ category whose $A_N$ truncation is homotopy equivalent to $\cC_N.$ Applying the homological perturbation lemma to $\cC_N$ we get a gapped $A_\iy$ category that we want.
\qed

\subsection{Bounding Cochains}\l{sec: bounding}

We define the units and bounding cochains of filtered $A_\iy$ categories.
\begin{dfn}\l{dfn: strict A_infty}
Let $(\cA,\fm)$ be a filtered $A_\iy$ category. By {\it units} of this we mean a family $(e_X\in\cA^0_{XX})_{X\in\obj\cA}$ of degree-zero morphisms such that $\fm^1_X e_X=0$ for $X\in\obj\cA;$ $x=\fm^2_{XXY}(e_X,x)$ for $x\in\cA_{XY}$ with $X,Y\in\obj\cA;$  $x=(-1)^{\deg x}\fm^2_{YXX}(x,e_X)$ for $x\in\cA_{YX}$ with $X,Y\in\obj\cA;$ and $\fm^k_{X_0\dots X_k}(x_1,\dots, x_k)=0$ for $X_0,\dots,X_k\in\obj\cA$ with $k\ge3$ and for $x_1\in\cA_{X_0X_1},\dots,x_k\in\cA_{X_{k-1}X_k}$ containing at least one of $e_{X_1},\dots,e_{X_k}.$

We call $(\cA,\fm)$ a {\it strict} $A_\iy$ category if $\fm^0_X=0\in\cA^1_{XX}$ for every $X\in\obj\cA.$
If it is strict then define the cohomology category $H\cA$ to be the associative category consisting of the same objects; the morphism space $(H\cA)_{XY}:=H^*(\cA_{XY},\fm^1_{XY})$ for $X,Y\in\obj\cA;$ and the product map induced from the $\fm^2$ operators. Define also the category $H^0\cA$ to be the subcategory of $H\cA$ consisting of the same objects and whose every morphism space $H^0\cA_{XY}$ is the degree-zero part $H^0(\cA_{XY},\fm^1_{XY}).$

Let $\cA,\cB$ be filtered strict $A_\iy$ categories and $f:\cA\to\cB$ a filtered $A_\iy$ functor. We call $f$ a {\it strict} $A_\iy$ functor if $f^0_X=0\in \cA^0_{XX}.$ This induces then a functor between the cohomology categories, in the obvious way.

Let $X\in\obj\cA$ be any object and $b\in\cA^1_{XX}$ a degree-one element with $\v b>0.$ We call $b$ a {\it bounding cochain} if $\sum_{k=0}^\iy\fm^k_{X\dots X}(b^{\t k})=0.$ The sum converges because $\v b>0.$
We define then a filtered strict $A_\iy$ category $(\cA',\fn).$ Define $\obj\cA'$ to be the set of pairs $(X,b)$ with $X\in\obj\cA$ and $b\in\cA^1_{XX}$ a bounding cochain. For $Y_0=(X_0,b_0),$ $Y_1=(X_1,b_1)\in\obj\cA'$ set $\cA'_{Y_0Y_1}:=\cA_{X_0X_1}.$ 
For $Y_0=(X_0,b_0),\dots,Y_k=(X_k,b_k)\in\obj\cA'$ define $\fn^k_{Y_0\dots Y_k}:\cA'_{Y_0Y_1}[1]\times\dots\times\cA'_{Y_{k-1}Y_k}[1]\to \cA'_{Y_0Y_k}[1]$ by
\[(x_1,\dots,x_k)\mapsto\ts\sum_{p_0,\dots,p_k=0}^\iy\fm^{k+p_0+\dots+p_k}(b_0^{\t p_0}, x_1, b_1^{\t p_1},\dots,b_{k-1}^{\t p_{k-1}}, x_k, b_k^{\t p_k}).\]
The sum converges because $\v b_0>0,\dots,\v b_k>0.$
Since $\v\fm^{k+p_0+\dots+p_k}\ge0$ it follows that $\v\fn^k_{Y_0\dots Y_k}\ge0;$ that is, $\fn^k_{Y_0\dots Y_k}$ is a {\it filtered} homomorphism.
The $A_\iy$ equations of $(\cA,\fm)$ imply those of $(\cA',\fn).$
Since $\sum_{k=0}^\iy\fm^k(b^{\t k})=0$ it follows that $\fn^0=0.$ Thus $(\cA',\fn)$ is a filtered strict $A_\iy$ category.

If $\cA$ is a filtered $A_\iy$ category with units then the strict $A_\iy$ category $\cA'$ corresponding to it has units too. This is a straightforward extension of the $A_\iy$ algebra case \cite[Lemma 5.2.17]{FOOO}.
\end{dfn}

We introduce now the strict $A_\iy$ category $\cF_\nc(X)$ where `nc' stands for not necessarily closed Lagrangians.
\begin{dfn}\l{dfn: bound coch}
We define the strict $A_\iy$ category $\cF_\nc(X)$ to be that made from $\cC(X)$ by using bounding cochains as above. 
We denote an object of $\cF_\nc(X)$ by a triple $\b=(L,E,b)$ where $(L,E)$ is an object of $\cC(X)$ and $b$ a bounding cochain in $CF^*(E,E).$
We say that $\b$ is {\it supported} on $L$ or that $L$ {\it underlies} $\b.$
We define the $\Z$-graded $\La$-linear associative category $H\cF_\nc(X)$ to be the cohomology category of $\cF_\nc(X).$
We denote by $HF^*(\b_0,\b_1)$ its morphism space between two objects $\b_0,\b_1.$
\end{dfn}
The following is a straightforward extension of the work of Fukaya, Oh, Ohta and Ono \cite[\S7.3]{FOOO} for embedded Lagrangians. For $\k=\R$ or $\C$ the de Rham version of this is made already by Fukaya \cite{Fuk}.
\begin{thm}\l{thm: unit}
The filtered $A_\iy$ categories $\cC(X),\cF_\nc(X)$ have units. \qed
\end{thm}
\begin{rem}\l{rem: unit}
For $\k=\R$ or $\C$ we can make the units directly as Fukaya \cite{Fuk} does, using the de Rham model for \eq{CF}. Strict unitality fails for the singular chain model because the perturbations of Kuranishi spaces of the form \eq{pert} need not be compatible with forgetting marked points. One can however prove that for $(L,E)\in \cC(X)$ there exists a {\it homotopy} unit of $CF^*(E,E)$ whose leading term is equal to the fundamental class of $L.$ This may be done by taking fundamental chains of the underlying Lagrangians in $\cC(X)$ and constructing in a consistent way the homotopies between those perturbed Kuranishi spaces for which some of $x_1,\cdots,x_k$ in \eq{pert} are the fundamental chains, and those obtained from the forgetful maps. The homotopy units become the ordinary units after the homological perturbations.
\end{rem}

We recall the notion of zero objects in $H\cF(X)$ and give a corollary to Theorem \ref{thm: unit}.
\begin{dfn}
We call $\b\in \obj\cF_\nc(X)$ a {\it zero} object if the unit of $HF^*(\b,\b)$ is zero.
\end{dfn}
\begin{cor}\l{cor: unit}
If $\b=(L,E,b)\in \obj\cF_\nc(X)$ is a zero object, the Lagrangian $L\subset X$ has at least one self-intersection point of index $-1.$ 
\end{cor}
\begin{proof}
Since $HF^*(\b,\b)$ is $\Z$-graded it follows that the unit in $HF^0(\b,\b)$ may be cancelled out only by elements of $CF^{-1}(E,E),$
which come from self-intersection points of $L$ of index $-1.$
\end{proof}
Corollary \ref{cor: unit} gives another proof of the following fact.
\begin{cor}[Fukaya, Oh, Ohta and Ono {\cite[Theorem E]{FOOO}}]\l{embedded non-zero}
Every $\b\in \obj\cF_\nc(X)$ supported on an embedded Lagrangian is non-zero. \qed
\end{cor}

Take $\Om^{\otimes2},J$ as in Definition \ref{dfn: Z-grading} so that we can speak of special Lagrangians in $X.$
Consider {\it nearly} special Lagrangians which have phase sufficiently close (in the $C^0$ sense) to a constant function. The estimate \eq{mu} implies then 
\begin{prop}\l{prop: TY0}
If $\b\in \obj\cF(X)$ is supported on a generically immersed nearly special Lagrangian then the Floer cochain group $CF^*(\b,\b)$ is supported in degrees $0,\cdots,n;$ and accordingly, so is the cohomology group $HF^*(\b,\b).$ \qed
\end{prop}
This and Corollary \ref{cor: unit} imply
\begin{cor}\l{cor: special non-zero}
Every $\b\in \obj\cF_\nc(X)$ supported on a nearly special Lagrangian is non-zero. \qed
\end{cor}

We define the Fukaya category which we use most often.
\begin{dfn}\l{dfn: Fuk}
We denote by $\cF(X)\sb \cF_\nc(X)$ the full subcategory of objects supported on {\it closed} Lagrangians; that is, the underlying Lagrangian of an object $\b$ of $\cF(X)$ is a Lagrangian immersion $L\to X$ from a compact manifold $L$ without boundary. 
\end{dfn}

We prove now a Poincar\'e duality theorem. We are in the right context for doing so; but we shall not need the result for the later treatment, and the reader in a hurry can proceed safely to Proposition \ref{prop: shift} below.
\begin{thm}\l{thm: PD}
Let $\b$ be an object of $\cF(X)$ whose local system has one-dimensional fibres. 
Then there exists a non-degenerate pairing
  \begin{equation}
    HF^*(\b,\b) \otimes HF^{n-*}(\b,\b) \to \La.
  \end{equation}
\end{thm}
\begin{proof}
We begin by recalling the notion of filtered $A_\iy$ bimodules.
Let $C$ be a $\Z$-graded $\La$-linear filtered curved $A_\iy$ algebra.
By a {\it bimodule} over $C$ we mean
a $\Z$-graded $\La$ module $\cP$ equipped with operations
\begin{equation}
  \label{eq:bimodule_structure_maps}
  \fm_{\cP}^{r|1|s} :  \left(C[1]\right)^{\otimes r} \otimes \cP \otimes \left(C[1]\right)^{\otimes s} \to  \cP,
\end{equation}
$r,s=0,1,2,\cdots,$ which satisfy the bimodule version of the $A_\infty$ equation. An $A_\infty$ bimodule bimodule is {\it filtered} if it is equipped with a valuation, so that $A_\infty$ bimodule operations do not decrease valuations.
\begin{exam}
The $A_\iy$ algebra $C$ itself has the structure of a {\it diagonal} bimodule with operations
  \begin{equation} \label{eq:product_diagonal_bimodule}
  \fm^{r|1|s} (x_1, \ldots, x_r, p, y_1, \ldots, y_s  ) = (-1)^{r + \sum \deg x_i} \fm^{r+1+s} ( x_1, \ldots, x_r, p, y_1, \ldots, y_s   ),
\end{equation}
$r,s=0,1,2,\cdots;$ $x_1,\cdots,x_r,p,y_1,\cdots,y_r\in C.$
\end{exam}
\begin{exam}
If $\cP$ is an $A_\iy$ bimodule over $C$ then its {\it dual} 
$  \cP^{\vee}:= \hom(\cP^{\vee}, \La)$ has the structure of a bimodule given by
\begin{equation} \label{eq:structure_maps_dual_bimodule}
  \fm_{\cP^{\vee}}(y_1, \ldots, y_s ,  \psi, x_1, \ldots, x_r)(p) = (-1)^\#  \psi\left( \fm_{\cP}(x_1, \ldots, x_r, p,  y_1, \ldots, y_s )\right),
\end{equation}
$r,s=0,1,2,\cdots;$ $x_1,\cdots,x_r,y_1,\cdots,y_r\in C,$ $\ps\in\cP^\vee,$ $p\in \cP$ and
$\#=1+\deg\ps+(\deg y_1+\dots+\deg y_s)(\deg x_1+\dots+\deg x_r+\deg p+\deg \ps).$
\end{exam}

Next let $\b=(L,E,b)\in \obj\cF(X)$ and apply to the single brane $(L,E)$ the proof of Theorem \ref{thm: Fuk cat}.
Then we get a gapped $A_\iy$ algebra $C.$
We use the same notation for the diagonal bimodule, and denote by $C^\vee$ its dual bimodule.
Recall that the $\fm^1$ of $C$ has leading term equal to the usual differential, whose cohomology group is
\e CF^*(E,E):=H^*(L\t_XL, \hom(E,E))[-\mu_L].\e
We prove
\begin{prop}
There exists a $\Z$-graded $\La$-linear natural isomorphism of graded vector spaces
\e\l{PD} CF^*(E,E)\cong CF^*(E,E)^\vee[-n].\e
\end{prop}
\begin{proof}
Recall that Poincar\'e duality makes sense because of the relative spin structure data, so that there is a pairing
\e H^*(\hat{L}, \hom(E,E))\t H^*(\hat{L},\hom(E,E))\to \La[-n]\e
where we do not have to change the local system term because $\hom(E,E)^\vee\cong\hom(E,E).$
For an off-diagonal self-intersection pair $(x,y)\in L\t_X L$ there is also a pairing
\e (x,y)\ot (y,x)\to \La[-n]\e
which completes the proof.
\end{proof}
The following is the heart of the proof of Theorem \ref{thm: PD}.
\begin{lem}\l{lem: bimod hom}
There exists a bimodule homomorphism
\e\l{bimod} \Ps:C\to C^\vee[-n]\e
such that the map $\Ps^{0|1|0}:C\to C^\vee[-n]$ has leading term equal to \eq{PD}.
\end{lem}
\begin{proof}
  Recall that $C$ is by construction an inverse limit of some gapped $A_K$ algebras, which we denote by  $(C_K)_{K=0}^\iy.$
  Regard this as a diagonal bimodule with operations of the form
  \e\fm^{r|1|s}=\sum \fm^{r|1|s}_\ga T^\ga,\e
  $r,s\ge0$ and $[r+s+1,\ga] \leq K.$
  For each $K$ we construct a bimodule map
  $       \Ps_K: C_K \to C^{\vee}_{K}$
which consists of a collection of
\e\l{bimod K}   \Ps_K^{r|1|s}: C_K^{\t r}\t  C_K\t C_K^{\t s}  \to C^{\vee}_{K}[-n]\e
indexed by non-negative integers $r$ and $s$.
We regard the output as an input by duality, so that we have to construct a compatible collection of maps
\e\l{Ps_K}  C_K   \t C_K^{\t r}\t  C_K\t C_K^{\t s}  \to \La[-n].\e
Here we use the hypothesis that the local system of $\b$ has one-dimensional fibres.

In order to construct this map, we use the stable moduli spaces of pseudo-holomorphic discs with $r+s+2$ boundary points, ordered as in Figure \ref{fig:duality}: the marked point on the left arrow corresponds to the input, and that on the right arrow corresponds to the output.
      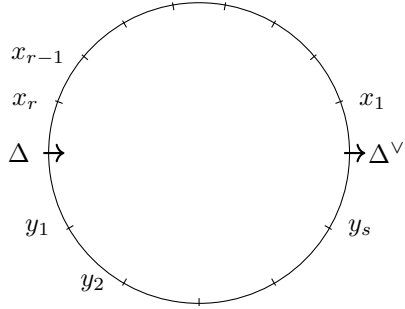
\begin{figure}[h]
        \centering
        \begin{tikzpicture}
          \draw  (0,0) circle [radius=2];
          \node[label=right:{$C^\vee$}] (D) at (0:2) {};
          \node[label=left:{$C$}] (Dv) at (180:2) {};
          \draw[->, thick] (Dv)++(0:-2pt) -- ++(0:8pt);
          \draw[->, thick] (D)++(180:2pt) -- ++(180:-8pt);
          \node[label=left:{$x_r$}] (x1) at (160:2) {};
          \node[label=left:{$x_{r-1}$}] (x2) at (140:2) {};
          \node[label=left:{$y_1$}] (x1) at (210:2) {};
          \node[label=left:{$y_2$}] (x1) at (240:2) {};
          \node[label=right:{$x_1$}] (x1) at (20:2) {};
          \node[label=right:{$y_s$}] (x1) at (330:2) {};
          \foreach \i in {1,...,8}
          \draw[thin] (\i*20:2)++(\i*20:1pt) -- ++(\i*20:-3pt);
          \foreach \j in {7,...,12}
          \draw[thin] (\j*30:2)++(\j*30:1pt) -- ++(\j*30:-3pt);
        \end{tikzpicture}
        \caption{Moduli spaces of discs for the duality map}
        \label{fig:duality}
      \end{figure}
      We choose inductively a family of perturbation data for the fibre products of the form
      \e\cM_\ga^{r+s+2}\t_\ev(q \t x_1\t\cdots \t x_r \t p \t \tau_1\t \cdots \t \tau_s)\e
      where $\ev$ denotes the evaluation maps at the boundary marked points ordered along the boundary and $q,x_1,\dots,x_r,p,y_1,\dots,y_s$ are singular simplices in $C$ with $[r+s+1,\ga]\le K.$
We make these compatible with the previous perturbations of \eq{pert} and those for the lower-order maps relative to $r,s.$ Then we get a $A_K$ bimodule homomorphism, which we denote by $\Ps_K: C_K \to C^{\vee}_{K}[-n].$

From the $A_K$ bimodule maps of this form we construct an $A_\iy$ bimodule map, by an induction and obstruction theory argument similar to that in the proof of Theorem \ref{thm: Fuk cat}. We give a brief account of this; for more details see Fukaya, Oh, Ohta and Ono \cite[\S5.2]{FOOO} (who use their results to prove the deformation invariance of Floer cohomology groups \cite[Theorem D]{FOOO} which is in fact a homotopy equivalence of bimodules).

     Recall that we have constructed an $A_K$ homotopy equivalence between $C_K$ and the $A_K$ truncation of $C_{K+1}.$
     From this we get an $A_{K+1}$ algebra $C_{K,K+1}$ which is $A_{K+1}$ homotopy equivalent to $C_{K+1}$ and whose $A_K$ truncation is $C_{K}.$ 
     The $A_{K+1}$ bimodule map
$                C_{K+1} \to C^{\vee}_{K+1}[-n]$
which induces an $A_{K+1}$ bimodule map (over the algebra $C_{K,K+1}$)
           \begin{equation}
\l{stage K}               C_{K,K+1} \to C^{\vee}_{K,K+1}[-n].
      \end{equation}
      which we denote still by $\Ps_K.$ In the previous stage we have constructed an $A_K$ bimodule map
           \begin{equation}
             \Ps_{K-1}:  C_{K-1,K} \to C^{\vee}_{K-1,K}[-n].
      \end{equation}
      This is homotopic to the $A_K$ reduction of \eq{stage K} by a choice of homotopy between the defining perturbation data. Using homological perturbation, and proceeding by induction, we get the desired system of bimodule maps.
\end{proof}

Combining Lemma \ref{lem: bimod hom} with the $A_\iy$ bimodule version of Whitehead's theorem \cite[Theorem 5.2.35]{FOOO} we get
\begin{cor}
$\Ps$ is a homotopy equivalence of bimodules.\qed
\end{cor}

The bimodule homotopy equivalence $\Psi : C\to C^\vee[-n]$ induces for every bounding cochain $b\in C^1$ 
another bimodule homotopy equivalence
\e\Psi_b : C\to C^{\vee}[-n],\e
relative to the $b$-deformed $A_\iy$ structure; and for $x_1,\cdots,x_r,p, y_1,\cdots, y_s\in CF^*(E,E)$ the term
\e
    \Psi^{r|1|s}_{b}(x_1, \ldots, x_r, p, y_1, \ldots, y_s  ) 
    \e
    is given by 
\e    \sum \Psi(b, \ldots, b, x_1, b, \ldots, b, x_r, b, \ldots, b, p, b, \ldots, b, y_1,b, \ldots, b , y_s,b , \ldots ,b). \e
Hence, passing to the cohomology groups, we get an isomorphism
\e HF^*(\b,\b)\cong HF^*(\b,\b)^\vee[-n]\e
which completes the proof of Theorem \ref{thm: PD}.
\end{proof}

Theorem \ref{thm: PD} implies
\begin{cor}\l{cor: PD}
Let $\b$ be a non-zero object of $\cF(X)$ whose local system has one-dimensional fibres. 
Then $HF^n(\b,\b)\ne0.$ \qed
\end{cor}
\begin{rem}\l{rem: PD}
This is proved already for embedded Lagrangians without local systems by Fukaya, Oh, Ohta and Ono \cite[Theorem E]{FOOO}.
The second author \cite[Lemma 4.4]{I} gives two other proofs, including immersed Lagrangians. One uses the cyclic symmetry \cite{FukCyc} assuming that the ground field $\k$ contains $\R.$ The other uses open-closed maps \cite[Theorem 6.4.2]{FOOO}. 
\end{rem}

Corollary \ref{cor: PD} gives a sufficient condition for the hypothesis of the following proposition.
\begin{prop}\l{prop: shift}
Let $\b,\b'\in \obj\cF(X)$ be isomorphic in $H^0\cF(X)$ and supported on closed special Lagrangians of phase $\ph,\ph'\in\R$ respectively which need not be $0$ modulo $2\pi\Z.$ 
Suppose $HF^n(\b,\b)\ne0.$
Then $\ph'=\ph.$ \qed
\end{prop}
\begin{proof}
Put $\th:=\frac{\ph'-\ph}{2\pi}\in\R.$ Recall then from \eq{mu0} that the graded vector space $HF^k(\b,\b)\cong HF^k(\b,\b')$ vanishes unless $k\in [\th,n+\th].$ On the other hand, recall from Corollary \ref{cor: special non-zero} that $\b,\b'$ are non-zero objects so that both $HF^0(\b,\b)$ and $HF^n(\b,\b)$ are non-zero. So $0\in[\th,n+\th]$ and $n\in[\th,n+\th],$ which imply $\th=0.$ 
\end{proof}  

\subsection{Thomas--Yau Theorems}\l{sec: TY}
We recall a few generalizations of Thomas--Yau's theorem proved by the second author \cite{I}.
In this subsection we work with Calabi--Yau manifolds and integrable complex structures.
\begin{prop}\l{prop: TY}
 Let $X$ be a Calabi--Yau manifold of complex dimension $n.$ Let $\b,\b'\in \obj\cF(X)$ be supported on closed irreducibly immersed special Lagrangians $L,L'$ of phase $\ph,\ph'$ respectively which need not be $0$ modulo $2\pi\Z.$ Suppose that either $L$ or $L'$ embedded. Then the following hold. 
\iz
\item[\bf(i)] If $\ph=\ph'$ then either $L=L'\subset X$ or 
$HF^*(\b,\b')$ is supported in degrees $1,\cdots,n-1.$
\item[\bf(ii)] If $\b,\b'$ are isomorphic in $H^0\cF(X)$ with $HF^n(\b,\b)\ne0$ then $L=L'\subset X.$
\iz
\end{prop}
\begin{rem}\l{rem: TY}
The point of (i) is that the degrees are strictly $>0$ and $<n.$
It follows immediately from \eq{mu} that they belong to $\{0,\dots,n\}.$
\end{rem}
\begin{proof}[Proof of Proposition $\ref{prop: TY}$]
Let $L$ be embedded. Then we can work in a Weinstein neighbourhood of $L$ which is a Stein manifold. This is the condition (ii) of the second author's result \cite[Theorem 6.2]{I}, which implies (i) above. We prove (ii) now. By Proposition \ref{prop: shift} the two special Lagrangians have the same phase. Hence it follows by (i) that either $L=L'$ or $HF^*(\b,\b')$ is supported in $1,\dots,n-1.$ But the latter is impossible because the graded vector space $HF^*(\b,\b')\cong HF^*(\b,\b)$ has two non-trivial degrees with difference $n.$ So $L=L'.$
\end{proof}

We recall also the following version.
We say that a Calabi--Yau manifold $X$ is {\it real analytic} if the K\"ahler form of $X$ is real analytic with respect to the underlying real analytic structure of the complex manifold $X.$
\begin{prop}\l{prop: TY2}
Let $X$ be a real analytic Calabi--Yau manifold and $L,L'\subset X$ two closed irreducibly-immersed Lagrangians which are, near $L\cap L',$ both special and graded by $0\in\R.$
Then either $L,L'$ agree with each other near $L\cap L'$ or any pair of objects supported on them have Floer cohomology supported in degrees $1,\dots,n-1.$  
\end{prop}
\begin{proof}
Suppose $L\ne L'$ and denote by $S\sb L\cap L'$ the set of intersection points at which $L,L'$ have at least one common tangent space. The second author \cite[Theorem 5.1(i)]{I} proves then that there exists a neighbourhood $U\subset X$ of $S$ and arbitrarily-small Hamiltonian perturbations of $L,L'$ respectively which intersect generically each other with no intersection point in $U$ of index $0$ or $n.$
On the other hand, these have no intersection point of index $0$ or $n$ outside $U$ \cite[Lemma 2.7]{I}.
So they have no intersection point at all of index $0$ or $n.$
Since $L,L'$ are both graded by $0\in\R$ it follows by \eq{mu} that every intersection point of the two perturbed Lagrangians has degree in $\{1,\dots,n-1\}.$ So their Floer complex is supported in degrees $1,\dots,n-1,$ which completes the proof.
\end{proof}

\subsection{Hamiltonian Continuations}\l{sec: Ham}
We return now to the symplectic topology context. 
We recall the definition of homotopies between strict $A_\iy$ functors.
\begin{dfn}
Let $\cA,\cB$ be strict $A_\iy$ categories. Then we can define the functor category $\fct(\cA,\cB)$ which is a strict $A_\iy$ category whose objects are strict $A_\iy$ functors from $\cA$ to $\cB.$
We say that two objects $f,g\in\obj\fct(\cA,\cB)$ are {\it homotopic} if they are isomorphic in $H^0\fct(\cA,\cB).$
Here we do not impose any conditions on the object sets as we do in Definition \ref{dfn: homotopies}.

We say that a strict $A_\iy$ functor $f:\cA\to\cB$ is a {\it homotopy equivalence} if there exists $g\in\fct(\cB,\cA)$ such that $g\cm f$ is isomorphic in $H^0\fct(\cA,\cA)$ to the identity on $\cA$ and $f\cm g$ isomorphic in $H^0\fct(\cB,\cB)$ to the identity on $\cB.$
Seidel \cite[Theorem 2.9]{Seid3} proves that $f$ is a homotopy equivalence if and only if the induced functor $Hf:H\cA\to H\cB$ is an equivalence of categories. 
\end{dfn}

We explain now the effect of symplectomorphisms of $X$ upon the Fukaya categories.
Recall from Definition \ref{dfn:Liouv} that the radius function $r_\la$ is defined near $\d X.$
Note also that symplectomorphisms act naturally upon the objects of $\cC(X).$
\begin{lem}\l{lem: continuation}
Let $\ph:X\to X$ be a symplectomorphism isotopic to the identity and which agrees near $\d X$ with the Hamiltonian diffeomorphism of $cr_\la$ for some $c>0.$
Suppose that $\ph$ induces a bijection of $\cC(X).$
Then there exists an equivalence
\e\l{ph*}\ph_*:H\cF_\nc(X)\to H\cF_\nc(X)\e
such that for every $\b=(L,E,b)\in \obj\cF_\nc(X)$ the image $\ph\b$ is given by a bounding cochain $\ph_*b$ for $(\ph_*L,\ph_*E)\in\cC(X).$
The functor $\ph_* $  is compatible with composition in the sense that, for a pair $\ph$ and $\ps$ of boundary-linear symplectomorphisms, 
the composite $\ps_* \circ \ph_*$ is naturally isomorphic to $(\ps\circ\ph)_*.$
\end{lem}
\begin{proof}
The strict $A_\iy$ category $\cF_\nc(X)$ is made from a gapped $A_\iy$ category $\cC(X)$ whose definition involves the choice of an almost complex structure $J$ and other data.
We can then define another gapped category by using $\ph_*J.$
We want to show that the two gapped $A_\iy$ categories are homotopy equivalent. This is done by Fukaya, Oh, Ohta and Ono \cite[Theorem A]{FOOO} and by Akaho--Joyce \cite[Theorem 11.2(a)]{AJ} for gapped $A_\iy$ algebras by interpolating between the almost complex structures;  we do so while maintaining compatibility with the contact boundary, thus ensuring compactness. 
They use the obstruction theory arguments for gapped $A_N$ functors \cite[Lemma 7.2.129]{FOOO} for gapped $A_\iy$ algebras. We extend them to gapped $A_\iy$ categories in the same way as in the proof of Theorem \ref{thm: ob1}. 
Then we get a homotopy equivalence of the two gapped categories.
Fukaya \cite[Proposition 13.18]{Fuk} proves that this induces a homotopy equivalence of the corresponding strict categories.
There is thus an equivalence \eq{ph*} of the cohomology categories. 

To find the natural isomorphism between $\ps_* \circ \ph_*$ and $(\ps\circ\ph)_*$ we begin with  homotopies of gapped $A_N$ categories. We lift them to gapped $A_\iy$ categories by an obstruction theory for homotopies of homotopies. This is done by Fukaya, Oh, Ohta and Ono \cite[Corollary 7.2.219]{FOOO} for gapped $A_N$ algebras. We extend it in the same way as in the proof of Theorem \ref{thm: ob1}. Then we get a homotopy between  $\ps_* \circ \ph_*$ and $(\ps\circ\ph)_*$ as gapped $A_\iy$ functors. This induces again a homotopy between the strict $A_\iy$ categories, which induces in turn a natural isomorphism that we want.
\end{proof}

There is more to say if the symplectomorphism $\ph$ in Lemma \ref{lem: continuation} is a Hamiltonian diffeomorphism:
\begin{prop}\l{prop: continuation} 
Let $\ph:X\to X$ be the Hamiltonian diffeomorphism of a smooth function $X\to\R$ which agrees near $\d X$ with $c r_\la$ for some $c>0.$
Suppose that $\ph$ induces a bijection of $\cC(X).$
Then there exists a natural transformation $\ka_\ph$ from $\ph_*$ to the identity functor; that is,
for any $\b_1,\b_2\in\cF_\nc(X)$ there is a commutative diagram
\begin{equation} \label{eq:continuation_is_natural}
  \begin{tikzcd}
    HF^*(\b_1,\b_2)  \ar[dr,swap, "\kappa_{\phi}\b_1\circ"] \ar[r,"\phi_*"] & HF^*( \ph\b_1, \ph\b_2) \ar[d, "\circ\kappa_\ph\b_2 "] \\
  &  HF^*(\ph\b_1,\b_2).
  \end{tikzcd}
\end{equation}
If either $\b_1$ or $\b_2$ are supported on a compact Lagrangian,
then these arrows are isomorphisms.
If there is another such Hamiltonian diffeomorphism $\ps:X\to X$
then 
\e\l{ka functorial} \ka_{\ps\ph}\b=\ka_\ps(\ph\b)\cm \ka_\ph \b,\;\;\;\b\in\obj\cF(X).\e
\end{prop}
\begin{proof}
We explain how to deduce the statement from the result of Fukaya, Oh, Ohta and Ono \cite[Theorem G (G.4)]{FOOO} who prove that pairs of Hamiltonian diffeomorphisms induce bimodule homomorphisms of Floer groups (Akaho--Joyce explain in \cite[Theorem 13.6]{AJ} how to extend this to immersed Lagrangians, and it is also straightforward to include local systems). As in Lemma \ref{lem: continuation}, the construction of the bimodule homomorphisms involves pseudo-holomorphic curve equations with respect to varying almost complex structures on the target, and we require their compatibility with the contact boundary to ensure compactness.

Define $\ka_\ph \b\in HF^*(\ph \b,\b)$ to be the image of the unit of $HF^*(\b,\b)$ under the bimodule homomorphism $(\ph,\id)_*:HF^*(\b,\b)\to HF^*(\ph\b,\b)$ induced by the pair $(\ph,\id)$ of Hamiltonian diffeomorphisms.
We prove then that the diagram \eq{eq:continuation_is_natural} commutes.
The pair $(\ph,\id)$ of Hamiltonian diffeomorphisms induces the bimodule homomorphism $(\ph,\id)_*:HF^*(\b_2,\b_2)\to HF^*(\ph\b_2,\b_2).$ In particular, this is a left module homomorphism and there is a commutative diagram
\begin{equation}\l{ph,id}
\begin{tikzcd}
HF^*(\b_1,\b_2)\times HF^*(\b_2,\b_2)\ar[d]\ar[r]& HF^*(\b_1,\b_2)\ar[d]\\
HF^*(\ph \b_1,\ph \b_2)\times HF^*(\ph \b_2,\b_2)\ar[r]& HF^*(\ph\b_1,\b_2).
\end{tikzcd}
\end{equation}
Consider now the direct sum $\b_1\op\b_2.$ If these are supported on distinct Lagrangians $L_1,L_2$ respectively then regard $\b_1\op\b_2$ as an object of $\cF_\nc(X)$ supported on $L_1\cup L_2.$ If they are supported on the same Lagrangian then regard it as an object of $\cF_\nc(X)$ with local system $E_1\op E_2$ where $E_1,E_2$ are the local systems of $\b_1,\b_2$ respectively.
In either case the pair $(\ph,\id)$ induces a bimodule homomorphism $(\ph,\id)_*:HF^*(\b_1,\b_1\op\b_2)\to HF^*(\ph \b_1,\b_1\op \b_2).$ In particular, this is a right module homomorphism and there is a commutative diagram
\[
\begin{tikzcd}
HF^*( \b_1,\b_1\op\b_2)\times HF^*(\b_1\op\b_2,\b_1\op\b_2)\ar[d]\ar[r]& HF^*(\b_1,\b_1\op\b_2)\ar[d]\\
HF^*(\ph \b_1,\b_1\op \b_2)\times HF^*(\b_1\op \b_2,\b_1\op\b_2)\ar[r]& HF^*(\ph\b_1,\b_1\op\b_2).
\end{tikzcd}
\]
We look at the following component, which commutes too:
\begin{equation}\l{ph,id2}
\begin{tikzcd}
HF^*( \b_1,\b_1)\times HF^*(\b_1,\b_2)\ar[d]\ar[r]& HF^*(\b_1,\b_2)\ar[d]\\
HF^*(\ph \b_1,\b_1)\times HF^*(\b_1,\b_2)\ar[r]& HF^*(\ph\b_1,\b_2).
\end{tikzcd}
\end{equation}
This combined with \eq{ph,id} implies the commutative diagram \eq{eq:continuation_is_natural}.

We finally prove \eq{ka functorial}.
Using \eq{ph,id2} with $\ps$ in place of $\ph,$ $\ph\b$ in place of $\b_1,$ and $\b$ in place of $\b_2$ we get a commutative diagram
\begin{equation}
\begin{tikzcd}
HF^*(\ph\b,\ph\b)\times HF^*(\ph\b,\b)\ar[d]\ar[r]& HF^*(\ph\b,\b)\ar[d]\\
HF^*(\ps\ph \b,\ph\b)\times HF^*(\ph\b,\b)\ar[r]& HF^*(\ps\ph\b,\b).
\end{tikzcd}
\end{equation}
Consider the unit of $HF^0(\ph\b,\ph\b)$ and $\ka_\ph\b\in HF^*(\ph\b,\b)$ and look at the image of this pair. Then we see that
\e\l{ka psph}\ka_\ps(\ph\b)\cm\ka_\ph\b=(\ps,\id)_*\ka_\ph\b=(\ps,\id)_*(\ph,\id)_*\b.\e
But Fukaya, Oh, Ohta and Ono \cite[Theorem G (G.4)]{FOOO} prove $(\ps,\id)_*(\ph,\id)_*=(\ps\ph,\id)_*$ so the right-hand side of \eq{ka psph} is equal to $(\ps\ph,\id)_*\b=\ka_{\ps\ph}\b.$
This completes the proof.
\end{proof}

\subsection{Wrapped Fukaya Categories}\l{sec: loc}
In general, if $\cA$ is an $A_\iy$ category and $Z \sb H^0\cA$ a set of morphisms, 
there exists a universal pair consisting of an $A_\iy$ category $\cA_{Z^{-1}}$ and an $A_\iy$ functor $\cA\to\cA_{Z^{-1}}$ such that every element of $Z$ becomes invertible in $H^0 \cA_{Z^{-1}}.$ The key property that we shall use is the fact that this construction is compatible with passage to triangulated categories in the sense that there is a natural equivalence of triangulated categories
\begin{equation}
  H^0 \left( \cA_{Z^{-1}} \right) \cong (H^0 \cA)_{Z^{-1}}
\end{equation}
where the right-hand side is the localisation of a triangulated category.

The next result is a key computational tool for computing Fukaya categories from localisations, and also implies the above equivalence of triangulated categories in this specific context:
\begin{prop}\l{prop: loc}
Let $\cA$ be an $A_\iy$ category,
and $f:H\cA\to H\cA$ an equivalence.
For objects $a,b\in\obj\cA$ write $\cA(a,b)$ the morphism space of the pair $(a,b).$
Let $\kappa$ be a natural transformation from $f$ to the identity functor;
and $\cA_{\kappa^{-1}}$ the localisation of $\cA$ by the set $\{\kappa_a \in \cA(fa,a)\}_{a\in\obj\cA}$.
Then for any two objects $a,b\in\obj\cA$ there exist natural isomorphisms
\begin{multline}
  \l{HW colim}
H\cA_{\kappa^{-1}}(a,b)\cong \dirlim_{i\to\iy} H\cA(f^ia,b) \cong 
 \dirlim_{j\to\iy} H\cA(a,f^{-j}b)
\cong \dirlim_{i,j\to\iy} H\cA(f^ia,f^{-j}b)    
\end{multline}
where $f^{-1}:H\cA\to H\cA$ is an inverse to $f.$
Also for $a,b,c\in\obj\cA$ there exists a commutative diagram
\begin{equation}\l{products in HW}
\begin{tikzcd}
\dirlim_{i\to\iy}H\cA(f^ia,b)\ot\dirlim_{j\to\iy}H\cA(b,f^{-j}c)\ar[r]\ar[d,"\rotatebox{270}{$\sim$}"]&
\displaystyle
\dirlim_{i,j\to\iy}H\cA(f^ia,f^{-j}c)\ar[d,"\rotatebox{270}{$\sim$}"]\\
\displaystyle
H\cA_{\kappa^{-1}}(a,b)\ot H\cA_{\kappa^{-1}}(b,c)\ar[r]& H\cA_{\kappa^{-1}}(a,c)
\end{tikzcd}
\end{equation}
where the horizontal arrows are products maps and the vertical isomorphisms are obtained from those of \eq{HW colim}.
\end{prop}
\begin{proof}
Assuming that the category $ \cA$ is triangulated, the localisation $\cA_{Z^{-1}} $  may be constructed as the categorical quotient of $\cA$ by the subcategory consisting of the cones of all morphisms $\kappa : a \to f(a)$.  Morphisms in the quotient category are then computed by the direct limit of the groups $H\cA(a,f^{-j}b) $ as proved in \cite[Lemma 7.18]{SeidLec}. The fact that this is isomorphic to the direct limit of the groups $H\cA(f^ia,b) $ is then a consequence of iteratively applying the functor, and the expression in terms of the direct limit of the groups $H\cA(f^ia,f^{-j}b)$ is implied by the constancy of the colimit with respect to either variable. The expression for the product is an immediate consequence of functoriality. 
\end{proof}
We shall also need the following observation:
\begin{lem} \label{lem:localisation_does_not_change}
  Assume that $K \sb Z\sb H^0\cA$ are two sets of morphisms such that for each element $\beta \in Z$, there exists  a morphism $\gamma$ such that $\beta \circ \gamma$ is a composition of morphisms in $K$. 
  Then the natural functor $H\cA_{K^{-1}} \to H\cA_{Z^{-1}}$ is an equivalence.
\end{lem}
\begin{proof}
  It suffices to prove that $ \cA_{K^{-1}} $ satisfies the universal property of the localisation away from $Z$, i.e.\ that every morphism in $Z$ is already invertible in $ \cA_{K^{-1}} $. This is immediate from the assumption because a composition of morphisms in $K$ becomes invertible  in $ \cA_{K^{-1}} $, and the only way for a product of morphisms to be invertible is if both factors are so.
\end{proof}

We now arrive at the definition of the wrapped category in the contact-type setting:
\begin{dfn}\label{localisation}
Let $\Ph$ be a set of Hamiltonians $\ph:X\to\R$ which are linear at the boundary in the sense that they agree with  $\ph=cr,$ for some constant $c>0,$ near $\d X$. Assume that this set contains Hamiltonians whose slopes are arbitrarily large.
Identify these with the time-one maps they generate, and suppose that if $\ph\in\Ph$ and $E\in \cC(X)$ then $\ph_*E \in\cC(X).$
Then $\cW(X)$ is defined as the {\it localisation} of $\cF_\nc(X)$ by the set
$\{\ka_\ph\in HF^0(\ph\b,\b):\ph\in\Ph, \b\in \obj\cF_\nc(X)\}$
of continuation morphisms.
\end{dfn}

As a consequence of Lemma \ref{lem:localisation_does_not_change} this definition does not depend on the set of Hamiltonians which are chosen: if we add or subtract a Hamiltonian diffeomorphism to the set $\Phi$, then the resulting localisation does not change because of the existence of continuation maps to Hamiltonians of larger slope. 

The closed Fukaya category $\cF(X)$ is embedded naturally as a full subcategory of $\cW(X)$
because the continuation morphisms are invertible in $H^0\cF(X).$

\subsection{Exact Lagrangians}\l{sec: ex}
Suppose now that the Liouville $1$-form $\la$ is given globally on $X$ so that we can speak of exact Lagrangians in $(X,\la)$ and we can make the following definition.
\begin{dfn}
Denote by $\cF_\ex(X)\sb \cF_\nc(X)$ the full subcategory of those objects $(L,E)$ for which
\iz
\item
the Lagrangian $L\subset X$ is embedded and exact; that is, $L$ has no self-intersection point and $\la|_L$ is an exact $1$-form;
\item
the local system $E$ is of the form $\cE\ot_\k\La$ where $\k$ is the ground field of $\La,$ and $\cE$ some $\k$ local system $\cE$ on $L;$ and
\item
$(L,E)$ is given the trivial bounding cochain, which exists because
$L$ is embedded and does not bound any non-constant pseudo-holomorphic discs.
\iz
For each $(L,E)\in\cF_\ex(X)$ choose and fix a primitive function $h:L\to\R$ with $dh=\la|_L.$ For $(L_0,E_0),(L_1,E_1)\in\cF_\ex(X)$ define the {\it action} function 
\e \cA:L_0\t_XL_1\to\R \e
by $\cA(x):=h_0(x)-h_1(x)$ where $h_0,h_1$ are the primitives of $(L_0,E_0),(L_1,E_1)$ respectively.
\end{dfn}
We recall the following fact. 
\begin{prop}
Let $(L_0,E_0),\cdots,(L_k,E_k)\in\cF_\ex(X)$ and $h_0,\cdots,h_k$ their primitives respectively.
As in the proof of Theorem \ref{thm: Fuk cat} take on $X$ a compatible almost-complex structure $J,$ define the moduli space $\cM(L_1,\cdots,L_k;L_0)$ and define the evaluation map
\e\l{ev}\cM(L_1,\cdots,L_k;L_0)\to \Biggl[\prod_{i=0}^{k-1}(L_i\t_XL_{i+1})\Biggr]\times (L_0\t_XL_k).\e 
Then the fibre of this map over a point $(x_1,\cdots,x_k;x_0)$ consists of $J$-holomorphic discs of area
\e\l{Action} \left(\sum_{i=1}^k\cA(x_i)\right)-\cA(x_0).\e
\end{prop}
\begin{proof}
In the cyclic notation, modulo $k+1$ for indices, the area is equal to
\e\begin{split}
\sum_{i=1}^{k+1} \int_{x_{i-1}}^{x_i}\la&=\sum_{i=1}^{k+1} h_{i-1}(x_i)-h_{i-1}(x_{i-1}) 
=\sum_{i=1}^{k+1} h_{i-1}(x_i)-\sum_{i=0}^k h_i(x_i)\\
&=\left(\sum_{i=1}^k h_{i-1}(x_i)-h_i(x_i)\right)+h_k(x_0)-h_0(x_0)
\end{split}
\e
which is equal to \eq{Action}.
\end{proof}
This implies that the $A_\iy$ structure of $\cF_\ex(X)$ is reducible to $\k$:
\begin{cor}\l{cor: ex}
$\cF_\ex (X)$ has the structure of a $\k$-linear $A_\iy$ category
whose hom space between two objects $(L_0,\cE_0),(L_1,\cE_1)$ is given by
\e
CF_\ex^*(\cE_0,\cE_1):=H^*(L_0\t_XL_1,\hom_\k(\cE_0,\cE_1)[-\mu_{L_0L_1}])
\e
and whose structure maps $\fm_\ex^k$ are defined by the same disc counts but without the $T^\ga$ factor.
The $\La$-linear extension of this $A_\iy$ category is isomorphic to that of the subcategory $\cF(X)$ with the same objects, under the maps
\e \l{eq:exact_to_non-exact}
T^\cA:CF_\ex^*(\cE_0,\cE_1)\ot_\k\La\to CF^*(E_0,E_1),\;\;\;E_i=\cE_i\otimes_\k\La,
\e
defined as the multiplication by $T^{\cA(x)}$ on the component of each $x\in L_0\t_XL_1.$
\qed
\end{cor}
\begin{proof}
This follows from the proof of Theorem \ref{thm: Fuk cat};
the point is that once we have fixed the data $(x_1,\cdots,x_k;x_0)$ relevant to the $A_\iy$ operation in Equation \eq{m^k_ga},
the $\ga$ in the structure map $\fm^k_\ga$ will be equal to \eq{Action}.
\end{proof}

Choose now a countable set of Hamiltonians as in Definition \ref{localisation}.
Denote by $\cW_\ex(X)$ the localisation of $\cF_\ex(X)$ with respect to this.
Using \eq{eq:exact_to_non-exact} we regard $\cW_\ex(X)$ as a full subcategory of $\cW(X).$

Seidel \cite{Seid3} makes another definition of Fukaya categories of exact Lagrangians. This is different in two respects from that definition made in \S\ref{sec: Fuk2}. One is that Seidel defines the hom space $CF^*(L,L)$ of a single exact Lagrangian by introducing a Hamiltonian perturbation $\ph L$ which intersects $L$ transversely, and does Lagrangian Floer theory for the pair $(\ph L, L).$ The other is that Seidel perturbs the moduli spaces of pseudo-holomorphic curves by direct perturbations without using the Kuranishi space machinery. 

We compare now the two definitions. Denote by $\cF_{\rm dir}(X)$ the Fukaya category Seidel defines with the same object set as $\cF_\ex(X).$ This is already a strict $A_\iy$ category because the underlying Lagrangians are exact. Denote by $\cW_{\rm dir}(X)$ its localization with respect to the same Hamiltonians as we use for $\cF_\ex(X)\sb \cF(X).$ We prove then the following proposition, in which we consider only the cohomology categories rather than the original $A_\iy$ categories; this will do for our later treatment.
\begin{prop}\label{prop: ex}
$H\cW_\ex(X)$ is equivalent to $H\cW_{\rm dir}(X).$
\end{prop}
\begin{proof}
We define a functor $H\cF_\ex(X)\to H\cF_{\rm dir}(X).$ We leave out the local systems to save notation.
For two exact embedded Lagrangians $L_a,L_b$ we define a $\k$-linear map
$HF^*_\ex(L_a,L_b)\to HF^*_{\rm dir}(L_a,L_b).$
Suppose either $L_a$ or $L_b$ has no boundary.
Then $HF^*_{\rm dir} (L_a,L_b)$ is defined as $HF_\ex^*(\ph_aL_a,\ph_bL_b)$ for some Hamiltonian diffeomorphisms with compact support in $X\-\d X.$
More precisely, the Floer group $HF_\ex^*(\ph_aL_a,\ph_bL_b)$ is defined by the direct perturbations. But these may be regarded as a different choice of virtual chains, so the resulting Floer groups will be canonically isomorphic. This is done by Fukaya, Oh, Ohta and Ono \cite[Theorem G]{FOOO} (in the context of the present paper the similar arguments have appeared in the proof of Lemma \ref{lem: continuation}).
There is thus a canonical isomorphism $HF^*_{\rm dir} (L_a,L_b)\cong HF_\ex^*(\ph_aL_a,\ph_bL_b).$
There is another isomorphism $HF_\ex^*(\ph_aL_a,\ph_bL_b)\cong HF_\ex^*(L_a,L_b)$ from which we get an isomorphism $HF^*_{\rm dir} (L_a,L_b)\cong HF_\ex^*(L_a,L_b)$ that we want.

Suppose $L_a,L_b$ have both Legendrian collar. Then the morphisms of $(L_a,L_b)$ form the wrapped Floer groups $\varinjlim HF_\ex^*(\ph L_a, L_b)$ and $\varinjlim HF_{\rm dir}^*(\ph L_a, L_b).$
Since the isomorphism
$HF_\ex^*(\ph L_a, L_b)\to HF^*_{\rm dir}(\ph L_a, L_b)$ exists for every $\ph$ we get an isomorphism $HW_\ex^*(L_a, L_b)\to HW^*_{\rm dir}(L_a, L_b).$

We prove that the $\k$-linear maps thus defined are compatible with the products. 
Let $a,b,c\in\N$ and suppose first that at least two of $L_a,L_b,L_c$ have no boundary. 
Write $HF^*_{\rm dir}(L_a,L_b)=HF_\ex^*(\ph_aL_a,\ph_bL_b)$ and $HF^*_{\rm dir}(L_b,L_c)=HF_\ex^*(\ph_bL_b,\ph_cL_c).$
Put $K:=\ph_b^{-1}\ph_aL_a\sqcup \ph_bL_b.$
Recall that the isomorphism $HF_\ex^*(K,L_c)\cong HF_\ex^*(\ph_b K,\ph_cL_c)$ is a bimodule isomorphism with respect to $(\ph_b)_*:HF_\ex^*(K,K)\to HF_\ex^*(\ph_bK,\ph_bK)$ and $(\ph_c)_*:HF_\ex^*(L_c,L_c)\to HF_\ex^*(\ph_cL_c,\ph_cL_c).$ It is in particular a left module isomorphism with respect to $(\ph_b)_*$ so the diagram
\[
\begin{tikzcd}
HF_\ex^*(K,K)\times HF_\ex^*(K,L_c)\ar[r]\ar[d]& HF_\ex^*(K,L_c)\ar[d]\\
HF_\ex^*(\ph_bK,\ph_bK)\times HF_\ex^*(\ph_bK,\ph_cL_c)\ar[r]& HF_\ex^*(\ph_bK,\ph_cL_c)
\end{tikzcd}
\]
commutes. The upper arrow contains the product map $HF_\ex^*(L_a,L_b)\times HF_\ex^*(L_b,L_c)\to HF_\ex^*(L_b,L_c)$ and the lower arrow contains the product map $HF_\ex^*(\ph_aL_a,\ph_bL_b)\times HF_\ex^*(\ph_bL_b,\ph_cL_c)\to HF_\ex^*(\ph_bL_b,\ph_cL_c).$ These are thus compatible.

Suppose now that at least two of $L_a,L_b,L_c$ have Legendrian collar. Then $HW^*_{\rm dir}(L_a,L_b)=\varinjlim HF_\ex^*(\ph^nL_a,L_b)$ and $HW^*_{\rm dir}(L_b,L_c)=\varinjlim HF_\ex^*(L_b,\ph^{-n}L_c).$
The isomorphisms
$HF_\ex^*(\ph^n L_a, L_b)\to HF^*_{\rm dir}(\ph^n L_a, L_b)$
and
$HF_\ex^*(L_b, \ph^{-n}L_c)\to HF^*_{\rm dir}(L_b, \ph^{-n}L_c)$
are compatible with the product maps, completing the proof.
\end{proof}

\subsection{Results in Cotangent Bundles}\l{sec: cotangent}
Let $Q$ be a closed connected manifold, and $X\subset T^*Q$ a compact disc sub-bundle. Identify again the non-compact fibre $T^*_qQ$ over a point $q\in Q$ with the compact fibre in $X$ over the same point. Recall that the wrapped Floer group $HW_\ex^*(T_q^*Q,T_q^*Q)$ may be computed by using a result of Abbondandolo, Portaluri and Schwarz \cite[Theorem 7.1]{APS}: there is a $\Z$-graded $\k$-algebra isomorphism
\e\l{HW based loop, intro} HW_\ex^*(T_q^*Q,T_q^*Q)\cong H_{-*}(\Om_qQ,\k)\e
where $\Om_qQ$ denotes the based loop space and $H_{-*}(\Om_qQ)$ is non-positively graded. We shall not need an $A_\iy$ lift of \eq{HW based loop, intro} as treated by the first author \cite{Ab3}.

To define the Fukaya categories $\cF(X),$ $\cW(X)$ we do the following:
\iz
\item
We take the background class on $T^*Q$ to be the pull-back by the projection $T^*Q\to Q$ of the Stiefel--Whitney class $w_2(Q)\in H^2(Q,\Z/2\Z).$
So the zero-section $Q\subset T^*Q$ has a relative spin structure, and we give it a natural one which we shall define shortly below.
\item
Take the $\Z$-grading of $T^*Q$ to be that by the complexification of a squared real volume form on $Q.$
So there exist $\Z$-gradings of $Q.$
We include in $\cF(X)$ the object supported on $Q$ with the trivial rank-one local system and the trivial bounding cochain.
\item
We give $T^*Q$ the standard Liouville $1$-form $\la$ so that $Q$ is an exact Lagrangian.
We fix a point $q\in Q$ and include in $\cW(T^*Q)$ the object supported on the fibre $T_q^*Q$ with the trivial relative spin structure, an arbitrary $\Z$-grading, the trivial rank-one local system and the trivial bounding cochain.
\iz
Here is the natural relative spin structure on $Q.$ Note that the $\fm^2$ operator $HW^0(T_q^*Q,T_q^*Q)\times HF^*(T_q^*Q,Q)\to HF^*(Q,T_q^*Q)$ and the isomorphism $HW^0(T_q^*Q,T_q^*Q)\cong\La[\pi_1(Q,q)]$ defines a representation $\pi_1(Q,q)\to \GL HF^*(Q,T_q^*Q).$ We often put $\pi_1Q=\pi_1(Q,q)$ for short.
\begin{lem}\l{lem: torsor}
The Lagrangian $Q\subset T^*Q$ has a relative spin structure such that the $\pi_1Q$ representation on $HF^*(T_q^*Q,Q)$ is trivial. If $\k$ has characteristic $2$ this representation is automatically trivial.
\end{lem}
\begin{proof}
If $\k$ has characteristic $2$ then $Q$ defines an object of the exact Fukaya category over $\Z/2\Z$
and the $\pi_1Q$ representation on $HF^*(T_q^*Q,Q)$ makes sense over $\Z/2\Z.$
But this is automatically trivial because $\Z/2\Z$ has only one unit.

Suppose now that $\k$ has characteristic $\ne2$ and choose at first any  relative spin structure on  $Q$. 
Then $Q$ defines an object of the exact Fukaya category over $\Z$ (it is only here in the present paper that we use Fukaya categories over a ring).
The object $Q$ defines a representation $\pi_1Q\to\GL_1\Z \cong \Z/2 \Z$ and such representations are classified by elements of $H^1(Q,\Z/2\Z).$  
Since the set of relative spin structures is a torsor over the same group $H^1(Q,\Z/2\Z),$ we can then choose a possibly different relative spin structure which makes this representation trivial.
\end{proof}

We compute the representations corresponding to $C^1$ perturbations of the zero-section $Q\subset T^*Q.$
\begin{lem}\l{lem: smooth deformations}
Let $\al$ be a closed $1$-form on $Q$ whose graph $Q^\al\subset T^*Q$ lies in the sub-bundle $X\subset T^*Q.$
Let $E$ be a filtered local system over $Q^\al$ and regard this as an object of $\cF(X),$ which makes sense because $H_2(T^*Q,Q^\al;\Z)=0$ and no non-constant pseudo-holomorphic curve exists for $(T^*Q,Q^\al).$
Denote by $T^{-[\al]}:\pi_1Q\to\La^*$ the one-dimensional representation $x\mapsto T^{-[\al]\cdot x}.$
Regard $E$ as the representation of $\pi_1Q^\al\cong\pi_1Q.$
Then $T^{-[\al]}\ot E$ is isomorphic to the $\pi_1Q$ representation on $HW^*(T_q^*Q,E)$ corresponding to the $\fm^2$ operator $HW^0(T_q^*Q,T_q^*Q)\times HF^*(T_q^*Q,E)\to HF^*(T_q^*Q,E)$ under the isomorphism $HW^0(T_q^*Q,T_q^*Q)\cong\La[\pi_1Q].$
\end{lem}
\begin{proof}
Let $x\in \pi_1(Q,q)$ be any element. 
Then there exists a corresponding element of $HW^*_\ex(T_q^*Q,T_q^*Q)$ under the isomorphism $HW^*_\ex(T_q^*Q,T_q^*Q)\cong \k[\pi_1Q].$
Choose a Hamiltonian diffeomorphism $\ph:T^*Q\to T^*Q$ such that $\ph(T_q^*Q)$ intersects $T_q^*Q$ transversely at $q\in Q$ in the zero-section and there exists an element of $CF^*_\ex(\ph(T_q^*Q),T_q^*Q)$ corresponding to $x.$

Denote by $\ta:T^*Q\to T^*Q$ the fibrewise translation by $\al.$ Put $\la':=\ta_*\la=\la+\pi^*\al$ where $\pi$ is the projection $T^*Q\to Q.$ This induces a category isomorphism $H\cF_\ex(X,\la)\to H\cF_\ex(X,\la')$ and accordingly a commutative diagram
\[
\begin{tikzcd}
HF_\ex^*(\ph(T_q^*Q),T_q^*Q)\times HF_\ex^*(T_q^*Q,\ta^*E)\ar[r]\ar[d]& HF_\ex^*(\ph(T_q^*Q),\ta^*E)\ar[d]\\
HF_\ex'(\ph(T_q^*Q),T_q^*Q)\times HF_\ex'(T_q^*Q,E)\ar[r]& HF_\ex'(\ph(T_q^*Q),E)
\end{tikzcd}
\]
where we have written $HF'$ for the morphism spaces of $H\cF_\ex(X,\la').$
Write $CF_\ex'$ for the morphism spaces of $\cF_\ex(X,\la').$
Since $CF_\ex'(T_q^*Q,Q^\al)$ has no differential it follows then that the $\fm^2$ operator
\[CF_\ex'(\ph(T_q^*Q),T_q^*Q)\times CF_\ex'(T_q^*Q,E)\to CF_\ex'(\ph(T_q^*Q),E)\]
agrees with the parallel transport map of $E|_q.$
The areas of the $J'$-holomorphic discs are equal to $\int_q^x\la'+\int_x^q\la'+\int_q^q\la'$ which we explain next.
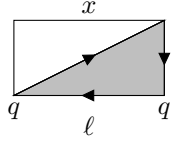
\begin{figure}
  \centering
  \begin{tikzpicture}
         \draw [fill=lightgray] (0,0)--(2,0)--(2,1)--cycle;
       
         \draw (2,0)--(0,0) node[
    currarrow,
    pos=0.5, 
    xscale=-1,
    sloped,
    scale=1] {};
         \draw (0,0)--(2,1) node[
    currarrow,
    pos=0.5, 
    xscale=1,
    sloped,
    scale=1] {};
    \draw (2,1)--(2,0) node[
    currarrow,
    pos=0.5, 
    xscale=1,
    sloped,
    scale=1] {};
         \draw (0,0)--(2,1)--(0,1)--cycle;
         \node  at (0,-0.2) {$q$};
          \node  at (2,1.2) {$x$};
\node at (2,-0.2) {$q$};
\node at (1,-0.4) {$\ell$};
\node at (1,1.2) {$\bar x$};
         \end{tikzpicture}
         \caption{a pseudo-holomorphic disc for $(q,\hat x,q)$}\l{fig: tri}
\end{figure}
The first path from $q$ to $x$ is the diagonal from the bottom left to the top right in Figure \ref{fig: tri}, which lies on $\ph(T_q^*Q);$ and the integral over this is equal to $h_{\ph(T_q^*Q)}(x)-h_{\ph(T_q^*Q)}(q)=\cA(x).$ The second path from $x$ to $q$ is the vertical line in Figure \ref{fig: tri}, which lies on $T_q^*Q;$ and the integral over this vanishes. The third path from $q$ to $q$ is the horizontal line in Figure \ref{fig: tri}, which lies on $Q;$ and the integral over this is equal to $-[\al]\cdot x.$
Thus the map
\[CF^*(\ph(T_q^*Q),T_q^*Q)\times CF^*(T_q^*Q,E)\to CF^*(\ph(T_q^*Q),E)\]
maps $(x,q')$ to $T^{\cA(x)-\al\cdot x}\rho(x)q'.$
So $T^{\cA(x)}x$ acts by $T^{-\al\cdot x}$ upon $CF^*(T_q^*Q,E).$
As the functor $\ta_*:H\cF(X,\la)\to H\cF(X,\la')$ is an isomorphism, $T^{\cA(x)}x$ has the same action upon $CF^*(T_q^*Q,\ta^*E),$ which completes the proof.
\end{proof}

We need some more preparation for the proof of Theorem \ref{thm: main} (ii)--(iv).  We will give a general enough treatment which applies to all the three cases. There would in fact be a simpler way if we wanted to prove only (ii); we could then focus on a single loop space $\Om_qQ$ whereas we shall in what follows be concerned with many path spaces $\Om_{q_0q_1}.$ The latter will be needed, however, at the end of \S\ref{sec: Yoneda}; and the results in it will be used for the proof of (iii),(iv).

Fix an arbitrary triangulation $\De$ of $Q$ which contains the base-point $q\in Q$ and such that the fibre $T^*_{q_0}Q$ over each vertex $q_0$ of $(Q,\De)$ is an underlying exact Lagrangian in the category $\cW(X).$ Denote by $\Om_{q_0q_1}$ the space of paths from one vertex $q_0$ to another vertex $q_1$ and recall that there is a $\La$-vector space isomorphism $H_0(\Om_{q_0q_1},\La)\cong HW^0(T_{q_0}^*Q,T_{q_1}^*Q).$ Given three vertices $q_0,q_1,q_2\in Q$ there is a map $\Om_{q_0q_1}\times\Om_{q_1q_2}\to \Om_{q_0q_2}$ defined by concatenating paths. Using the isomorphisms $H_0(\Om_{q_aq_b},\La)\cong HW^0(T_{q_a}^*Q,T_{q_b}^*Q)$ for $a,b=0,1,2$ we get then a $\La$-bilinear map
\e\l{concat}
HW^0(T_{q_0}^*Q,T_{q_1}^*Q)\times HW^0(T_{q_1}^*Q,T_{q_2}^*Q)\to HW^0(T_{q_0}^*Q,T_{q_2}^*Q)
\e
which in fact agrees with the $\fm^2$ product map in $\cW(X).$

Let $\b$ be an object of $\cF(X)$ and put $V:=HF^*(T_q^*Q,\b),$ a finite-dimensional $\La$-vector space. Denote by $\rho:\pi_1(Q,q)\to \GL V$ the $\La$-linear representation defined by the $\fm^2$ operator $HW^0(T_q^*Q,T_q^*Q)\times V\to V$ and the isomorphism $HW^0(T_q^*Q,T_q^*Q)\cong\La[\pi_1Q].$ Let $\rho:\pi_1(Q,q)\to \GL V$ have a sub-representation $\si:\pi_1(Q,q)\to \GL W$ on some {\it one-dimensional} subspace $W\sb V.$ 

We define then an {\it approximate} local system $(W_{q_0})_{q_0\in\text{ is a vertex of }(Q,\De)}$ of one-dimensional valued $\La$-vector spaces. More precisely, this is a contravariant functor from the full subcategory $\Pi_1(Q,\De)$ of the fundamental groupoid $\Pi_1Q$ whose objects are the vertices of $(Q,\De).$ We define the target category $FV^1$ by saying that its objects are one-dimensional valued $\La$-vector spaces and that its morphisms are plain vector space homomorphisms (which need not be filtered).

For each vertex $q_0\in Q$ the one-dimensional valued $\La$-vector space $W_{q_0}$ is defined as follows. Let $x\sb\Om_{q_0q}$ be any connected component and define $W_{q_0}\sb HF^*(T_{q_0}^*Q,\b)$ to be the image of $\{x\}\times W$ under the $\fm^2$ operator $HW^0(T_{q_0}^*Q,T_q^*Q)\times HF^*(T_q^*Q,\b)\to HF^*(T_{q_0}^*Q,\b).$ We also write $W_{q_0}=:xW$ for short, which is in fact independent of the choice of $x.$ For if $y\sb\Om_{q_0q}$ is another component then as $\fm^2$ is associative we can write $yW=xx^{-1}yW$ using a homotopy inverse $x^{-1}.$ Noting that $x^{-1}y\in \pi_1(Q,q)$ and recalling that $W$ is a representation of $\pi_1(Q,q)$ we find $x^{-1}y W=W.$ So $yW=xx^{-1}yW=xW.$ We have also $x^{-1}xW=W,$ which implies that the multiplication by $x$ is bijective and that $xW$ is one-dimensional over $\La.$ We give $xW\sb HF^*(T_q^*Q,\b)$ the filtration induced from that of $HF^*(T_q^*Q,\b).$ Thus $W_{q_0}=xW$ is a one-dimensional valued $\La$-vector space.

Note that for any second vertex $q_1\in Q$ and for any morphism from $q_0$ to $q_1$ the $\fm^2$ product map $H_0(\Om_{q_0q_1},\La)\times W_{q_1}\to W_{q_0}$ induces a $\La$-vector space homomorphism $W_{q_1}\to W_{q_0}.$ This defines the desired functor $\si:\Pi_1(Q,\De)\to FV^1$ because the $\fm^2$ operator \eq{concat} is given by concatenation of paths.

Note that the valuation of linear maps between one-dimensional $\La$-vector spaces may be defined in a natural way such that there is a functor $\v:FV^1\to\R,$ where the latter is regarded as a groupoid of one object whose hom space is the additive group $\R$ composed by the ordinary addition $\R\times\R\to\R.$  Composing the two functors $\si:\Pi_1(Q,\De)\to FV^1$ and $\v:FV^1\to\R$ we get a functor $\v\si:\Pi_1(Q,\De)\to\R.$ 

We introduce another piece of notation: if $e$ is a directed edge in $(Q,\De)$ from a vertex $q_0$ to another vertex $q_0$ then regard $e$ as a path $[0,1]\to Q$ and denote by $[e]\in H_0(\Om_{q_0q_1},\La)$ its homotopy class relative to the end points. We prove now
\begin{lem}\l{lem: est for rep}
In the circumstances above there exists a constant $C>0,$ independent of $\b\in\obj\cF(X),$ such that for every directed edge $e$ of $(Q,\De)$ we have $|\v\si[e]|\le C.$
\end{lem}
\begin{proof}
Recall that the wrapped category $\cW(X)$ is defined by an $A_\iy$ localization after choosing a countable set $\cH$ of Hamiltonians $X\to\R$ which are linear near the boundary of the disc sub-bundle $X\subset T^*Q.$ Suppose first that for any two vertices $q_0,q_1$ of $(Q,\De)$ no Reeb chords from $T_{q_0}^*Q$ to $T_{q_1}^*Q$ have integer length. Then for any $H\in\cH$ its Hamiltonian chords stay in the region where the Hamiltonian is not linear, and accordingly, away from the boundary of $X\subset T^*Q.$ Using this fact we can find a constant $C>0$ such that for any Hamiltonian $H\in\cH,$ for any two vertices $q_0,q_1\in Q$ and for any Hamiltonian chord $\ga:[0,1]\to X\subset T^*Q$ with $\ga(0)\in T_{q_0}^*Q$ and $\ga(1)\in T_{q_1}^*Q$ we have $\sup_{t\in[0,1]}|H\cm \ga(t)| \le \frac C2$

For each directed edge $e$ in $(Q,\De)$ between vertices $q_0,q_1$ fix a finite sum $\sum_ic_i\ga_i\in HW_\ex^0(T_{q_0}^*Q,T_{q_1}^*Q),$ where $c_i\in\k$ and $\ga_i$'s are Hamiltonian chords from $T_{q_0}^*Q$ to $T_{q_1}^*Q,$ that maps to the class $[e]\in H_0(\Om_{q_0q_1},\La)$ under the isomorphism $HW_\ex^0(T_{q_0}^*Q,T_{q_1}^*Q)\cong H_0(\Om_{q_0q_1},\La).$ 
Let $C>0$ be so large that for every such $e$ and $\ga_i$ we have always $\left|\int_0^1 \ga_i^*\la\right|\le \frac C2,$ which is possible because $(Q,\De)$ is a {\it finite} complex.

Recall now that $\si [e]=(\rho [e])|_{W_{q_0}}:=\fm^2(\sum c_iT^{\cA(\ga_i)}\ga_i,*)|_{W_{q_0}}$ and that $\v\fm^2$ is a filtered homomorphism.
Then $\v\si [e]\ge \min\{\cA(\ga_i)+\v\fm^2(\ga_i,*)\}\ge \min\{\cA(\ga_i)\}$ and 
\e\l{v si} -\v \si[e]=\max\{-\cA(\ga_i)\}\le \left|\int_0^1 \ga_i^*\la\right|+\left|\int_0^1H\cm \ga_i(t) dt\right|\le \frac C2+\frac C2=C.\e
Replacing $e$ by its reverse edge $e^{-1}$ we get also $\v\si[e]=-\v\si[e^{-1}]\le C.$

The integer length condition on Reeb vector fields is a generic condition which may be satisfied by re-scaling the symplectic form $\om$ and the Liouville form $\la$ as Abouzaid--Seidel \cite[last pargraph of \S3.1]{AS} do. That is,  consider the fibrewise re-scale of $T^*Q$ by $\ka>0.$ The pull-back by this of the canonical $1$-form $\la$ on $T^*Q$ is $\ka\la.$ Re-scale also by $\ka$ the Hamiltonians used above. Replacing $C$ above with $C\ka^{-1}$ we get then the same conclusion  $\v\si[e]\le C.$
\end{proof}
Given a functor $\ta:\Pi_1(Q,\De)\to \R$ we can define a $1$-cochain $\ta'\in C^1(\De,\R)$ by $\ta' e=\ta[e]$ for every directed edge $e$ of $(Q,\De).$ We show that $\ta'$ is a cocycle. Let $e_1,\dots,e_\nu$ be directed edges of $(Q,\De)$ such that the starting point of $e_a,$ $a=1,\dots,\nu,$ agrees with the endpoint of $e_{a-1},$ where $e_0:=e_\nu.$ Suppose that the concatenation $x=e_1*\dots* e_\nu$ is a $1$-boundary of $(Q,\De).$ Then $\ta' x$ may be computed under the group homomorphism $\ta:\pi_1(Q,q_0)\to\R$ where $q_0$ is the starting point of $e_0.$  But this may be pushed down to the homology group $H_1(Q,\Z),$ on which the $1$-boundary vanishes; and accordingly, so does $\ta'x.$ Define a norm $|\,\,|$ on $C^1(\De,\R)$ by $|\ta'|:=\max\{|\ta' e|: e\text{ is a directed edge of }(Q,\De)\}.$ Define also $|\ta|$ by $|\ta|:=|\ta'|.$
Lemma \ref{lem: est for rep} may then be re-stated as follows: 
\begin{cor}\l{cor: est for norm}
Let $C$ be as in Lemma $\ref{lem: est for rep}.$ Then $|\v\si|\le C.$ \qed
\end{cor}

We re-state the effect of fibrewise re-scales of $T^*Q.$
\begin{cor}\l{cor: est for rep}
In the circumstances of Lemma $\ref{lem: est for rep}$ suppose that the underlying Lagrangian of $\b\in\obj\cF(X)$ lies within distance $\de>0$ from the zero-section $Q\subset T^*Q.$
Then there exists $C>0$ independent of $\b$ and such that $|\v\si|\le   C\de.$ \qed
\end{cor}

Lemma \ref{lem: est for rep} and Corollary \ref{cor: est for rep} are only a few pieces of action-complete Floer theory \cite{Seid2012,Varolgunes,Venkatesh} and local Fukaya categories \cite{AGV}. In particular, the estimates above (which will do for our purpose) would be far from optimal; for instance, given a homotopy class $x\in \Om_{q_0q_1}$ we could try to bound $|\v\si(x)|$ by means of the least length or energy of $x$ after giving $Q$ a Riemannian metric.

\subsection{The Split-Generation Theorem}\label{sec:gener-fukaya-categ}
We recall the definitions we shall need about split-generations of $A_\iy$ categories. 
\begin{dfn}
Let $\cA$ be a strict $A_\iy$ category with units and $\cF,\cG\sb\cA$ full subcategories.
We say that $\cF$ is {\it split-generated} by $\cG$ if the image of $\cF$ under $\cA\to \Pi\Tw\cA$ lies in the full subcategory split-generated by $\cG.$ Here $\Pi\Tw\cA$ is defined for instance by Seidel \cite[Lemma 4.7]{Seid3}. The full subcategory split-generated by $\cG$ is defined by Seidel \cite[after Lemma 4.8]{Seid3}.
\end{dfn}

We state and prove now the split-generation theorem including non-exact Lagrangians.
\begin{thm}\l{thm: gen}
Let $Q$ be a closed connected manifold, and $X\subset T^*Q$ a compact disc sub-bundle. The Fukaya category $\cF(X)\subset \cW(X)$ is then split-generated by the full subcategory of $\cW(X)$ consisting of the single object $T_q^*Q.$
\end{thm}

We begin by explaining the relevant properties of Viterbo restriction functors. We work in a slightly more general context than above. Let $E$ be a Liouville domain and $X\subset E$ a Liouville subdomain.  We assume that we have fixed a background class in $H^2(E,\Z/2\Z)$ for the construction of Fukaya categories on $E$, and use the restriction of this class to $X$ when discussing Fukaya categories on $X.$ We assume that every underlying Lagrangian $L$ in $\cW_\ex(E)$ that is not contained in $X\-\partial X$ intersects $\partial X$ transversely and is equipped with a primitive function for the Liouville $1$-form, which is constant near $L\cap \partial X$. Note that the embedding $X\subset E$ induces an $A_\iy$ functor
\begin{equation}\l{XE}
    \cF(X)\to\cF(E)
\end{equation}
because of the maximum principle for pseudo-holomorphic discs in $E$ bounded by closed Lagrangians in $X.$ In this setting, we prove the following result in Appendix \ref{sec:proof-proposition}
\begin{prop}\l{prop: restriction compatible}
There is a Viterbo restriction functor $\cW(E) \to \cW(X)$, extending the functor on exact Fukaya categories defined in \cite{AS}.  
\end{prop}

Given the functoriality of restriction maps, we conclude:
  \begin{cor} \label{cor:generation_preserved_by_restriction}
    If $\b$ is an object of $\cF(X)$ whose image in $\cW(E)$ lies in the category split-generated by a collection of exact Lagrangians, then $\b$ lies in the subcategory of $\cW(X)$ split-generated by their images under the restriction functor. \qed
  \end{cor}

We return now to the case where $X$ is a compact disc sub-bundle in $T^*Q.$ In \cite{FSS1} Fukaya, Seidel and Smith construct a Lefschetz fibration with total space a Liouville subdomain of $T^*Q$ (which they call $T^*N$), whose critical points are of either disjoint from the zero-section $Q$ or correspond to the $k$ distinct critical points $q_1,\dots,q_k$ of some Morse function on $Q$ (see the discussion preceding  \cite[Lemma 7]{FSS1}). Taking the inverse image of a small region in the base of the Lefschetz fibration produces a Liouville subdomain $D\subset T^*Q$ containing the zero-section $Q$ and which has the property that all its Lefschetz thimbles intersect $Q$ transversely at $q_1,\dots,q_k$ respectively. We change by some exact $1$-form on $D$ the Liouville $1$-form on it so that there exist on the Lefschetz thimbles some primitives of the new Liouville form which are constant near $q_1,\dots,q_k$ respectively. This is possible because the old Liouville $1$-form is exact on some neighbourhoods of $q_1,\dots,q_k\in D.$ 

We use $D$ with the new Liouville $1$-form and its primitives on the Lefschetz thimbles. Applying to them the doubling construction by Seidel \cite[\S18]{Seid3} we get
\iz
\item[\bf(i)]  a Liouville manifold $E$ (which Seidel denotes by $\tilde{E}$) containing $X$ as a compact Liouville subdomain, and $Q\subset X$ as a closed exact embedded Lagrangian;
\item[\bf(ii)] a collection $S_1,\cdots,S_k \subset E$ of embedded exact Lagrangian spheres (which Seidel denotes by $\tilde{\Delta}_1,\dots,\tilde{\Delta}_k$ respectively) which are given primitives that are constant near $q_1,\dots,q_k\in Q$ respectively, and whose intersections with $X$ are the respective Lefschetz thimbles associated to $q_1,\dots,q_k;$ and 
\item[\bf(iii)]
a symplectomorphism $\ph:E\to E$ such that $\phi X$ is disjoint from $X$ and which is Hamiltonian isotopic to a composition of Dehn twists around $S_1,\dots,S_k;$ see the first paragraph in the proof of \cite[Lemma 18.15]{Seid3}. 
\iz
Since we are interested in studying Fukaya categories that are possibly twisted by background classes, we need to know that any background class on $X$ extends to a background class on $E$:
\begin{lem} \label{lem:double_cover_surject_cohomology}
The restriction map $H^*(E) \to H^*(X)$ admits a splitting which vanishes on every sphere.
\end{lem}
\begin{proof}
Seidel's construction exhibits $E$ as a branched double cover of a Liouville domain containing $X$ as a deformation retract. A splitting is then provided by the double covering map.
\end{proof}
The data (i)--(iii) above and Lemma \ref{lem:double_cover_surject_cohomology} imply that the following holds.
\begin{lem}
  If $X \subset E$ is a sufficiently small Weinstein neighbourhood of $Q$, then
\iz
\item the intersection of each Lagrangian $S_i$ with $X$ is Hamiltonian isotopic to a cotangent fibre, and
\item the image of every closed Lagrangian in $X$ under $\phi$ is disjoint from $X$.\qed
\iz  
\end{lem}
       
Fix a background class on $X$ and the class induced in $E$ by the splitting above. Introduce the Fukaya category $\cF(E)$ which contains the Lagrangians $S_i$ 
with the trivial relative spin structures (which exists because every $S_i$ is the double of a disc in $X$), the trivial rank-one local systems and the trivial bounding cochains.
\begin{rem}
Here we allow the ground field $\k$ to have characteristic $2,$ because we do {\it not} take the $\Z/2\Z$ invariant Fukaya category as Seidel does in the discussion of equivariant Fukaya categories in \cite{Seid3}, which is used in  \cite{FSS1}.
\end{rem}
The key result about this Fukaya category is:
\begin{lem} \label{lem:split_generation_double_cover}
  Every object of $\cF(E)$ supported in $X$ lies in the category split-generated by the spheres $S_i$.
\end{lem}
\begin{proof}
  The proof is the same as that for exact Lagrangians by Seidel \cite[Proposition 18.17]{Seid3}, using  Oh's result \cite[Section 9]{Oh} for Seidel sequences in the non-exact case.  Oh proves indeed that given two objects $\b,\b'$ and a Lagrangian sphere $S$, with corresponding Dehn twist $\tau_S$, we have a null homotopy for the composition
  \begin{equation}
    CF^*(\tau_S \b,S) \ot CF^*(S, \b')  \to   CF^*(\tau_S \b, \b')  \to   CF^*(\b,\b')
\end{equation}
where the first map is given by composition and the second by multiplication with a distinguished closed morphism in $  CF^*(\b,\tau_S \b)$. So the cone of this morphism $ \b \to \tau_S \b$  lies in the category generated by $S$.

Using this sequence repeatedly we see that if $\b$ is an object of $\cF(E)$ supported in $X$ then there is some closed morphism $\al$ from $\b$ to $\phi \b$ whose cone lies in the category generated by $S_1,\cdots, S_k.$
But the Lagrangians underlying $\b$ and $\ph\b$ may be made disjoint after a Hamiltonian isotopy,
so $\al$ is exact and the cone of $\al$ is isomorphic to a direct sum of $\b$ and $\ph\b$ (with shifts).
This implies that $\b$ is a summand of the object lying in the category generated by $S_1,\cdots, S_k$.
\end{proof}

We are ready now to prove the split-generation result:
\begin{proof}[Proof of Theorem $\ref{thm: gen}$]
  Lemma \ref{lem:double_cover_surject_cohomology} implies that any background class on $X$ is the restriction of a background class on $E$. If $\dim Q \neq 2$, the Lagrangian spheres $S_i$ are automatically relatively spin. In the remaining case, we note that the only possible background class is represented by the Poincar\'e dual of the zero section, which corresponds to a cotangent fibre; and its pull-back to $E$ is represented by the inverse image of a cotangent fibre, whose intersection with $S_i$ vanishes with $\Z/2\Z$ coefficients because the submanifolds intersect an even number of times. The result then follows by applying Lemma \ref{lem:split_generation_double_cover} and Corollary \ref{cor:generation_preserved_by_restriction}, using as well the fact that all cotangent fibres are equivalent in the wrapped Fukaya category, so that one suffices.
\end{proof}

\subsection{Yoneda Functors}\l{sec: Yoneda}

We recall now some standard facts about $A_\iy$ Yoneda functors. We follow Seidel \cite{Seid3} with minor modifications.
For a strict $A_\iy$ category $\cA$ with units denote by $\Mod\cA$ the category of left $A_\iy$ modules (Seidel \cite[(2f)]{Seid3} works with right modules in place of left modules, and with cohomology units in place of strict units). 
Seidel \cite[(1l) and (2g)]{Seid3} defines the {\it Yoneda} functor $\cA\to\Mod\cA$ which is a strict $A_\iy$ functor, and proves that this induces a fully faithful functor $H\cA\to H(\Mod\cA)$ between the cohomology categories \cite[Corollary 2.13]{Seid3}. 
 
Seidel \cite[(2.13)]{Seid3} proves also that if $\cA,\cB$ are strict $A_\iy$ categories and $F:\cA\to\cB$ a cohomologically fully faithful $A_\iy$ functor then the pull-back functor $F^*:\Mod\cB\to\Mod\cA$ and the Yoneda functors $\cA\to \Mod\cA,$ $\cB\to\Mod\cB$ fit into the diagram
\begin{equation}\l{eq: Yoneda functorial}
\begin{tikzcd}  [column sep=large]
  \cA \ar[d] \ar[r,"F"]& \cB\ar[d]\\
 \Mod \cA &  \Mod\cB\ar[l,swap, "F^*"]
\end{tikzcd}
\end{equation}
which commute after passing to the cohomology categories.
 
\begin{lem}\l{lem: Yoneda}
Let $\cA$ be a strict $A_\iy$ category with units. Let $\cF,\cG\sb\cA$ be full subcategories such that $\cF$ is split-generated by $\cG.$ Consider the composite $\cF\to\Mod\cG$ of the inclusion functor $\cF\sb\cA,$ the Yoneda functor $\cA\to\Mod\cA$ and the pull-back functor $\Mod\cA\to\Mod\cG.$
Then the induced functor $H\cF\to H(\Mod\cG)$ is fully faithful.
\end{lem}
\begin{proof}
We use a result of Seidel \cite[Lemma 4.7]{Seid3} and its proof. Let $\cA\to \Pi\Tw\cA$ be a split-closure of the triangulated envelop $\Tw\cA$ of $\cA$ as Seidel defines \cite[after Lemma 4.8]{Seid3}. Denote by $\cG'\sb \Pi\Tw\cA$ the full subcategory split-generated by $\cG.$
Then the composite functor $\cF\sb\cA\to\Pi\Tw\cA$ has image in $\cG'.$ Using \eq{eq: Yoneda functorial} twice we get now a commutative diagram
\begin{equation}\l{Yoneda and split-generation}
\begin{tikzcd}
H\cF\ar[r]\ar[drr]&H\cG'\ar[r]\ar[r] &H(\Pi\Tw\cA)\ar[r]& H(\Mod\Pi\Tw\cA)\ar[d]\\
 & & H\cA\ar[r]\ar[u]& H(\Mod\cA)\ar[d]\\
 & & H\cG\ar[r]\ar[uul]\ar[u]& H(\Mod\cG).
\end{tikzcd}
\end{equation}
Here $H\cG\to H(\Mod\cG)$ is induced from the Yoneda functor and accordingly fully faithful. 
As \eq{Yoneda and split-generation} commutes the composite $H\cG\to H(\Mod\cG)$ factoring through $H\cG',H(\Pi\Tw\cA),H(\Mod\Pi\Tw\cA)$ and $H(\Mod\cA)$ is fully faithful. 
But $\cG'$  is split-generated by $\cG$ so  $H\cG'\to H(\Mod\cG)$ is fully faithful too. 
The arrow $H\cF\to H\cG'$ is induced from the Yoneda functor $\cA\to\Mod\cA$ and therefore fully faithful too.
The composite $H\cF\to H\cG'\to H(\Mod\cG)$ is thus fully faithful. 
\end{proof}

We return now to the circumstances of Theorem \ref{thm: gen}. Denote by $CW^*(T_q^*Q,T_q^*Q)$ the endomorphism $A_\iy$ algebra of the single object $T_q^*Q\in \obj\cF(X).$ Lemma \ref{lem: Yoneda} implies then
\begin{cor}
The induced functor $H\cF(X)\to H(\Mod CW^*(T_q^*Q,T_q^*Q))$ is fully faithful. \qed
\end{cor}

We show in Proposition \ref{prop: fin cov} below that this functor is compatible in a certain sense with a finite cover $f:P\to Q.$ Denote by the same $f$ the induced map $T^*P\to T^*Q$ between the cotangent bundles, which is also a finite cover. Fix a point $p\in f^{-1}(q).$ Identify again the non-compact fibre $T^*_pP$ with the compact fibre in $f^*X$ over the same point $p.$

We define Fukaya categories of the pre-image $f^*X\subset T^*P.$ We pull back to $f^*X$ the branes on $X;$ or more precisely, the pull-back of an underlying Lagrangian in $X$ is possibly disconnected, in which case we include all the connected components of it; for instance, we include the fibre $T_p^*P.$ We pull back to $f^*X$ the relative spin structure on $X,$ with which we introduce the Fukaya category $\cF_\nc(f^*X).$


We define then an $A_\iy$ algebra homomorphism $f_*:HF^*(T_p^*P,T_p^*P)\to HF^*(T_q^*Q,T_q^*Q)$ which acts as the set-theoretic push-forward upon the underlying Lagrangians. Choose a compatible almost complex structure on $X$ and pull this back to $f^*X.$ Recall that $CF^*(T_p^*P,T_p^*P), CF^*(T_q^*Q,T_q^*Q)$ are made from gapped $A_N$ categories, $N\in\N=\{0,1,2,\dots\}.$ We define gapped $A_N$ algebra homomorphisms between these which respect all the $\fm^k_\ga$ operators, simply by pushing forward the pseudo-holomorphic discs. We define in the same way an inverse system of homotopy equivalences with which we can pass to the limits. Then we get a gapped $A_\iy$ algebra homomorphism between the curved $A_\iy$ complexes $CF^*(T_p^*P,T_p^*P),CF^*(T_q^*Q,T_q^*Q).$ Including the bounding cochains we get a strict $A_\iy$ algebra homomorphism $f_*:CF^*(T_p^*P,T_p^*P)\to CF^*(T_q^*Q,T_q^*Q)$ that we want.

We pass now to the wrapped categories. By definition there exist a localization functor $\cF_\nc(X)\to \cW(X)$ and accordingly a composite functor $\cF_\nc(f^*X)\to \cW(f^*X).$ Restricting this to the single object $T_p^*P$ we get an $A_\iy$ algebra homomorphism $f_*:CF^*(T_p^*P,T_p^*P)\to CW^*(T_q^*Q,T_q^*Q).$ We pull back to $f^*X$ the set $\Ph$ of Hamiltonians used to define $\cW(X),$ with which we define the wrapped category $\cW(f^*X).$ We show that the $A_\iy$ algebra homomorphism $f_*:CF^*(T_p^*P,T_p^*P)\to CW^*(T_q^*Q,T_q^*Q)$ descends to the hom space $CW^*(T_p^*P,T_p^*P)$ in the localized category $\cW(f^*X).$ Let $\ps:T^*P\to T^*P$ be the time-one diffeomorphism of a Hamiltonian we have chosen, which is the pull-back of some Hamiltonian $\ph\in\Ph.$ Denote by the same $\ph$ the time-one diffeomorphism $X\to X.$ Then the push-forward of $\ps(T_p^*P)$ is equal to $\ph(T_q^*Q)$ and the push-forward of the continuation morphism $\ka_\ps\in HF^*(\ps(T_p^*P),T_p^*P)$ is the continuation morphism $\ka_\ph\in HF^*(\ph(T_q^*Q),T_q^*Q).$ The latter is already invertible in $\cW(X)$ and hence we get by the universal property an $A_\iy$ algebra homomorphism  $CW^*(T_p^*P,T_p^*P)\to CW^*(T_q^*Q,T_q^*Q)$ that we want. 

We define now a pull-back functor $f^*:\cF(X)\to\cF(f^*X)$ between the Fukaya categories whose underlying Lagrangians are compact without boundary. Take $(L,E,b)\in\obj \cF(X)$ with $L$ a Lagrangian, $E$ a local system and $b$ a bounding cochain. We define a gapped $A_\iy$ homomorphism 
  \begin{equation} \label{eq:pull-back_map_curved_Aoo2}
      C_*(L,E) \to   C_*(f^* L,f^* E) 
  \end{equation}
that assigns to each singular chain in $L$ the sum of all possible inverse images in $f^* L.$
Using the fact that discs are simply connected, we see that the moduli space of discs in $f^*X$ is a cover of the moduli space of discs in $X,$ and hence that a choice of virtual chains for $X$ determines the corresponding choice for $f^*X$ by passing to the cover. This implies that \eqref{eq:pull-back_map_curved_Aoo2}  is compatible with the gapped $A_N$ algebra structures on the singular chain complexes $C_*(L,E), C_*(f^*L,f^*E).$ The same argument shows that the homotopies required to construct the inverse system are also compatible with the pull-back by $f.$ So we can take the limit with respect to $N$ and \eqref{eq:pull-back_map_curved_Aoo2} is a gapped $A_\iy$ homomorphism.
This maps the bounding cochain $b\in C_*(L,E)$ to a bounding cochain in $C_*(f^* L,f^* E),$ which defines an $A_\iy$ functor $f^*:\cF(X)\to\cF(f^*X)$ that we want.

Note that the push-forward $A_\iy$ algebra homomorphism $f_*: CW^*(T_p^*P,T_p^*P)\to CW^*(T_q^*Q,T_q^*Q)$ induces a pull-back functor $f^*:\Mod CW^*(T_q^*Q,T_q^*Q)  \to \Mod CW^*(T_p^*P,T_p^*P).$ Hence we get a diagram
\begin{equation} \label{eq:pull-back_map_curved_Aoo}
\begin{tikzcd}
  \cF(X) \ar[d] \ar[r,  "f^*"]& \cF(f^*X)\ar[d]\\
   \Mod CW^*(T_q^*Q,T_q^*Q)  \ar[r,  "f^*"] & \Mod CW^*(T_p^*P,T_p^*P) 
\end{tikzcd}
\end{equation}
where the vertical arrows are the Yoneda functors. We prove
\begin{prop}\l{prop: fin cov}
The diagram \eq{eq:pull-back_map_curved_Aoo} commutes in the sense that for every object $\b\in\cF(X)$ the two Yoneda modules $CF^*(T_q^*Q,\b)$ and $CF^*(T_p^*P,f^*\b)$ are isomorphic $A_\iy$ modules over  $CW^*(T_p^*P,T_p^*P).$
\end{prop} 
\begin{proof}
This follows from the fact that the relevant holomorphic curve moduli spaces are bijective under the covering of the moduli spaces used above; the underlying Lagrangians in $f^*X$ include $T_p^*P,$ which specifies the lift of each disc in $X.$ 
\end{proof}

Finally we extend Lemma \ref{lem: est for rep} to the finite cover $(P,p)$ of $(Q,q).$ Suppose therefore that we are given again a triangulation of $Q$ and pull back this to a triangulation of $P,$ which we denote by $f^*\De.$ Suppose given an object $\b\in\obj\cF(X);$ and using the same notation as in Lemma \ref{lem: est for rep}, put $V:=HF^*(T_q^*Q,\b)$ and denote by $\rho: \pi_1(Q,q)\to \GL V$ the representation given by the $\fm^2$ products. Suppose that the restriction of $\rho$ to the subgroup $\pi_1(P,p)$ has a sub-representation $\si:\pi_1(P,p)\to \GL W$ on some one-dimensional subspace $W\sb V.$ Then we get a functor $\v\si:\Pi_1(P,f^*\De)\to\R$ as in Lemma \ref{lem: est for rep} with $(P,f^*\De)$ in place of $(Q,\De)$ and with $f^*\b\in\obj \cF(f^*X)$ in place of $\b\in\obj\cF(X).$ The estimate in Lemma \ref{lem: est for rep} holds {\it uniformly} with respect to $f:P\to Q;$ that is, the following holds. 
\begin{lem}\l{lem: paths in covers}
There exists a constant $C>0$ independent of $\b\in\obj\cF(X),$ independent of $f:(P,p)\to (Q,q)$ and such that $|\v\si|\le C.$
\end{lem}
\begin{proof}
The estimate is uniform with respect to $f:P\to Q$ because the Hamiltonians may be chosen first in $X.$ We pull them back to $f^*X$ so as to follow the proof of Lemma \ref{lem: est for rep}. Then for every directed edge $e$ in $(P,f^*\De)$ between vertices $p_0,p_1,$ the isomorphism $H_0(\Om_{p_0p_1},\La)\cong HW_\ex^0(T_{p_0}^*P,T_{p_1}^*P)$ maps the class $[e]\in H_0(\Om_{p_0p_1},\La)$ to the finite  sum $\sum_ic_i\ga_i\in HW_\ex^0(T_{p_0}^*P,T_{p_1}^*P)$ where we can use the same $\ga_i$'s as in the proof of Lemma \ref{lem: est for rep}. We can use accordingly the same constant $C>0$ as in the proof of Lemma \ref{lem: est for rep}, which implies the estimate above we want. 
\end{proof}
The re-scaling argument in Corollary \ref{cor: est for rep} implies 
\begin{cor}\l{cor: paths in covers}
In the circumstances of Lemma $\ref{lem: paths in covers},$ if the underlying Lagrangian of $\b$ lies within distance $\de>0$ from the zero-section $Q\subset T^*Q$ then we have $|\v\si|\le C\de.$ \qed
\end{cor}

\section{Proof of Theorem \ref{thm: main}}\l{sec: proof}
We begin in \S\ref{sec: proof main i} with the proof of Theorem \ref{thm: main} (i). 
In \S\ref{sec: key lemma} we prove the key lemma for Theorem \ref{thm: main} (ii). For this we do not need any hypothesis on $\pi_1Q.$
In \S\ref{sec: proof abelian} we prove Theorem \ref{thm: main} (ii) when $\pi_1Q$ is virtually abelian.
In \S\ref{sec: finite McLean} we prove a few facts we shall need to deal with finite covers of $Q.$
In \S\ref{sec: proof main iii iv} we prove Theorem \ref{thm: main} (iii),(iv).

\subsection{Proof of Theorem \ref{thm: main} (i)}\l{sec: proof main i}
Let $Y$ be a Calabi--Yau manifold, $Q\subset Y$ a closed connected embedded special Lagrangian, and $X$ a Weinstein neighbourhood of $Q\subset Y$ in the sense of Definition \ref{dfn: br}. We begin by proving
\begin{prop}\l{prop: pi1 triv}
If the fundamental group $\pi_1Q$ is trivial then the Fukaya category $\cF(X)$ is generated by the single object $Q,$ the zero-section.
\end{prop}
\begin{proof}
Let $\b$ be an object of $\cF(X)$ and consider its image under the Yoneda functor $H\cF(X)\to H \Mod CW^{*}(T_q^*Q, T_q^*Q ).$ Applying the homological transfer lemma, we may replace $  CW^{*}(T_q^*Q, T_q^*Q )$ by an $A_\infty$ algebra whose underlying graded vector space is given by its cohomology, which is isomorphic to $H_{-*}(\Om_qQ,\La)$ and hence  is supported in non-positive degrees. It follows (by the degree-filtration argument \cite[Appendix A]{Ab4}) that the Yoneda image of $\b$ is a repeated extension of ungraded modules over the degree $0$ algebra $  HW^{0}(T_q^*Q, T_q^*Q ) \cong H_0(\Om_qQ,\La).$
But $H_0(\Om_qQ,\La)$ is isomorphic to $\La[\pi_1Q]\cong\La$ so the image of $\b$  is a repeated extension of ungraded $\La$-vector spaces. 
As $\b$ is supported on a compact Lagrangian without boundary, the number of these extensions is finite and each ungraded piece is finite-dimensional over $\La.$
Recall now that $Q$ intersects $T_q^*Q$ transversely and exactly at one point $q$ and that $HF^*(T_q^*Q,Q)$ is one-dimensional over $\La.$
Then each ungraded piece of the Yoneda image of $\b$ may be written as $HF^*(T_q^*Q,Q)\otimes V$ for some finite-dimensional $\La$-vector space $V.$ 
As the functor $H\cF(X)\to H \Mod CW^{*}(T_q^*Q, T_q^*Q )$ is fully faithful, the object $\b$ is isomorphic in the cohomology Fukaya category to a repeated extension of objects of the form $Q\ot V.$
This completes the proof.  
\end{proof}

We use an algebraic lemma by the first author \cite[Lemma A.4]{Ab4} who proves the first general results on Arnold's nearby Lagrangian problem \cite{FSS2}. The original statement is over $\Z/2\Z$ but the proof shows that it applies to any field including $\La$:
\begin{lem}\l{lem: pi1 triv}
Let $\cA$ be a $\Z$-graded $\La$-linear triangulated $A_\iy$ category generated by one object $\mathbf{a}\in\cA$  whose endomorphism algebra $\hom^*({\bf a},{\bf a})$ is supported in non-negative degrees.
Then every object $\b\in \cA$ with $\hom^*(\b,\b)$ supported in non-negative degrees
is isomorphic in $H^0\cA$ to ${\bf a}[k]\otimes V$ for some integer $k\in\Z$ and for some finite-dimensional $\La$-vector space $V.$
\qed
\end{lem}

We complete now the proof of Theorem \ref{thm: main} (i).
\begin{proof}[Proof of Theorem {\rm\ref{thm: main} (i)}]
Suppose first that $\pi_1Q$ is trivial. Let $\b\in \obj\cF(X)$ be supported on a closed special Lagrangian $L.$ Then $\b$ is a non-zero object and by Proposition \ref{prop: TY0} the graded group $HF^*(\b,\b)$ is supported in non-negative degrees. Proposition \ref{prop: pi1 triv} and Lemma \ref{lem: pi1 triv} imply then that $\b$ is isomorphic in $H^0\cF(X)$ to some $Q[k]\otimes V,$ $k\in\Z,$ with the trivial bounding cochain.
Since $\b$ is non-zero it follows that $V$ is a non-zero vector space. 
Note also that there are vector space isomorphisms
\[HF^n(\b,\b)\cong HF^n(Q[k]\otimes V,Q[k]\otimes V)\cong HF^n(Q\otimes V,Q \otimes V)\cong HF^n(Q,Q)\otimes \End(V,V),\]
which is non-zero (by Corollary \ref{cor: PD} or because $HF^n(Q,Q)$ is isomorphic to the ordinary cohomology group). 
Hence it follows by Proposition \ref{prop: TY} (ii) that $L=Q.$

If $\pi_1Q$ is not trivial but finite, taking the universal cover $f:T^*P\to T^*Q$ we see that the pre-image $f^*L$ agrees with $P.$ So $L=Q$ as we have to prove.
\end{proof}

\subsection{The Key Lemma}\l{sec: key lemma}
We prove the single degree property of $C^0$ nearby special Lagrangians. 
\begin{lem}\l{lem: single degree}
Let $Y$ be a Calabi--Yau manifold of complex dimension $n,$ $Q\subset Y$ a closed connected embedded special Lagrangian, and $X$ a Weinstein neighbourhood of $Q\subset Y.$ Fix a point $q\in Q$ and a $\Z$-grading of $T_q^*Q$ with phase $\frac{n\pi}2\in\R$ at $q.$ Then there exists a neighbourhood $U$ of $Q\subset X$ such that for every object $\b\in \obj\cF(X)$ whose underlying Lagrangian is special $($of phase $0\in\R)$ and contained in $U,$ the graded vector space $HF^*(T_q^*Q,\b)$ is supported in degree $0.$
\end{lem}
\begin{rem}\l{rem: single degree}
Here the speciality is essential because otherwise $HF^*(T_q^*Q,\b)$ may be supported in two or more degrees; see Example \ref{ex: non-special}.

There is in fact an analytical theorem for those nearby special Lagrangians in a given homology class as in Theorem \ref{thm: main} (iii),(iv). These are close to the zero-section $Q$ as varifolds in the sense of geometric measure theory, so we can use Almgren's theorem (for area-minimizing currents). As explained by De Lellis and Spadaro \cite[Theorem 2.4]{LS} there exists a small set $B\subset Q$ such that outside the fibre over $B$ the nearby special Lagrangians $L$'s are $C^1$ close to the zero-section $Q.$ This implies that the Maslov index of the pair $(T_q^*Q,L)$ is always zero, so that the chain complex $CF^*(T_q^*Q,\b)$ is already supported in degree $0.$

In Theorem \ref{thm: main} (ii) we do not fix the homology class of the nearby special Lagrangians, for which we shall need Theorem \ref{lem: single degree}.
\end{rem}
\begin{proof}[Proof of Lemma \ref{lem: single degree}]
Suppose first that the Calabi--Yau structure near $Q\subset T^*Q$ is real analytic.
For each $k=1,\cdots,n-1$ we define a compactly-supported Hamiltonian perturbation $L$ of $T_q^*Q,$ which is special Lagrangian near $q\in T^*Q.$ 
To start with, take a $\C$-linear isomorphism $T_q(T^*Q)\cong\C^n$ which maps $T_qQ$ to $\R^n\sb\C^n,$ $T_q^*Q$ to $(i\R)^n,$ the symplectic form to $(i/2)\sum_{j=1}^ndz_j\w d\bar{z}_j$ and the holomorphic volume form to $\bw_{j=1}^ndz_j$ where $z_1,\cdots,z_n$ are the coordinates of $\C^n.$
Take $\th_1,\cdots,\th_k\in(-\frac\pi2,0)$ and $\th_{k+1},\cdots,\th_n\in(0,\frac\pi2)$ with
$\th_1+\cdots+\th_n=0$ so that the plane
$\Pi:=\{(e^{i\th_1}x_1,\cdots,e^{i\th_n}x_n)\in\C^n:x_1,\cdots,x_n\in\R\}$
is special Lagrangian with respect to the flat Calabi--Yau structure.
Take a small ball in $\Pi$ about $0$ and embed it into $T^*Q$ near $q,$ mapping $0\in T_q(T^*Q)$ to $q\in T^*Q.$
Denote the image of the ball by the same $\Pi$ for short.
Take a neighbourhood $W$ of $q$ in $Q$ over which $T^*Q$ is a product bundle.
We can suppose that the ball $\Pi$ lies in the product bundle $(T^*Q)|_W.$
Denote by ${\rm pr}:(T^*Q)|_W\to T_q^*Q$ the vertical projection of the product bundle.
Put $U_q:={\rm pr}(\Pi),$ which is a ball in $T_q^*Q$ about the origin. 
Choose a neighbourhood $U\subset T^*Q$ of $Q$ which is so small that $(U\cap T_q^*Q)\Subset U_q;$ here $A\Subset B$ means that the closure of $A$ is contained in $B.$

We perturb $\Pi$ into a special Lagrangian with respect to the Calabi--Yau structure of $U\sb X.$ 
There are two ways of doing this.
The first is to use Cartan--K\"ahler's theorem as Harvey and Lawson do \cite[Chapter III, Proof of Theorem 5.5]{HL}, which we can do because we have supposed that our Calabi--Yau structure is real analytic.
The second is to use the implicit function theorem again as Harvey and Lawson do \cite[Chapter III, Proof of Corollary 2.14]{HL}.
In either case we get a special Lagrangian near $q\in T^*Q.$
We can suppose that this is a graph over $U_q.$
Using a cut-off function we get then a compactly-supported Hamiltonian perturbation $L$ of $T_q^*Q$ such that $L\cap {\rm pr}^*U_q$ is a special Lagrangian in $U.$

Grade the Lagrangians $L$ and $\Pi$ by phase $0$ and the Lagrangians $T_q^*Q$ and $(i\R)^n$ by phase $\frac{n\pi}2$ at $q.$
For $j=1,\dots,k$ define an affine function $[0,1]\to (-\frac\pi2,\th_j)$ by $t\mapsto -\frac\pi2+t(\th_j+\frac\pi2).$ 
For $j=k+1,\dots,n$ define an affine function $[0,1]\to (\frac\pi2,\th_j)$ by $t\mapsto \frac\pi2+t(\th_j-\frac\pi2).$ 
Using these we get a Hamiltonian isotopy from $(i\R)^n$ to the plane $\Pi$ which does not cross $\R^n.$
As $k(-\frac\pi2)+(n-k) \frac\pi2=\frac{n\pi}2-k\pi$ there is a graded Hamiltonian isotopy from $(i\R)^n[k]$ to $\Pi.$
This induces a $\Z$-graded Hamiltonian isotopy from $T_q^*Q[k]$ to $L.$
Suppose now that $\b\in \obj\cF(X)$ is supported on a closed special Lagrangian $N\subset U.$
Since $(U\cap T_q^*Q)\Subset U_q$ it follows then, by the unique continuation principle, that the special Lagrangian part of $L$ does not agree with $N.$
We can thus apply Proposition \ref{prop: TY2} to $N,L$ and consequently $HF^*(L,\b)\cong HF^*(T_q^*Q[k],\b)$ is supported in degrees $1,\cdots,n-1.$
So $HF^*(T_q^*Q,\b)$ is supported in degrees $1-k,\cdots,n-1-k.$ Noting that this holds for every $k\in\{1,\cdots,n-1\}$ and that $\bigcap_{k=1}^{n-1}[1-k,n-1-k]=\{0\},$ it follows that $HF^*(T_q^*Q,\b)$ is supported in degree $0.$

Finally, if the K\"ahler form $\om$ near $Q\subset T^*Q$ is not real analytic we can approximate it by a real analytic one (after taking a K\"ahler potential) and find a compactly-supported Hamiltonian diffeomorphism $\ph: T^*Q\to T^*Q$ such that $\ph_*\om$ is real analytic. Applying to this the result above, we get a corresponding neighbourhood $V$ of $Q\subset T^*Q.$
Put $U:=\ph^*V.$ Let $\b\in \obj\cF(X)$ be supported on a closed special Lagrangian (relative to $\om$) and denote by $\ph\b\in \obj\cF(V)$ the push-forward.
Then $HF^*(T_q^*Q,\ph\b)$ is supported in the single degree $0.$
But as $\ph$ is compactly supported, we have $HF^*(T_q^*Q,\ph\b)\cong HF^*(T_q^*Q,\b),$ which completes the proof.
\end{proof}

The effect of the lemma above is that we do not have to deal with the extension problem which we explain now.
Let $\b$ be an object of $\cF(X).$
As in the proof of Theorem \ref{thm: main} (i) the Yoneda module $HF^*(T_q^*Q,\b)$ is a repeated extension of ungraded modules over $\La[\pi_1Q],$ that is, representations of $\pi_1Q.$
These extensions may be non-trivial.
In Example \ref{ex: non-special}, for instance, the direct sum of $Q$ and the graph of $\al$ are different from their surgery.
More precisely, these define non-isomorphic objects of $H^0\cF(X).$ They are made up of the same representations but with different extensions.

\subsection{Proof of Theorem \ref{thm: main} (ii)}\l{sec: proof abelian}
Let $Y$ be a Calabi--Yau manifold of complex dimension $n,$ complex structure $J,$ K\"ahler form $\om,$ and nowhere-vanishing holomorphic $(n,0)$ form $\Om.$ Let $Q\subset Y$ be a closed connected embedded special Lagrangian with $\pi_1Q$ virtually abelian, and $X$ a Weinstein neighbourhood of $Q\subset Y.$ 

We make several definitions we will use for a version \cite[Proposition 2.13]{J5} of McLean's theorem. Denote by $g$ the metric on $Q$ induced from $(J,\om).$ Define a $C^\iy$ function $\ps:Q\to(0,\iy)$ by
\e\l{ps}\ps^2\frac{\om^n}{n!}=\Bigl(\frac i2\Bigr)^n(-1)^{\frac{n(n-1)}2}\Om\wedge\ov\Om.\e
Denote by $d': C^\iy(T^*Q)\to C^\iy(Q,\R)$ the linear operator $\al\mapsto d^*(\ps \al)$ where $d^*$ is computed with respect to $g.$ Then $dd'+d'd:C^\iy(T^*Q)\to C^\iy(T^*Q)$ is an elliptic operator and we have $C^\iy(T^*Q)=(\ker d\cap\ker d')\oplus\im d\oplus \im d'$
where the direct sums are $L^2$ orthogonal with respect to $g.$ This implies that the de Rham cohomology group $H^1(Q,\R)$ is isomorphic to $\ker d\cap\ker\Ps.$ 

We fix a triangulation $\De$ of $Q$ and use the norm on $C^1(\De,\R)$ introduced before Corollary \ref{cor: est for norm}. The de Rham theorem implies then that there exists $M>0$ such that every $1$-cocycle $\be\in C^1(\De,\R)$ may be represented on $Q$ by some $1$-form $\al\in\ker d\cap\ker d'$ with $\sup_Q|\al|_{C^0}\le M|\be|$ where the pointwise $C^0$ norm is computed with respect to $g.$ The McLean theorem \cite[Proposition 2.13]{J5} implies in turn that there exists $\ep>0$ with the following property: for every $\al\in\ker d\cap\ker d'$ with $\sup_Q|\al|_{C^0}\le\ep$ the graph $Q^\al$ of $\al$ is Hamiltonian equivalent to some special Lagrangian in $X.$ Let $C>0$ be as in Corollary \ref{cor: est for rep}. Denote by $U\sb X$ the disc sub-bundle of radius $\de\in(0,(MC)^{-1}\ep]$ which is so small that we can use Lemma \ref{lem: single degree}.

Let $L\subset X$ be a closed special Lagrangian with $HF^*$ unobstructed and contained in $U.$ We use an algebraically closed Novikov field $\La$ to define the Fukaya category $\cF(X).$ Let $\b\in\obj\cF(X)$ be an object supported on $L.$ Corollary \ref{cor: special non-zero} implies then that $\b$ is a non-zero object. By Lemma \ref{lem: single degree} the graded vector space $HF^*(T_q^*Q,\b)$ is supported in a single degree. Denote by $\r:\pi_1Q\to \GL (HF^*(T_q^*Q,\b))$ the $\La$-linear representation made from the $\fm^2$ products in $H\cW(X).$ 

Suppose first that $\pi_1Q$ is abelian so that $\rho$ has a one-dimensional sub-representation $\si:\pi_1Q\to \La^*.$ Recall that $\v\si\in \hom(\pi_1Q,\R)\cong H^1(Q,\R)$ and denote by $\al$ the closed $1$-form on $Q$ which represents the de Rham class $\v\si$ and satisfies the equation $\Ps\al=0.$ Using the notation of Lemma \ref{lem: smooth deformations} write $\si\cong T^{-[\al]}\otimes E$ where $E$ is some group homomorphism $\pi_1Q\to\GL_1\La^0.$ This defines a rank-one $\La^0$ local system over $Q^\al\cong Q.$ As in Remark \ref{rem: ell inf} the category of rank-one free $\La^0$ modules is equivalent to that of one-dimensional valued vector spaces. Hence we get a filtered local system over $Q^\al$ which we denote by the same $E.$ Lemma \ref{lem: smooth deformations} implies then that $\si:\pi_1Q\to\La^*$ is isomorphic to the representation on $HF^*(E,T_q^*Q).$ As $\si$ is a sub-representation of $\rho$ there exists a degree-zero non-zero morphism from $\si$ to $\rho.$ Thus $HF^0(E,\b)\ne0.$

By Corollary \ref{cor: est for rep} we have $|\v\si|\le \de.$ We have chosen $\de>0$ so that $\sup_Q|\al|_{C^0}\le MC\de\le \ep.$ Then the graph $Q^\al$ is Hamiltonian equivalent to some special Lagrangian $Q'.$ As this Hamiltonian perturbation does not change the isomorphism class of $\b$ it follows from Proposition \ref{prop: TY} (i) that $Q'=L.$

For $\pi_1Q$ virtually abelian we have only to pass to a finite cover of $Q$ which has abelian fundamental group. This completes the proof.
\qed

\subsection{Finite Covers and McLean's Theorem}\l{sec: finite McLean}

We begin by proving a version of the de Rham theorem which is uniform with respect to finite covers.
\begin{prop}\l{prop: de Rham}
Let $(Q,g)$ be a closed connected Riemannian manifold, $\De$ a triangulation of $Q,$ and $f:P\to Q$ a finite cover. Let $\ta\in C^1(f^*\De,\R)$ be a $1$-cocycle of the simplicial complex $(P,f^*\De)$ and put $|\ta|:=\sup_{e\text{ is an edge of }(P,f^*\De)}|\ta(e)|.$ Then for every integer $k>0$ there exists $M>0$ independent of the finite cover $f:P\to Q$ and so large that the de Rham class $[\ta]\in H^1(P,\R)$ may be represented on $P$ by some closed $1$-form with $C^k$ norm $\le M|\ta|$ with respect to the induced metric $f^*g.$ 
\end{prop}
\begin{proof}
We use the $\R$-linear map $C^1(f^*\De,\R)\to \Om^1_P$ from the $\R$-vector space of $1$-cochains to that of $1$-forms, defined by Singer and Thorpe \cite[\S6.2, Lemma 1]{ST}, which maps $1$-cocycles to closed $1$-forms and induces the de Rham isomorphism of the two cohomology groups.
Part (4) of their lemma shows that for each edge $e$ in $P$ the corresponding $1$-form is supported near $e$ in $P.$
The process of assigning the $1$-form is thus local, which we can do essentially in $Q$ before passing to the finite cover $P.$ More precisely, denote by $\{e^*:e\text{ is an edge of }(P,f^*\De)\}\subset C^1(f^*\De,\R)$ the dual basis to $\{e\in C_1(f^*\De,\R):e\text{ is an edge of }(P,f^*\De)\}.$ Write $\ta=\sum_{e\text{ is an edge of }(P,f^*\De)}\ta_e e^*.$ The $1$-form corresponding to $e^*$ is then $C^k$ bounded uniformly with respect to $f:P\to Q.$ On the other hand, we have $|\ta_e|\le |\ta|$ for every $e.$ These imply the estimate we want.
\end{proof}

We prove then that McLean's theorem holds uniformly with respect to finite covers. 
\begin{prop}\l{prop: McLean and finite covers}
Let $Y$ be a Calabi--Yau manifold of complex dimension $n\ge3.$ Let $Q\subset Y$ be a closed connected embedded special Lagrangian, and $X$ a Weinstein neighbourhood of $Q\subset Y.$ Define a $C^\iy$ function $\ps:Q\to(0,\iy)$ by \eq{ps} $($using the K\"ahler form and the nowhere-vanishing holomorphic $(n,0)$ form of $Y).$ Then for every $D>0$ there exists $\ep>0$ such that the following holds: let $f:P\to Q$ be a finite cover of degree $\le D,$ and $\al$ a closed $1$-form on $P$ with $\sup_P|\al|_{C^0}\le \ep$ and $d^*(f^*(\ps \al))=0,$ both computed with respect to the induced metric $f^*g;$ then there exists a $C^\iy$ function $v:P\to\R$ such that the graph of $\al+dv$ is a closed embedded special Lagrangian in the pre-image $f^*X$ $($whose Calabi--Yau structure is induced by $f:T^*P\to T^*Q).$
\end{prop}
\begin{proof}

For an integer $k\ge0$ denote by $L^2_k(T^*P)$ the Sobolev space of $1$-forms on $P$ and by $L^2_k(P,\R)$ that of $0$-forms on $P.$ Denote by $\Xi\subset L^2_k(P,\R)$ the co-dimension one closed vector subspace consisting of $\xi\in L^2_k(P,\R)$ with $\int_P\xi d\mu=0$ where $d\mu$ is the volume form of $(P,f^*g).$ Denote by $\Ps:L^2_{k+2}(P,\R)\to\Xi$ the modified Laplacian $v\mapsto d^*(f^*\ps d v).$ We often regard $f^*\ps:P\to(0,\iy)$ as a multiplication operator and write $\Ps=d^*(f^*\ps) d.$ We prove 

\begin{lem}\l{lem: Green}
There exists $M_0>0$ depending on $D$ but independent of the finite cover $f:P\to Q$ and so large that for every $v\in L^2_2(P,\R)$ we have $\|dv\|_{L^2}\le M_0\|\Ps v\|_{L^2}.$
\end{lem}
\begin{proof}
Gallot \cite[3.1]{Gall} proves the following: let $(P,h)$ be a compact Riemannian manifold of dimension $n\ge3$ and of Ricci curvature $\ge -C\diam(P,h)^{-2}h;$ then there exists a constant $M>0$ which depends only on $n,C$ and is such that for every $v\in L^2_2(P,\R)$ with $\int_P vd\mu=0$ we have
\e\l{Gallot}\|v\|_{L^{2n/{n-2}}}\le M\frac{\diam(P,h)}{[\vol(P,h)]^{1/n}}\|dv\|_{L^2},\e 
where the Sobolev norms and $d\mu$ are computed with respect to $h.$ 
We apply this to $h=f^*g,$ the pull-back by $f:P\to Q.$ As the volume and diameter of $(P,f^*g)$ may be estimated in terms of those of $(Q,g)$ combined with the $\deg f\le D,$ there exists a constant $N>0$ independent of $f$ such that for every $v\in L^2_2(P,\R)$ with $\int_P vd\mu=0$ we have $\|v\|_{L^{2n/{n-2}}}\le N\|dv\|_{L^2}.$

Suppose now that  $v\in L^2_2(P,\R)$ is any and put $c:=\int_Pvd\mu.$ Then
\e\l{Sob0}\begin{split}
(\min_Pf^*\ps)\|dv\|_{L^2}\le (dv,f^*\ps dv)_{L^2}\le(v,d^*(f^*\ps)dv)_{L^2} 
=(v,\Ps v)_{L^2}\\
=(v-c,\Ps v)_{L^2}\le \|v-c\|_{L^{2n/(n-2)}} \|\Ps v\|_{L^{2n/{n+2}}}\le  N\|dv\|_{L^2} \|\Ps v\|_{L^{2n/{n+2}}},
\end{split}\e
where the $L^2$ products and the Sobolev norms are computed with respect to $f^*g.$ Note that $\min_Pf^*\ps=\min_Q\ps=:\de>0$ is independent of $f.$ The estimate \eq{Sob0} implies then
\e\l{Sob00} \|dv\|_{L^2}\le \de^{-1}M\|\Ps v\|_{L^{2n/{n+2}}}.\e
Since the volume of $(P,f^*g)$ is bounded above by a constant depending only on $(Q,g)$ and $D$ it follows that so is $\|1\|_{L^n}.$ H\"older's inequality, on the other hand, implies $\|\Ps v\|_{L^{2n/{n+2}}}\le \|1\|_{L^n}\|\Ps v\|_{L^2}.$ This combined with \eq{Sob00} completes the proof.
\end{proof}
\begin{rem}
If $P$ is a finite cover of $Q$ then the factor $\diam(P,h)[\vol(P,h)]^{-1/n}$ tends to infinity as the covering degree increases.
\end{rem}

Fix now $k>\frac n2$ so that we have the Sobolev embedding $L^2_{k+1}(T^*P)\subset C^1(T^*P).$ We prove that the Sobolev embedding and the elliptic regularity are $f$-uniform too.
\begin{lem}\l{lem: Sob}
There exists $M_1>0$ independent of the finite cover $f:P\to Q$ and so large that for every $u\in L^2_{k+1}(T^*P)$ we have
\e\l{Sob}
\sup_P|u|_{C^1}\le M_1\|u\|_{L^2_{k+1}}\le  (M_1)^2(\|u\|_{L^2}+\|(d+ d^*f^*\ps)u\|_{L^2_k}).\e
\end{lem}
\begin{proof}
Let $\rho>0$ be less than half the injectivity radius of the Riemannian manifold $(Q,g).$ For each point $a\in Q$ denote by $B_a\subset Q$ the open geodesic ball of radius $\rho$ about $a.$
As $Q$ is compact there exists a finite set $A\subset Q$ such that $\bigcup_{a\in A}B_a=Q.$
By the Sobolev embedding on each $B_a,$ $a\in A,$ there exists $M_a>0$ independent of $f:P\to Q$ and so large that for every $u\in L^2_{k+1}(T^*B_a)$ we have
$\sup_{B_a}|u|_{C^1}\le M_a\|u\|_{L^2_{k+1}(B_a)}.$
Note that $M:=\max_{a\in A}M_a$ is also independent of $f:P\to Q.$ 
We look now at the finite cover $f:P\to Q.$ Take any $x\in P$ with $f(x)\in A$ and denote by $B_x\subset P$ the lift of $B_{f(x)}\subset Q$ with respect to $f:P\to Q,$ which exists because $B_{f(x)}$ is simply connected.
Then for any $u\in L^2_{k+1}(T^*P)$ we have
\e\l{loc Sob}|u|_{B_x}|_{C^1(B_x)}=|f_*(u|_{B_x})|_{C^1(B_a)}\le M_a\|f_*(u|_{B_x})\|_{L^2_{k+1}(B_a)}= M_a\|u|_{B_x}\|_{L^2_{k+1}(B_x)}\e
which is bounded above by $M\|u\|_{L^2_{k+1}(P)}.$ Since $P$ is covered by $B_x$'s with $f(x)\in A$ it follows by \eq{loc Sob} that $\|u\|_{L^2_{k+1}(P)}\le M\|u\|_{L^2_{k+1}(P)},$ proving the left inequality of \eq{Sob}. 

We turn now to the other part. We show first that for $a,b\in A$ with $B_a\cap B_b\ne\es$ the non-empty intersection is necessarily connected. If there were two components, there would be two distinct geodesics connecting these two.
But $B_a,B_b$ have been chosen to have the common radius $\rho$ which is less than half the injectivity radius. So $B_a\cup B_b$ lies in a single geodesic ball, which contradicts that the two geodesic exist. Thus $B_a\cap B_b$ is connected. The union $B_a\cup B_b$ is accordingly simply connected.

As $Q$ is covered by $B_a$'s of radius $\rho>0$ there exists $\si\in(0,\rho)$ so close to $\rho$ that if we denote by $B_a'\sb B_a$ the concentric ball of radius $\si$ then $(B_a')_{a\in A}$ still covers $Q.$ 
Then for $u\in L^2_{k+1}(T^*P)$ we have
\e\l{sum}
\|u\|_{L^2_{k+1}}
=\sum_{y\in f^*A} \|f_*(u|_{B_y})\|_{L^2_{k+1}(B_{f(y)})}
\le \sum_{y\in f^*A}\sum_{B_a'\cap B_{f(y)}\ne\es}\|f_*(u|_{B_y})\|_{L^2_{k+1}(B_{f(y)}\cap B_a')}.\e
For $a\in A$ and $x\in f^{-1}(a)$ denote by $Y_{ax}\sb f^*A$ the set of $y$ such that $B_{f(y)}\cap B_a'\ne\es$ and the simply connected set $B_{f(y)}\cap B_a$ lifts to $B_y\cap B_x.$
Exchanging the order of the two sums in \eq{sum} we see then that
\e\l{sumsumsum}
\begin{split}
\|u\|_{L^2_{k+1}}
&=\sum_{a\in A}\sum_{x\in f^{-1}(a)}\sum_{y\in Y_{ax}}\|f_*(u|_{B_y})\|_{L^2_{k+1}(B_{f(y)}\cap B_a')}\\
&=\sum_{a\in A}\sum_{x\in f^{-1}(a)}\sum_{y\in Y_{ax}}\|f_*(u|_{B_y\cup B_x})\|_{L^2_{k+1}(B_{f(y)}\cap B_a')}\\
&\le \sum_{a\in A}\sum_{x\in f^{-1}(a)}(\# Y_{ax})\|f_*(u|_{B_x})\|_{L^2_{k+1}(B_a')}.
\end{split}
\e
Note that there is an injection $Y_{ax}\to\{z\in A: B_a'\cap B_z\ne\es\}$ defined by $y\mapsto f(y)$ so that the number $\#Y_{ax}$ is bounded above by some $f$-independent constant $N>0.$
It follows then from \eq{sumsumsum} that
\e\l{sumsum}
\|u\|_{L^2_{k+1}}\le N\sum_{a\in A}\sum_{x\in f^{-1}(a)}\|f_*(u|_{B_x})\|_{L^2_{k+1}(B_a')}.
\e
On the other hand, by the ordinary a priori estimates over $Q$ there exists $M>0$ independent of $f:P\to Q$ and so large that for every $a\in A$ and $x\in f^{-1}(a)$ we have
\[\|f_*(u|_{B_x'})\|_{L^2_{k+1}(B_a')}\le M(\|f_*(u|_{B_x})\|_{L^2(B_a)}+\|(d+d^*\ps)f_*(u|_{B_x})\|_{L^2_k(B_a)}).\]
Hence it follows by \eq{sumsum} that
\[\begin{split}
\|u\|_{L^2_{k+1}}
&\le MN\sum_{a\in A}\sum_{x\in f^{-1}(a)} (\|f_*(u|_{B_x})\|_{L^2(B_a)}+\|(d+d^*\ps)f_*(u|_{B_x})\|_{L^2_k(B_a)})\\
&\le MN\sum_{a\in A}\sum_{x\in f^{-1}(a)} 
(\|u\|_{L^2(B_x)}+\|(d+d^*f^*\ps)u\|_{L^2_k(B_x)})\\
&= MN(\|u\|_{L^2(P)}+\|(d+d^*f^*\ps)u\|_{L^2_k(P)}).
\end{split}\]
Thus $M_1:=MN$ will do for the right part of \eq{Sob}.
\end{proof}
\begin{rem}
For the proof above we have not used the degree bound $D$ as we did after \eq{Gallot}. This may be justified by recalling that the latter \eq{Gallot} is a global property of Riemannian manifolds. The current Sobolev inequality, which is of the form $|u|_{C^1}\le M\|u\|_{L^2_{k+1}},$ would hardly be such a global property. Interior estimates about elliptic regularity are also of local nature.
\end{rem}

The definition of $\Xi$ implies that $\Ps :L^2_{k+2}(P,\R)\to\Xi$ is surjective, so there exists a right inverse to this, which we denote by $G:\Xi\to L^2_{k+2}(P,\R).$ We prove then that the following holds.
\begin{cor}\l{cor: Green}
There exists $M_2\ge 1$ independent of the finite cover $f:P\to Q$ and so large that for every $\xi\in\Xi$ we have $\|dG\xi\|_{L^2_{k+1}}\le M_2\|\xi\|_{L^2_k}.$
\end{cor}
\begin{proof}
Applying Lemma \ref{lem: Green} to $G\xi\in L^2_{k+2}(P,\R)\subset L^2_2(P,\R)$ in place of $v$ we see that $\|dG\xi\|_{L^2}\le M_0\|d^*\ps dG\xi\|_{L^2}=M_0\|\xi\|_{L^2}.$
This with the right inequality of \eq{Sob} implies
\[\|dG\xi\|_{L^2_{k+1}} \le M_1(\|dG\xi\|_{L^2}+ \|(d+d^*\ps)dG\xi\|_{L^2_k}
 \le M_1(M_0\|\xi\|_{L^2}+ \|\xi\|_{L^2_k}).\]
Thus $M_2:=M_1(M_0+1)$ will do.
\end{proof}

We complete now the proof of Proposition \ref{prop: McLean and finite covers}. We follow the standard proof of McLean's theorem \cite[Proposition 2.13]{J5} using the uniform estimates above. Denote by $\Om$ the nowhere-vanishing holomorphic $(n,0)$ form of $X.$ Let $\ep_0>0$ be so small that if $\be$ is a $1$-form on $P$ with $\sup|\be|\le\ep_0$ then the graph of $\al+\be$ lies in the pre-image $f^*X\subset T^*P.$ Put $V:=\{v\in L^2_{k+2}(P,\R): \|dv\|_{L^2_{k+1}}\le M_1^{-1}\ep_0\}.$ The Sobolev inequality \eq{Sob} implies then that for every $v\in V$ the graph of $\al+dv$ lies in the pre-image $f^*X$ and we can therefore define on $P$ the $n$-form $(\al+dv)^*\Im f^*\Om.$ Define an operator $\Ph:V \to \Xi$ by $-\Ph v\,d\mu:=(\al+dv)^*\Im f^*\Om$ for $v\in V.$ We have $\Ph v\in \Xi$ because $\int_P-\Ph v\, d\mu=\int_P(\al+dv)^*\Im f^*\Om$ depends only on the homology class $(\al+dv)_*[P]=[P]\in H_n(T^*P,\R)$ whose pairing with $[\Im f^*\Om]\in H^n(T^*P,\R)$ vanishes because $P\subset T^*P$ is a special Lagrangian. So $\Ph:V\to \Xi$ is well defined.   

Standard computation shows that $\Ph$ is differentiable at $v=0$ with respect to the Sobolev norms. By hypothesis we have $\Ps\al=0$ and hence we see that the differential $\nb_0\Ph:L^2_{k+2}(P,\R)\to \Xi$ agrees with $\Ps.$ Corollary \ref{cor: Green} implies then that for $\xi\in \Xi$ with $\|\xi\|\le M_2^{-1}M_1^{-1}\ep_0=:\ep_1$ we have $\|dG\xi\|_{L^2_{k+1}} \le M_1^{-1}\ep_0$ so $G\xi\in V.$ We solve the equation $\Ph G\xi=0$ for $\xi\in L^2_k(P,\R)$ with $\|\xi\|\le \ep_1.$ Define $\Th:V\to L^2_k(P,\R)$ by $\Th :=-\Ph +\Ps.$ We seek then a fixed point of the composite $\Th G$ and use therefore the contraction mapping principle. As $\Th$ may be computed pointwise on $P,$ there exists $M>0$ independent of $f$ and so large that the following holds: let $v,v'\in C^1(P,\R)$ satisfy $|dv|,|dv'|<\ep_0$ at every point of $P;$ and put $\be=\al+dv,$ $\be':=\al+dv';$ then
\[|\Th v-\Th v'|_{C^k}\le M|\be-\be'|_{C^{k+1}}(|\be|_{C^{k+1}}+|\be'|_{C^{k+1}})\]
at every point of $P.$ Integrating this and using the Cauchy--Schwarz inequality we get
\e\l{int P} \int_P|\Th \be-\Th \be'|_{C^k}d\mu\le M\sqrt{\int_P|\be-\be'|_{C^{k+1}}^2d\mu}\sqrt{\int_P(|\be|_{C^{k+1}}+|\be'|_{C^{k+1}})^2d\mu}.\e
Putting $A:=|\be|_{C^{k+1}}$ and $B=|\be'|_{C^{k+1}}$ we have 
\[\sqrt{\int_P(A+B)^2d\mu}
\le 
\sqrt{2\int_P(A^2+B^2)d\mu}
\le 
\sqrt2 \left(\sqrt{\int_PA^2d\mu}+\sqrt{\int_PB^2d\mu}\right).\]
This with \eq{int P} implies
\[\|\Th v-\Th v'\|_{L^2_{k+1}}\le \sqrt 2 M\|\be-\be'\|_{L^2_k}(\|\be\|_{L^2_k}+\|\be'\|_{L^2_k})\le 2\sqrt 2 M(\|\al\|_{L^2_k}+\ep_0)\|v-v'\|_{L^2_k}.\]
Hence choosing $\ep_2\in(0,\min\{(3\sqrt 2 M(\|\al\|_{L^2_k}+\ep_0))^{-1},\ep_1\})$ and using Corollary \ref{cor: Green} we see that $\Th G$ is a contraction map from the closed ball $\{\xi\in \Xi:\|\xi\|_{L^2_k}\le \ep_2\}$ to itself. So there exists a fixed point $\xi\in \Xi$ of the operator $\Th G,$ and $\xi=\Th G\xi=-\Ph G\xi+\Ps G\xi=-\Ph G\xi+\xi.$ Thus $\Ph G\xi=0.$ Now $\be:=\al+G\xi$ satisfies the elliptic system $d\be=\Ph\be=0$ and is therefore $C^\iy.$ The graph of $\be$ is thus the special Lagrangian we want.
\end{proof}

Here is a corollary of Proposition \ref{prop: de Rham} and the estimates in the proof above.
\begin{cor}\l{cor: de Rham}
Fix a constant $D>0$ and suppose that in the circumstances of Proposition $\ref{prop: de Rham}$ the finite cover $f:P\to Q$ has degree $\le D.$ Then there exists $M>0$ independent of $f:P\to Q$ and so large that the de Rham class $[\ta]\in H^1(P,\R)$ may be represented on $P$ by some closed $1$-form $\al$ with $C^1$ norm $\le M|\tau|$ and $d^*(f^*\ps \al)=0$ where the $C^1$ norm and $d^*$ are computed with respect to the induced metric $f^*g.$
\end{cor}
\begin{proof}
Fix an integer $k>\frac n2.$ Then by Proposition \ref{prop: de Rham} there exist a constant $N_0>0$ independent of $f$ and a closed smooth $1$-form $\be$ on $P$ with $\|\be\|_{L^2_{k+1}}\le N_0|\tau|.$
Define a $1$-form $\al$ on $P$ by 
$\al:= \be- d G(d^*\ps\be).$
Then $d \al=0$ and $d^*\ps\al=d^*\ps\be-\Ps G(d^*\ps\be)=0.$
We estimate the $C^1$ norm of $\al.$
As $G:\Xi\to L^2_{k+2}(P,\R)$ is again uniformly bounded, there exist constants $N_1,N_2>0$ independent of $f,\be$ and so large that 
\[ \|G(d^*\ps\be)\|_{L^2_{k+2}} \le N_1 \|d^*\ps\be\|_{L^2_k}=N_1N_2\|\be\|_{L^2_{k+1}}\le N_0N_1N_2|\tau|. \]
By \eq{Sob} there exists a constant $N_3$ independent of $f,\be$ and so large that $\sup_P|\be|_{C^1}\le N_3 \|\be\|_{L^2_{k+1}}\le N_0N_3|\tau|$ and
\[ \ts \sup_P|d G(d^*\ps\be)|_{C^1}\le  N_3\|d G(d^*\ps\be)\|_{L^2_{k+1}}\le
N_3\|G(d^*\ps\be)\|_{L^2_{k+2}}\le  N_0N_1N_2N_3|\tau|.\]
Thus $|\al|_{C^1}\le |\be|_{C^1}+ 
|d G(d^*\ps\be)|_{C^1}\le N_0N_3(1+N_1N_2)|\tau|,$ which completes the proof. 
\end{proof}

\subsection{Proof of Theorem \ref{thm: main} (iii),(iv)}\l{sec: proof main iii iv}
We prove Theorem \ref{thm: main} (iii) now. For $Q$ of dimension $n\le2$ the hypothesis that $\pi_1Q$ should be virtually solvable implies that $\pi_1Q$ virtually abelian. So we can use the result of (ii).  

Suppose therefore that $Q$ has dimension $n\ge3$ so that we can use the results of \ref{sec: finite McLean}. Suppose also that $\pi_1Q$ is solvable at first. We begin by choosing a disc sub-bundle $U\subset T^*Q.$ We impose two conditions. Firstly, $U$ should be so small that we can use Lemma \ref{lem: single degree}. Secondly, let $C>0$ be as in Lemma \ref{lem: est for rep} or Corollary \ref{cor: paths in covers}; $M>0$ as in Corollary \ref{cor: de Rham}; and $\ep>0$ as in Proposition \ref{prop: McLean and finite covers}. Then the radius $\de$ of $U$ should be $\le (MC)^{-1}\ep.$

Let $L\subset X$ be a closed special Lagrangian with $HF^*$ unobstructed and contained in $U.$ We use again an algebraically closed Novikov field $\La$ to define the Fukaya category $\cF(X).$ Let $\b\in\obj\cF(X)$ be an object supported on $L.$
Then $\b$ is a non-zero object. By Lemma \ref{lem: single degree} the graded vector space $HF^*(T_q^*Q,\b)$ is supported in a single degree. Denote by $\r:\pi_1Q\to \GL HF^*(T_q^*Q,\b)$ the representation made from the $\fm^2$ products in $H\cW(X).$
By hypothesis the group $\pi_1Q$ is solvable; and accordingly, so is its image $\r(\pi_1Q)\subset \GL HF^*(T_q^*Q,\b).$ We apply to this a version with estimate of the Lie--Kolchin theorem.
\begin{thm}[Lie--Kolchin]\l{thm: LC}
There exists a function $F:\{1,2,3,\dots\}\to \{1,2,3,\dots\}$ such that for $m\in\{1,2,3,\dots\}$ and for $\La$ an algebraically closed field (which need not be the Novikov field) every solvable subgroup of $\GL_m(\La)$ has a finite-index triangularizable subgroup of index $\le F(m).$
\end{thm}

By hypothesis the dimension of $V:=HF^*(T_q^*Q,\b)$ is bounded above by a given constant, say $m.$ Consequently there exists a finite-index subgroup of $\pi_1Q$ of index $\le \max\{F(1),\dots,F(m)\}$ and on which $\r$ is triangularisable. Let $f:(P,p)\to(Q,q)$ be the corresponding finite cover,
and take a one-dimensional sub-representation $\si:\pi_1(P,p)\to \La^*$ of the triangularisable group. 

Recall that $\v\si\in \hom(\pi_1P,\R)\cong H^1(P,\R)$ and denote by $\al$ the closed $1$-form on $P$ that represents the cohomology class $\v\si$ and satisfies the equation $d^*(f^*\ps d\al)=0.$ Denote by $P^\al\sb T^*P$ the graph of $\al.$
Using the notation of Lemma \ref{lem: smooth deformations} write $\si\cong T^{-[\al]}\otimes E$ where $E$ is some group homomorphism $\pi_1P\to\GL_1\La^0.$ This defines a rank-one $\La^0$ local system over $P^\al.$ Hence we get a filtered local system over $P^\al$ which we denote by the same $E.$
Lemma \ref{lem: smooth deformations} implies then that $\si:\pi_1P\to\La^*$ is isomorphic to the representation on $HF^*(T_p^*P,E).$
Denote by $f^*\rho$ the restriction to $\pi_1P$ of $\rho:\pi_1Q\to \GL HF^*(T_q^*Q,\b).$
As $\si$ is a sub-representation of $f^*\rho$ there exists a degree-zero non-zero morphism from $\si$ to $f^*\rho.$ Hence it follows that there exists a degree-zero non-zero morphism from $f^*\b$ to $E.$ Thus $HF^0(f^*\b,E)\ne0.$

Choose now a triangulation $\De$ of $Q$ so that we can define the functor $\v\si:\Pi_1(P,f^*\De)\to\R.$ By Corollary \ref{cor: paths in covers} we have $|\v\si|\le C\de.$ Lift $\v\si$ to a $1$-cocycle $\ta$ of $(P,f^*\De)$ with $\R$ coefficients. Using the notation of Proposition \ref{prop: de Rham} we have then $|\ta|\le C\de.$ We can therefore apply Corollary \ref{cor: de Rham} to the cohomology class $[\ta]\in H^1(P,\R),$ which may now be represented on $P$ by some closed $1$-form $\al$ with  $d^*(\ps d\al)=0$ and $\sup_P|\al|_{C^0}\le MC\de\le \ep.$
Proposition \ref{prop: McLean and finite covers} implies then that the graph of $\al$ is Hamiltonian equivalent to some special Lagrangian $P'.$ As this Hamiltonian perturbation does not change the isomorphism class of $f^*\b$ it follows from Proposition \ref{prop: TY} (i) that $P'$ is an irreducible component of $f^*L.$
Thus $L$ is unbranched as we want to prove.

In the general case in which $\pi_1Q$ is virtually solvable, all we have to do is to pass to a finite cover of $Q$ whose fundamental group is solvable. This completes the proof of Theorem \ref{thm: main} (iii). 

We turn now to the proof of Theorem \ref{thm: main} (iv). We follow the proof above till we use Theorem \ref{thm: LC}. We give then a new argument. By hypothesis the group $\pi_1Q$ has no non-abelian free subgroup; and accordingly, its image $\r(\pi_1Q)\subset \GL HF^*(T_q^*Q,\b)$ has no such subgroups either. We apply to this the following version of Tits' theorem \cite[Corollary 1]{Tits}.

\begin{thm}\l{thm: Tits}
For $m\in\{1,2,3,\dots\}$ and for $\La$ any field (of any characteristic) every finitely generated subgroup of $\GL_m(\La)$ that has non-abelian free subgroups has a finite-index solvable subgroup.
\end{thm}
\begin{rem}
This will be the most convenient to us because we are concerned only with finitely generated groups. The paper \cite{Tits} deals with those linear groups which need not be finitely generated.
\end{rem}

Theorem \ref{thm: Tits} and a few other known facts imply
\begin{cor}\l{cor: Tits}
There exists a function $G:\{1,2,3,\dots\}\to \{1,2,3,\dots\}$ such that for $m\in\{1,2,3,\dots\}$ and for $\La$ a field of characteristic $0$ every finitely generated subgroup of $\GL_m(\La)$ that has non-abelian free subgroups has a finite-index solvable subgroup of index $\le G(m).$
\end{cor}
\begin{proof}
We consult Wehrfritz \cite{Wehr0,Wehr} without tracing back the original papers of older results.
It is known \cite[Corollary 3.8]{Wehr} that every linear group $\Ga\le \GL_m(\La)$ has a unique normal solvable subgroup $\Si.$ Let $\Ga$ be finitely generated as above. It is then known \cite[Lemma 2]{Wehr0} that the quotient group $\Ga/\Si$ is a {\it finite} subgroup of $\GL_\mu(\La)$ for some $\mu\in\{1,2,3,\dots\}.$ This $\mu$ comes from other results \cite[Theorems 6.3 and 6.4]{Wehr} which show indeed that $\mu\le (m!)^2.$ There is an old theorem of Jordan \cite[Theorem 9.2]{Wehr} which says that every finite subgroup of $\GL_\nu(\La)$ has a finite-index abelian subgroup of index $\le G'(\nu)$ where $G'(\nu)$ depends only on $\nu.$ So there exists a finite-index abelian subgroup $A\le\Ga/\Si$ of index $\le G(m)$ where $G(m)$ depends only on $m.$ Denote by $\Th\le\Ga$ the pull-back of $A$ by the projection $\Ga\to\Ga/\Si.$ There is then an exact sequence $\Th\cap \Si\to \Th\to A,$ which implies that $\Th$ is solvable by abelian and accordingly solvable. Also $\Th\le\Ga$ is a finite-index subgroup of index less than that of $A\le \GL_\mu(\La),$ which is $\le G(m)$ as we have to prove.
\end{proof}

Recall from hypothesis that the dimension of $V:=HF^*(T_q^*Q,\b)$ is bounded above by a given constant $m.$ Corollary \ref{cor: Tits} and Theorem \ref{thm: LC} imply then that there exists a finite-index subgroup of $\pi_1Q$ on which the representation $\r:\pi_1Q\to \GL V$ is triangularisable and such that the index of the subgroup is bounded above by a constant depending only upon $m.$ The rest of the proof is the same as that for Theorem \ref{thm: main} (iii). We have thus proved Theorem \ref{thm: main} (iv).
\qed

\subsection{Versions of Theorem \ref{thm: main} (iii), (iv)}\l{sec: discuss}
We discuss first the difference between (iii) and (iv) in Theorem \ref{thm: main}. In the proof of (iv), if we did not suppose that $\La$ has characteristic $0$ then there would be a problem in the proof of Corollary \ref{cor: Tits}. Jordan's theorem in fact fails for $\La$ of characteristic $p>0.$ Take for instance $\La$ to be the algebraic closure of the finite field $\F_p:=\Z/p\Z$ and take $G:=\SL_m(\La).$ There is then a sequence $\{G_k:=\SL_m(\F_{p^k})\}_{k=1}^\iy$ of subgroups of $G;$ and as is well known, their normal subgroups will lie in the centre of $G$ except for $(p,m,k)=(2,2,1)$ or $(3,2,1).$ So the indices of them are nearly the order of $G_k$ which tend to infinity as $k$ increases. 

We make a positive characteristic version of Theorem \ref{thm: main} (iv), supposing that there exist infinite {\it finitely presented} torsion groups, although no such examples are known. We need these to be finitely presentable because they are $\pi_1Q$ with $Q$ a compact manifold.

\begin{prop}\l{prop: p version}
Let $Y,Q,X$ be as in Theorem \ref{thm: main}, and $p>0$ a prime number. Let $\pi_1Q$ be an infinite torsion group with no elements of order $p.$ Then for every $R>0$ there exists a neighbourhood $U\sb X$ of $Q$ such that the following holds: let $L\subset X$ be a closed irreducibly immersed special Lagrangian of degree $\le R,$ contained in $U,$ and which has $HF^*$ unobstructed with respect to some filtered local system of rank $\le R$ over some Novikov field of characteristic $p;$ then $L$ is unbranched.
\end{prop}
\begin{proof}
Recall that Theorem \ref{thm: Tits} holds in characteristic $p>0.$ We make a characteristic-$p$ version of Corollary \ref{cor: Tits}: {\it there exists a function $G:\{1,2,3,\dots\}\to \{1,2,3,\dots\}$ such that for $m\in\{1,2,3,\dots\},$ for $p>0$ a prime and for $\La$ a field of characteristic $p$ every subgroup of $\GL_m(\La)$ that has non-abelain free subgroups and has no elements of order $p$ has a finite-index solvable subgroup of index $\le G(m).$} As we have mentioned above, we have only to be careful about the last step to the proof of Corollary \ref{cor: Tits}. We replace Jordan's theorem \cite[Theorem 9.2]{Wehr} by Schur's theorem \cite[Corollary 9.4]{Wehr}: every finite subgroup of $\GL_\nu(\La)$ with {\it no elements of order $p$} has a finite-index abelian subgroup of index $\le G'(\nu)$ where $G'(\nu)$ depends only on $\nu.$ Once this has been done we can follow the proof of Theorem \ref{thm: main} (iv) to show that Proposition \ref{prop: p version} holds.
\end{proof}

We make another version of Theorem \ref{thm: main} (iii),(iv) which is closer to Theorem \ref{thm: main} (ii), not referring to the constant $R,$ but making again a strong hypothesis on $\pi_1Q.$ 

\begin{prop}\l{prop: strong pi1}
Let $Y,Q,X$ be as in Theorem \ref{thm: main}. Let $\pi_1Q$ have no non-abelian free subgroups and satisfy the property $\tau.$ Then there exists a neighbourhood $U\sb X$ of $Q$ such that there exists a neighbourhood $U\sb X$ of $Q$ such that every closed irreducibly immersed special Lagrangian $L\subset X$ with $HF^*$ unobstructed and contained in $U$ is unbranched. 
\end{prop}
\begin{proof}
The proof is in fact easier. The property $\tau$ implies indeed a uniform eigenvalue estimate \cite[Proposition 2.5]{LZ} for the Laplacians over finite covers of $Q.$ More precisely, as the linear operator $\Ps$ is a modified Laplacian we shall need the following computation: in the circumstances of Proposition \ref{prop: McLean and finite covers} introduce a conformal change $g':=\ps^{2/{n-2}}g$ of the Riemannian metric $g$ on $Q;$ then $\Ps=\ps^{2n/{n-2}}\De'$ where $\De'$ is the {\it ordinary} Laplacian of $(Q,g').$ Proposition \ref{prop: McLean and finite covers} will then hold automatically; and moreover, we shall not need the index estimates as in the proof of Theorem \ref{thm: main} (iii),(iv).
\end{proof}

There is however a problem: no known examples of strictly infinite, finitely presentable groups satisfy the hypotheses above. Here we are interested in strictly infinite groups because of Theorem \ref{thm: main} (i), in which we have proved already a stronger result for $\pi_1Q$ finite.

At least $3$-manifold groups do not satisfy all the conditions above. Every $3$-manifold group is always residually finite; and as we shall see in Lemma \ref{lem: 3-manifold} those $3$-manifold groups which have no non-abelian free subgroups are amenable. These two contradict the property $\tau$ \cite[Proposition 1.32]{LZ} (see also Brooks \cite{Br} and Sunada \cite{Sunada}).
  
There are examples \cite{Iv,LM,OS} of infinite, finitely presentable non-amenable groups without non-abelian free subgroups (counterexamples to von Neumann's problem) which however do not satisfy the property $\tau$ (because they have surjections onto $\Z$). There are, on the other hand, infinite torsion groups \cite{Ersh} with the stronger property T, called Golod--Shaferevich groups, and $C^*$ simple groups \cite{OO} with the property T which has no non-abelian free groups; but these would hardly be finitely presentable.

\subsection{Bad Immersions}\l{sec: bad immersions}
Theorem \ref{thm: main} holds for some badly immersed special Lagrangians. 
In the known construction of Fukaya categories, the underlying Lagrangians are required at least to be cleanly immersed \cite{Fuk}.
On the other hand, in our nearby problem we want to include as many (special) Lagrangians as possible, so we introduce the following notion \cite[before Corollary 4.6]{I}.

Let $\hat L$ be a closed manifolds and $\io:\hat L\to X$ be a Lagrangian immersion. Put $L:=\io(\hat L).$
We say that an object $\b$ in $H\cF(X)$ is supported {\it near} $L$ if there exist arbitrarily-small Lagrangian perturbations of $L$ which underlie some objects in the same isomorphism class $[\b]$ in $H^0\cF(X).$ 
More concretely, this means that for every $\ep>0$ there exists a perturbation of the immersion $\hat L\to X$ with $C^k$ norm bounded by $\epsilon$ for all $k$, and which supports an object representing the given isomorphism class $[\b].$

Theorem \ref{thm: main} holds for a special Lagrangian $L$ which admits such an isomorphism class $[\b]$ as above. 
The reason is that the results of the paper \cite{I} are stated and proved for objects of $\cF(X)$ supported near a special Lagrangian. In particular, Proposition \ref{prop: TY} and Lemma \ref{lem: single degree} are still valid for these objects. So the proof of Theorem \ref{thm: main} will work. 

But the condition above on $[\b]$ is rather strong. Here is one possible way of verifying it.
Suppose given a smooth family $\io^t:\hat L\to X,$ $t\in[0,1],$ such that $\io^0=\io$ and the other $\io^t,$ $t\in(0,1],$ are all generic and mutually isotopic under {\it domain} Hamiltonians on $\hat L$ in the sense of Akaho--Joyce \cite[Definition 13.14 (ii)]{AJ} (this is a family version of the classical transversality condition in differential topology).
If we could lift such an isotopy to an ambient Hamiltonian isotopy, then it would follow that any isomorphism class in $H^0\cF(X)$ supported on the immersion $\io^1$ is also supported on $\io^t$ for all $t,$ generalising the Hamiltonian invariance theoreom by Fukaya, Oh, Ohta and Ono \cite{FOOO}. In general, such a lift does not exist because then those strips which bound arcs between self-intersection points will have constant area under the isotopy. We can perhaps use the methods of \cite{PW} to prove invariance statements for this perturbation problem.

\section{Tsai--Wang's Method}\l{sec: analytical}
We recall now Tsai--Wang's method and explain then the relevant examples of $C^0$ nearby special Lagrangians. In this section we do not use Floer cohomology groups or Fukaya categories. 

We begin by recalling the key formula from geometric measure theory. Let $Y$ be a Riemannian manifold. Given a vector field $v$ on $Y,$ a point $y\in Y$, and an $\R$-vector subspace $S\sb T_yY$ we denote by $\di_Sv$ the trace of the map $S\sb T_yY\to T_yY\to S$ where the two arrows are the covariant derivative $\nb v:T_yY\to T_yY$ and the orthogonal projection $T_yY\to S.$ In particular, if $v$ is the gradient vector field of a smooth function $f:Y\to\R$ then
$\di_S\grad f=\tr_S\nb^2f.$
We use also the notion of {\it rectifiable integral varifolds} in $Y$ which we treat, following Simon \cite{Sm}, as Radon measures on $Y.$
Recall that if $V$ is a compactly supported stationary integral varifold in $Y$ then for every vector field $v$ on $Y$ we have $\int_Y\di_{TV}v\,dV=0.$
In particular, 
\e\l{FV}\int_Y\tr_{TV}\nb^2f\,dV=0\e
for every smooth function $f:Y\to\R.$

We recall now the following result from the literature \cite{LotayS,TW0,TW}. We prove this to show how to use \eq{FV}.
\begin{prop}\l{TW}
Let $Y$ be a Ricci-flat Calabi--Yau manifold of complex dimension $n,$ and $Q\subset Y$ a closed embedded special Lagrangian which has positive Ricci curvature with respect to the induced metric from $Y.$
Then there exists a neighbourhood $U\sb Y$ of $Q$ in which every compactly supported stationary integral $n$-varifold is supported on $Q.$
\end{prop}
\begin{proof}
Tsai and Wang prove that $Q\subset Y$ is {\it strongly stable} in their sense \cite[Proposition A(iii)]{TW} and that the following holds \cite[Proposition 4.1]{TW}: there exists an open neighbourhood $U\sb Y$ of $Q$ such that for every point $y\in U$ and for every $n$-dimensional plane $S\sb T_xU$ we have at $y$
\e\l{SS}\tr_S\nb^2d_Q^2\ge c\,d_Q^2\e
where $d_Q:U\to[0,\iy)$ is the distance from $Q$ and $c>0$ depends only upon $Q.$
Applying \eq{FV} with $f=d_Q^2$ and using \eq{SS} we find $d_Q^2=0$ on the support of the varifold $V,$ completing the proof.
\end{proof}
\begin{rem}
The idea of the proof is that the negative gradient flow of $d_Q^2$ decreases the area of $V.$ The computation above is more convenient for the rigorous treatment.
\end{rem}
As is well known, closed special Lagrangian currents are area-minimizing and in particular area-stationary \cite{HL}. Combining this with Proposition \ref{TW} we see immediately that the following holds:
\begin{cor}\l{cor: TW}
In the circumstances of Proposition \ref{TW} every compactly supported closed special Lagrangian current in $U$ is necessarily non-singular, of phase $0\in\R/2\pi\Z$ and supported on $Q.$
\qed
\end{cor}
\begin{exam}
Let $Y$ be either $T^*S^n,$ for $n>1,$ or $T^*\CP^{n}$ for arbitrary $n$, and let $Q$ denote its zero-section; then $Y$ has a Ricci-flat metric constructed respectively by Stenzel \cite{St} and Calabi \cite{Cal}. In these two cases the hypothesis of Proposition \ref{TW} holds and moreover its conclusion holds in the whole $Y$ rather than in a neighbourhood of $Q.$
\end{exam}

We turn now to the example which shows that $\pi_1Q$ strictly abelian is essential to the last part of Theorem \ref{thm: main} (ii).
We begin by proving a lemma in slightly more general circumstances.
\begin{lem}\l{lem: flat}
Let $n,k\ge0$ be integers. Take $Y:=T^n\t\R^k,$ $Q:=T^n\t\{0\}$ and give these the flat metrics.  
Then every compactly supported stationary integral $n$-varifold in $Y$ is parallel to $Q,$ that is, supported on $T^n\t A$ for some finite set $A\sb \R^k.$
\end{lem}
\begin{proof}
Let $V$ be a compactly supported stationary integral $n$-varifold in $Y.$
Applying \eq{FV} again with $v=\grad d_Q^2$ we get
$\int_Y\tr_{TV}\nb^2d_Q^2\,dV=0.$
But $\nb^2d_Q$ is pointwise non-negative and degenerates only in the direction of $T^n\t\{0\}$ so $V$-almost everywhere in $Y$ we have
$\tr_{TV}\nb^2d_Q^2=0$ and $TV$ is parallel to $T^n\t\{0\}.$

We recall a lemma of Simon \cite[Lemma 19.5]{Sm}.
Let $x$ be a point on the support of $V$ and we work near this so we can regard $x$ as a point of the universal cover $\R^n\times\R^k$ of $T^n\times\R^k.$
Let $\ep>0$ be so small that for every point $y$ of distance $<\ep$ from $x$ we can define $y-x\in\R^n\t\R^k.$
Fix $\de\in(0,1)$ and let $y$ be a point on the support of $V$ of distance $\ep\de/4$ from $x$ and such that $y-x$ is {\it not} parallel to $\R^n\t\{0\}.$
Simon \cite[Lemma 19.5]{Sm} proves then that 
\e\l{pq}\Th_V(x)+\Th_V(y)\le(1-\de)^{-n}\ep^{-n}\om_n^{-1}V(B_\ep),\e
where $\om_n$ is the Euclidean volume of the unit ball in $\R^n,$ $\Th_V$ the density of $V,$ and $B_\ep$ the ball of radius $\ep$ about $x.$
The original formula has in fact another term on the right-hand side, which contains the integral over $z\in B_\ep$ of $\|p_{\R^n\times\{0\}}-p_z\|$ where $p_{\R^n\times\{0\}},p_z$ are the projections of $\R^n\times\R^k$ onto $\R^n\times\{0\},T_zV$ respectively. But $T_zV$ is parallel to $\R^n\t\{0\}$ so the additional term vanishes and \eq{pq} holds.  

As $\ep,\de$ tend to $0$ the right-hand side of \eq{pq} tends to $\Th_V(x)$ and in particular becomes less than $\Th_V(x)+1.$ But then by \eq{pq} we have $\Th_V(y)<1,$ which is impossible. So the hypotheses above fail; that is, if $y$ is a point of the support of $V$ close enough to $x$ then $y-x$ is parallel to $\R^n\t\{0\}.$ This completes the proof. 
\end{proof}
\begin{rem}
The hypothesis of Proposition \ref{TW} is stable under small perturbations of the Ricci-flat K\"ahler metric (just because positivity is an open condition). On the other hand, Lemma \ref{lem: flat} will fail under perturbations of the ambient metric.
For instance, Kapouleas and Yang \cite{KY} construct a sequence of minimal surfaces converging to the flat Clifford torus $T^2\subset S^3$ with multiplicity two;
and a suitable rescaling of the ambient metric converges to the flat metric on the tubular neighbourhood of $T^2\subset S^3.$
\end{rem}
\begin{cor}\l{flat}
Let $Y$ be a flat Calabi--Yau manifold, $Q\subset Y$ a closed embedded totally geodesic special Lagrangian, and $U\sb Y$ a tubular neighbourhood of $Q.$
Then every compactly-supported closed special Lagrangian current in $U$ is parallel to $Q,$ that is, a locally constant multi-valued graph over $Q.$ Every such graph is a disjoint union of single-valued graphs if and only if $\pi_1Q$ is abelian.
\end{cor}
\begin{proof}
Since $Q$ is a closed flat manifold we get according to Bieberbach a finite Galois cover $f:T^n\to Q$
such that $\pi_1T^n$ is a maximal abelian subgroup of $\pi_1Q.$
Denote again by the same $f$ the finite cover $T^*T^n\to T^*Q$ between the cotangent bundles.
Identify $U$ with a neighbourhood of the zero-section $Q\subset T^*Q,$ and the pull-back $f^*U$ with a neighbourhood of $T^n\subset T^*T^n.$
As $T^*T^n$ is a product bundle we can use Lemma \ref{lem: flat} to show that the pull-back of a compactly supported closed special Lagrangian current $L$ is parallel; and accordingly, so is the original $L.$
Now if $\pi_1Q$ is abelian, this is the maximal abelian subgroup of itself and $Q=T^n;$ so every special Lagrangian is a disjoint union of single-valued graphs.
If $\pi_1Q$ is non-abelian, there exist a non-trivial covering transformation $\ga:T^n\to T^n$ and accordingly an element $\ga\in\pi_1Q\-\pi_1T^n$ which acts non-trivially upon $H^1(T^n,\R).$ So there exists a constant $1$-form $\al$ on $T^n$ that does change under $\ga:T^n\to T^n.$
By rescaling $\al$ we can suppose that $\al$ is so small that we can define its graph in $f^*U.$ The 
 push-forward of this defines then a truly multi-valued graph in $U.$
\end{proof}

\begin{exam}
Let $Y=U=T^*Q$ with $Q$ a Klein bottle. Write $Q$ as the mapping torus of 
the automorphism $S^1\to S^1$ given by $x\mapsto-x$ modulo $\Z$ for $x\in S^1=\R/\Z.$
Then there are corresponding double covers $T^2\to Q$ and $T^*T^2\to T^*Q,$ under which the graph of $dx$ is pushed forward to a multi-valued graph. Strictly speaking, $Y$ is not Calabi--Yau and $Q$ is not even orientable.
But $Q$ is still a special Lagrangian in the extended sense as in Definition \ref{dfn:Lag Z-grading}; see also Remark \ref{rem: Lag grading}. The multi-valued graph is also a special Lagrangian in the same sense. 
It is easy to find multi-valued graphs which are ordinary special Lagrangians. For instance, take $Q$ to be the mapping torus of the {\it orientation-preserving} diffeomorphism
$T^2\to T^2$ given by $(x,y)\mapsto(-x,-y)$ where $(x,y)\in T^2=\R^2/\Z^2;$ then the graph of $dx$ is pushed forward to a multi-valued graph which is an ordinary special Lagrangian.
\end{exam}

\section{Branched Examples}\l{sec: branch}
Here are examples of {\it branched} $C^0$ nearby closed Lagrangians, which will all be Lagrangian surgeries. We begin by recalling the result of Fukaya, Oh, Ohta and Ono \cite[Chapter 10]{FOOO} which describes the effect of applying Lagrangian surgeries to pairs of objects in the Fukaya category.
\begin{thm}[cf.\ {\cite[Chapter 10]{FOOO}}]\l{FOOO10}
Let $X$ be a compact $\Z$-graded symplectic manifold with contact boundary and $L,L'\subset X$ two closed embedded $\Z$-graded Lagrangians which intersect each other transversely at least at one point. Let the Maslov index function $\mu_{LL'}:L\t _XL'\to \Z$ be identically equal to $1$ so that we can define a $\Z$-graded Lagrangian surgery $L\#L'$ {\rm\cite[Lemma 2.14]{Seid1}.} Then for any $\b,\b'\in \obj\cF(X)$ supported on $L,L'$ respectively, there exist an object $\b''\in\obj\cF(X)$ supported on $L\#L'$ and for every other ${\bf a}\in \obj\cF(X)$ a long exact sequence
\e\l{FOOOex}
\cdots\to HF^*({\bf a},\b')\to HF^*({\bf a},\b'')\to HF^*({\bf a},\b)\to\cdots\e
where the connecting homomorphism $HF^*({\bf a},\b)\to HF^{*+1}({\bf a},\b')$ is the right multiplication by some element of $HF^1(\b,\b').$
\qed
\end{thm}
\begin{rem}
We allow $L,L'$ to have two or more intersection points, which is possible as Fukaya, Oh, Ohta and Ono mention \cite[after Theorem 55.7]{FOOO} although they deal only with single intersection points in their formal treatment. So there are many objects $\b''$ of the form above according to different scales of surgeries at different intersection points.

Here are remarks on the proof. Fukaya, Oh, Ohta and Ono use the cobordism method for the moduli spaces of holomorphic discs with boundary on the surgery; the cobordism is associated to stretching the neck (degenerating the almost complex structure) near the intersection points.
The method of Biran--Cornea \cite{BC} provides an easier proof which does not rely on a degeneration argument: one can associate to the surgery a  Lagrangian cobordism in $X \times \C$, and the count of holomorphic triangles with boundary on this cobordism yields the desired long exact sequence. Although their results are stated for monotone or exact Lagrangians, these assumptions are imposed only to avoid using the theory of virtual chains in case the Lagrangians bound holomorphic discs. We shall in practice take $L,L'$ to be graphical Lagrangians in a cotangent bundle, which do not bound any non-constant holomorphic discs; so their proof can be implemented with no difficulty.
\end{rem}

For special Lagrangians we have
\begin{cor}\l{cor: non-special}
In the circumstances of Theorem \ref{FOOO10} let $X$ be a Calabi--Yau manifold and $L,L'$ both special Lagrangians. Then $\b''\in \obj \cF(X)$ is not isomorphic in $H^0\cF(X)$ to any other object supported on a closed special Lagrangian of any phase.
\end{cor}
\begin{proof}
By shifting $L,L'$ we may suppose that $L$ has phase $2k \pi$ for some $k\in\Z$ and that $L'$ has phase $0.$
Then by \eq{mu} we have $-k\in(1-n,1)$ so $k\ge0.$
Applying \eq{FOOOex} to $\b,\b'$ in place of ${\bf a}$ we get two exact sequences
\ea HF^n(\b,\b'')\to HF^n(\b,\b)\to HF^{n+1}(\b,\b')=0,\\
0=HF^{-1}(\b',\b)\to HF^0(\b',\b')\to HF^0(\b',\b'').\ea
On the other hand $HF^n(\b,\b)$ and $HF^0(\b',\b')$ are both non-zero, so
$HF^n(\b,\b'')$ and $HF^0(\b',\b'')$ are both non-zero.
Suppose contrary to our conclusion that $\b''$ is isomorphic to an object ${\bf a}\in\cF(X)$ supported on a closed special Lagrangian $L''$ of phase $\th\in\R.$
By perturbing $L''$ we may suppose that ${\bf a}$ is supported on a nearly special Lagrangian $L^*$ which intersect $L$ and $L'$ generically.
Since $HF^n(\b,{\bf a})\cong HF^n(\b,\b'')\ne0$ we get an intersection point $x\in L\cap L^*$ with Maslov index $n.$
Hence denoting by $\ph:L^*\to\R$ the phase function and using \eq{mu} we find
$-\ph(x)/2\pi\in(0,n)$ and in particular $\ph(x)<0.$
Also since $HF^0(\b',{\bf a})\ne0$ we get an intersection point $x'\in L'\cap L^*$ with
$k-\ph(x')/2\pi\in (-n,0)$ and in particular $\ph(x')>k\ge0.$
But both $\ph(x)$ and $\ph(x')$ are near $\th$ so $k=\th=0;$ that is, $L,$ $L'$ and $L''$ have all phase $0.$
Since $HF^n(\b,{\bf a})$ and $HF^0(\b',{\bf a})$ are both non-zero it follows by Proposition \ref{prop: TY} (i)
that these three Lagrangians have a common irreducible component.
But this is impossible because $L,L'$ intersect each other transversely at least at one point. 
\end{proof}

We recall now a key fact about {\it Morse} $1$-forms; 
a Morse $1$-form on a manifold $Q$ is a closed $1$-form $\al$ on $Q$ such that every point of $Q$ has a neighbourhood on which $\al=df$ for some Morse function $f.$
\begin{thm}[{Honda \cite{Hn}}]\l{thm: Honda}
Let $Q$ be a closed manifold of dimension $n,$ and $\al$ a Morse $1$-form on $Q$ with no critical point of index $0$ or $n.$
Then there exists on $Q$ a Riemannian metric $g,$ and in the de Rham class $[\al]\in H^1(Q,\R),$ a $g$-harmonic Morse $1$-form $\be$ which has the same distribution of indices as $\al$ has;
that is, for each $i\in\{0,\cdots, n\}$ the number of index-$i$ critical points of $\be$ is equal to that of $\al.$
\qed
\end{thm}
\begin{rem}
If $Q$ is equipped with any real analytic structure, we can suppose that $g$ and $\be$ above are real analytic. To see this, perturb the metric from Honda's theorem to a real analytic Riemannian metric, which we denote still by the same $g$. Choosing the perturbation to be sufficiently $C^\iy$ small, there exists a unique $g$-harmonic $1$-form near $\be$ in the same de Rham class, which we denote again by the same $\be.$ This $\be$ is analytic by elliptic regularity.
\end{rem}

We make now the definition of standard Calabi--Yau structures in cotangent bundles.
\begin{dfn} \label{dfn: standard CY}
Let $Q$ be a closed connected orientable Riemannian manifold. Then a {\it standard} Calabi--Yau structure near the zero-section $Q\subset T^*Q$ is a pair consisting of a compatible (integrable) complex structure $J$ and a nowhere-vanishing $J$-holomorphic $(n,0)$ form relative to which $Q$ is a special Lagrangian. When $Q$ is oriented, we require this to agree with that induced by $\Om|_Q.$
\end{dfn}
We show that standard Calabi--Yau structures exist as follows.
\begin{lem}\l{GS}
Let $(Q,g)$ be a closed oriented real analytic Riemannian manifold.
Then there exist a disc sub-bundle $X\sb T^*Q$ and a standard Calabi--Yau structure on $X$ such that the restriction of the K\"ahler metric to $Q$ is equal to $g,$ and the complex structure is compatible with the Liouville structure of $X$ $($in the sense of Definition $\ref{dfn:Liouv}).$
\qed
\end{lem}
\begin{rem}
The real analytic structure induced from the complex structure above is possibly different from that induced from the real analytic structures of $Q$ and $T^*Q.$
\end{rem}
\begin{proof}[Proof of Lemma $\ref{GS}$]
Guillemin and Stenzel \cite[\S5]{GS} prove this except the part which has to do with $\Om;$ that is, there exist a sufficiently small disc sub-bundle $X\sb T^*Q$ and a complex structure $J$ on $X$ compatible with the Liouville structure of $X$ and such that the restriction to $Q$ of the corresponding K\"ahler metric is equal to $g.$
Making $X$ smaller if we need, extend $g$ to a holomorphic section of $\sym^2T^*X$ which is locally of the form $g_{ab}(z^1,\dots,z^n)dz^adz^b$ where $z^1,\dots,z^n$ are local holomorphic co-ordinates with which $Q$ is defined by $\Im z^1=\dots=\Im z^n=0.$ This is possible because every $g_{ab}$ is real analytic.
Since $Q$ is oriented we get, making $X$ yet smaller if we need, a holomorphic volume form on $X$ which is locally of the form $\sqrt{\det g_{ab}}dz^1\wedge\dots\wedge dz^n$ and whose restriction to $Q$ agrees with the volume form of $(Q,g).$ This completes the proof.
\end{proof}
\begin{rem}
Guillemin and Stenzel \cite[\S5]{GS} do not suppose that $Q$ is oriented.
They show also that $J$ is invariant under the involution $\si:T^*Q\to T^*Q$ defined by $p\mapsto -p$ and that $J$ is the unique complex structure satisfying this and the other conditions in the lemma.
\end{rem}

We give now an example of {\it non-special} Lagrangian surgeries.
\begin{thm}\l{Ea}
Let $Q$ be a closed connected oriented manifold of $n,$ and $\al$ a Morse $1$-form on $Q$ whose critical points have the sof dimension $n$ $($which must be $\frac n2$ by the Poincar\'e duality of Morse--Novikov homology$).$ Then there exist a compact disc sub-bundle $X\subset T^*Q,$ a closed embedded branched $\Z$-graded Lagrangian $L\subset X,$ and an object of $H\cF(X)$ supported on $L$ that is not isomorphic to any other object supported on a closed special Lagrangian (of any phase
relative to any standard Calabi--Yau structure).
\end{thm}
\begin{proof}
Applying Theorem \ref{thm: Honda} to $\al$ we get a harmonic $1$-form $\be$ with the same distributions of critical points as $\al$ has.
Applying McLean's theorem to $\be$ we get a $C^1$ nearby special Lagrangian $Q'$ such that the pair $(Q',Q)$ has the same distribution of indices as $\al$ has.
Consider now $Q'[\frac n2-1]\#Q$ as in Theorem \ref{FOOO10}.
Corollary \ref{cor: non-special} implies then that this cannot be represented by a special Lagrangian, completing the proof.
\end{proof}
\begin{exam}\l{ex: non-special}
Let $n\ge 4$ be an even number and take $Q$ to be the self-connected sum of $S^{n/2}\t S^{n/2};$ that is, $Q$ is obtained from $S^{n/2}\t S^{n/2}$ by eliminating two points of it and attaching the handle between them.
Take a Morse function $S^{n/2}\t S^{n/2}\to\R$ with one critical point of index $0,$ two of index $\frac n2$ and one of index $n.$ Performing the connected sum construction at the critical point of index $0$ and that of index $n,$ we obtain a $1$-form $\al$ on $Q$ satisfying the hypothesis of the theorem above. There is accordingly a branched nearby Lagrangian $Q'[\frac n2-1]\#Q$ with $HF^*$ unobstructed. 
On the other hand, as $n\ge 4$ there is an isomorphism $\pi_1Q\cong\Z$ so the hypotheses of Theorem \ref{thm: main} (ii) holds except the special Lagrangian condition.

Let $\b$ be an object of the Fukaya category supported on $Q'[\frac n2-1]\#Q$ and let $q\in Q$ be any point. Then the Floer cochain group $CF^*(T_q^*Q,\b)$ is non-zero in degrees $\frac n2-1$ and $0.$ As $n\ge4$ these two degrees differ at least by two and no differentials exist between them. So the cohomology group $HF^*(T_q^*Q,\b)$ is non-zero in two different degrees. This shows that Lagrangians being special is essential to Lemma \ref{lem: single degree}.
\end{exam}

We give next an example of {\it special} Lagrangian surgeries.
\begin{thm}\l{E}
Let $n\ge3$ and let $Q$ have a  Morse $1$-form $\al$ with at least one critical point of index $1,$
at least one of index $n-1$ and none of index $0$ or $n.$
Then near $Q\subset T^*Q$ there exist a standard Calabi--Yau structure and relative to it a closed branched special Lagrangian.
\end{thm}
\begin{proof}
Following the proof of Theorem \ref{Ea} we get a $C^1$ nearby special Lagrangian $Q'$ such that the Maslov indices of the pair $(Q',Q)$ has the same distributions of the Morse index of $\al.$
So $(Q',Q)$ has at least one pair intersection points, one of index $1$ and the other of index $n-1,$ which thus satisfies the hypothesis of Joyce \cite[Theorem 9.7]{J5} (one side of the equation (62) in that theorem corresponds to index $1$ critical points, and the other side of the same equation corresponds to index $n-1$ critical points).
We can define then a special Lagrangian surgery $Q'\#Q$ as we want.
\end{proof}

The branched special Lagrangian in the preceding theorem is possibly immersed; any pairs of critical points of $\al$ of indices $1$ and $n-1$ may be removed by surgery but the other critical points become transverse double points.

\begin{exam}\l{higher-dim ex}
Take $Q$ to be the selfconnected sum of $S^1\t S^{n-1};$ then $Q$ has a Morse $1$-form with exactly two critical points, one of index $1$ and the other of index $n-1.$
So with $n\ge3$ we can apply Theorem \ref{E} and if $n\ge4$ we have $b_2(L)=0$ so that $L$ is $HF^*$-unobstructed.
\end{exam}
But there are also {\it $HF^*$-obstructed} special Lagrangian surgeries:
\begin{cor}\l{O}
Let $(Q,g)$ be a compact connected real analytic Riemannian manifold of dimension $n\ge3$ whose fundamental group $\pi_1Q$ has no non-abelian free subgroup. Let $\al$ be a Morse $1$-form on $Q$ with at least one critical point of index $1,$ at least one of index $n-1$ and none of index $0$ or $n.$ Let $X\sb T^*Q$ be a disc sub-bundle equipped with a standard Calabi--Yau structure $(J,\Om)$ as in Lemma $\ref{GS}.$
\iz
\item[\bf(i)]
If $n=3$ then there exists in $(X;J,\Om)$ a closed embedded $HF^*$-obstructed special Lagrangian which has $HF^*$ obstructed for every Novikov field of characteristic $0.$ Moreover, if $\pi_1Q$ is virtually solvable then the same Lagrangian has $HF^*$ obstructed for every Novikov field.
\item[\bf(ii)] If $n\ge4$ then there exists in $(X;J,\Om)$ a closed generically-immersed special Lagrangian which has $HF^*$ obstructed for every Novikov field of characteristic $0.$ Moreover, if $\pi_1Q$ is virtually solvable then the same Lagrangian has $HF^*$ obstructed for every Novikov field.
\iz
\end{cor}
\begin{proof}
For (i) recall from Poincar\'e--Hopf's theorem that the numbers of critical points of indices $1$ and $2$ are the same. So by Theorem \ref{E} there are branched embedded special Lagrangians. Note that we can re-scale $\al$ to make these special Lagrangians as close to $Q$ as we like. We can then use Theorem \ref{thm: main} (iii),(iv) to show that they have $HF^*$ obstructed in such a way as stated above. 
Part (ii) may be proved in the same way.
\end{proof}

\begin{exam}\l{ob ex}
For (i) take $Q=S^1\t S^2$ and regard it as the selfconnected sum of $S^3.$
According to Milnor \cite[Lemma 8.2]{Miln} for every Morse function we can produce, by changing it locally, two additional critical points of any adjacent indices.
Applying this to the height function of $S^3$ we get a Morse function with exactly one critical point of index $0,$ exactly one critical point of index $3$ and arbitrarily many pairs of critical points of indices $1$ and $2.$
So $Q$ has such a Morse $1$-form as in (i) above.

In the same way, on $Q=S^1\t S^{n-1}$ with $n\ge4$ there is such a Morse $1$-form as in (ii) above with arbitrarily many pairs of critical points of adjacent indices $1,2$ or $n-2,n-1.$
\end{exam}

Corollary \ref{O} implies a statement without Lagrangians in it and perhaps of interest in itself:
\begin{cor}\l{O2}
Let $Q$ be a closed connected orientable manifold of dimension $n$ with $\pi_1Q$ having no non-abelian free subgroup. 
Let $\al$ be a Morse $1$-form on $Q$ with no critical points of index $0$ or $n.$ Then
\iz
\item[\bf(i)]
if $n=2$ we have $Q\cong T^2$ and $\al$ has no critical point at all.
\iz
Suppose now that $n\ge3$ and $\al$ has at least one critical point of index $1$ and one critical point of index $n-1.$ Then
\iz
\item[\bf(ii)]
if $n\ge 4$ then either $b_2(Q)>0$ or $\al$ has some critical point of index $2$ or $n-2.$
\iz
\end{cor}

\begin{proof}
We first prove (i).
Since $\al$ has no critical points of index $0$ or $n$ it follows that $\al$ is not exact, so $b_1(Q)>0.$
If $n=2$ this condition and the hypothesis on $\pi_1Q$ imply $Q\cong T^2$, but since the Euler characteristic of $T^2$ vanishes, every Morse $1$-form on $T^2$ with no critical point of index $0$ or $2$ has no critical point at all, as in (i).

We then prove (ii). 
Since the immersed Lagrangian $L$ in Corollary \ref{O} (ii) has $HF^*$ obstructed,
it follows that either $b_2(L)>0$ or $L$ has some self-intersection points of index $2$ or $n-2.$
So $b_2(Q)>0$ (because $n\ge 4$) or $\al$ has some critical points of index $2$ or $n-2.$
\end{proof}

We apply the same method to $n=3$ and prove 
\begin{prop}\l{prop: n=3}
Let $Q$ be a closed connected orientable manifold of dimension $3$ with $\pi_1Q$ virtually solvable. Let $\al$ be a Morse $1$-form on $Q$ with no critical points of index $0$ or $3.$ Then there exists a smooth function $f:Q\to\R$ such that $\al-df$ has no critical points at all.
\end{prop}
\begin{proof}
Let $x,y$ be two critical points of $\al,$ one of index $1$ and the other of index $2;$ such a pair exists by Poincar\'e duality of Novikov homology. As in Corollary \ref{O} (i) there is a double point surgery $L$ with $HF^*$ obstructed.
In dimension $n=3$ the bounding cochains on $L$ are critical points of the disc potential function. The relevant homology classes come from those of strips between $x,y$ with the right orientation.
For this we use again Fukaya, Oh, Ohta Ono's result \cite[Chapter 10]{FOOO}; the number of holomorphic discs homologous to this strip with multiplicity one is conserved under the surgery. 
We prove that the number $\nu$ of holomorphic discs homologous to the multiplicity-one strip is exactly one.
This $\nu$ is an integer because $T^*Q$ is an exact symplectic manifold and we can therefore use the result of Fukaya, Oh, Ohta and Ono \cite{FOOO-Z}.
We can also work with the Novikov field of any characteristic $p$ so the Lagrangian $L$ has $HF^*$ obstructed in characteristic $p;$ and in particular, $\nu$ is non-zero modulo $p.$  But $p$ is any, so $\nu=1.$
Denote by $Q^\al$ the graph of $\al.$
As the Floer complex of the pair $(Q^\al,Q)$ and the Novikov complex of $\al$ are quasi-isomorphic, the number of gradient trajectories from $y$ to $x$ is also one. 
We can then use the cancellation theorem in Morse theory \cite[Theorem 5.4]{Miln}; that is, we can cancel out $x,y$ with each other by changing $\al$ in its de Rham cohomology class.
Repeating this process we get a Morse $1$-form with no critical points at all, as claimed. 
\end{proof}

The hypothesis in Proposition \ref{prop: n=3} that $\pi_1Q$ should be virtually solvable is stronger than that in Corollary \ref{O2}. It is therefore natural to ask whether the same conclusion holds under the same hypothesis on $\pi_1Q$ as in Corollary \ref{O2}. This is in fact true as we show now.
\begin{prop}\l{O3}
Let the hypothesis of Corollary $\ref{O2}$ hold with $n=3.$ Then there exists a smooth function $f:Q\to\R$ such that $\al-df$ has no critical points at all.
\end{prop}
For the proof we use the geometrisation theorem for $3$-manifolds. We begin with
\begin{lem}\l{lem: 3-manifold}
Every closed orientable $3$-manifold $Q$ satisfies one of the following five statements:
{\bf(i)} $\pi_1Q$ is finite;
{\bf(ii)} $Q\cong S^1\t S^2;$
{\bf(iii)} $Q\cong \R P^3\# \R P^3;$
{\bf(iv)} there exists a finite cover $P\to Q$ where $P$ is the mapping torus of an orientation-preserving diffeomorphism $T^2\to T^2;$ or
{\bf(v)} $\pi_1Q$ has a non-abelian free subgroup.
\end{lem}
\begin{proof}
We suppose that $\pi_1Q$ is infinite and prove that one of (ii)--(v) holds; we shall refer to the book \cite{AFW} for those results about $3$-manifolds which we use, without tracing back their origins in the literature. 
As $\pi_1Q$ is infinite, $Q$ has no Riemannian metric of positive constant sectional curvature, so $Q$ is not spherical \cite[\S1.7]{AFW}.
According to Theorem 1.11.1 of \cite{AFW}, either one of (ii)--(iv) above holds; or $\pi_1Q$ has no finite-index solvable subgroup. It thus remains to show that, if $Q$ is an infinite $3$-manifold group with no finite-index solvable subgroup, then it must have a non-abelian free subgroup.

The first step in this direction is to note that  \cite[p50, (C.6)]{AFW} asserts that our assumption on the fundamental group implies that  one of the following (a)--(d) holds:
(a) $Q$ is reducible \cite[p3]{AFW};
(b) $Q$ is a hyperbolic manifold and $\pi_1Q$ is a linear group over $\ov\Q$ \cite[p51, (C.7)]{AFW};
(c) $Q$ is a Seifert manifold and $\pi_1Q$ is a linear group over $\Z$ \cite[p52, (C.10) and (C.11)]{AFW}; or
(d) $Q$ is a Haken manifold with an incompressible torus and (v) holds \cite[p58, (C.24)]{AFW}. Moreover:
\iz
\item
In Case (a) it is well known that either (ii) holds or $Q$ is a connected sum of two closed $3$-manifolds $P_1,P_2$ neither of which is a sphere.
In the latter case $\pi_1P_1,\pi_1P_2$ are both non-trivial (because Poincar\'e's conjecture is true). If $\#\pi_1P_1=\#\pi_1P_2=2$ then $P_1\cong P_2\cong \RP^3$ and (iii) holds. Otherwise, $\pi_1Q\cong\pi_1P_1*\pi_1P_2$ has so many elements that (v) holds as follows: one of $\{\pi_1P_1,\pi_1P_2\}$ has at least two non-trivial elements $x,y$ and the other of $\{\pi_1P_1,\pi_1P_2\}$ has at least one non-trivial element $z$ so there is a non-abelian free group generated by $xz,yz.$
\item
In Cases (b) and (c) we use Theorem \ref{thm: Tits}.
Since we have supposed that $\pi_1Q$ has no finite-index solvable subgroup, it follows then that (v) holds in Cases (b) and (c). 
\iz
This completes the proof.
\end{proof}
Lemma \ref{lem: 3-manifold} implies 
\begin{cor}\l{cor: 3-manifolds}
Let $Q$ be a closed orientable $3$-manifold with $\pi_1Q$ having no non-abelian free subgroup.
Then every non-zero element of $H^1(Q,\R)$ may be represented by a nowhere-vanishing closed $1$-form on $Q.$
\end{cor}
\begin{proof}
By Lemma \ref{lem: 3-manifold} one of (i)--(iv) holds.
In Cases (i) and (iii) we have $H^1(Q,\R)=0$ and there is nothing to prove.
In Case (ii) we have $Q\cong S^1\t S^2$ and letting $x$ be the coordinate of $S^1$ we have $H^1(Q,\R)\cong\R[dx],$ which has certainly the desired property.
In Case (iv) the mapping torus $P$ is determined by the monodromy matrix $M\in \SL_2(\Z)$ acting upon $H^1(T^2,\Z)\cong\Z^2$ which we can classify by direct computation into the following three cases;
(a) there is an integer $k>0$ such that $M^k=\begin{pmatrix} 1& 0\\ 0& 1\end{pmatrix};$
(b) $M^2$ is conjugate in $\GL_2(\Z)$ to $\begin{pmatrix} 1& l\\ 0& 1\end{pmatrix}$ for some $l\in\Z\-\{0\};$ or
(c) $M$ has two distinct real eigenvalues, neither equal to $1.$
We study these respectively as follows:
\iz
\item
In Case (a) we can rechoose the finite cover $P\to Q$ so that $P=T^3,$
from which we get a natural embedding $H^1(P,\R)\sb H^1(T^3,\R)=\R[dx]\op\R[dy]\op\R[dz]$ where $x,y,z$ are the three coordinates of $T^3=S^1\t S^1\t S^1.$ This expression implies the desired property of $H^1(P,\R).$
\item
In Case (b) we can rechoose the finite cover $P\to Q$ so that $M$ will have two eigenvalues both equal to one.
Trivialising $P$ over two intervals in $S^1$ and applying to them the Mayer--Vietoris sequence, we get an isomorphism
\e\l{mapping torus} H^1(P,\R)/\R[dx]\to H^1(T^2,\R)^M\e
where $[dx]$ is obtained from the projection $x:P\to S^1$ and the right-hand side is the $M$-invariant subspace of $H^1(T^2,\R).$
As $M$ has two eigenvalues both equal to one, there is a coordinate $y:T^2\to S^1$ such that $dy$ is invariant under $M$ and lifts to $P.$
Then by \eq{mapping torus} we have $H^1(P,\R)=\R[dx]\op\R[dy],$ in which $H^1(Q,\R)$ has certainly the desired property.
\item
In Case (c) there is again an isomorphism of the form \eq{mapping torus} but the right-hand side vanishes because $M$ has no eigenvalue equal to one. 
So $H^1(P,\R)=\R[dx],$ in which $H^1(Q,\R)$ has certainly the desired property. 
\iz
This completes the proof.
\end{proof}
Corollary \ref{cor: 3-manifolds} implies that Proposition \ref{O3} holds as we have claimed.

\appendix

\section{Viterbo Restriction Functors}

The purpose of this appendix is to prove Proposition \ref{prop: restriction compatible}. The strategy for the proof is essentially the same as the construction of the Viterbo functor from  \cite{AS}, and we shall see that there are essentially no new ideas. However, since \cite{AS} only involves exact Lagrangians, we go through the process of explaining how the energy estimates still apply to prove well-definedness of operations when allowing one of the Lagrangians to be non-exact, as long as it lies entirely inside the target of the restriction map. We discuss the energy estimates in a very general context, then explain the algebraic arguments required to reach the desired conclusions, which includes in particular recalling the formulation of the wrapped Fukaya category using linear Hamiltonians, as this will allows us to more directly use the existing construction of the Viterbo functor.

\subsection{Topological and Geometric energies}
\label{sec:topol-geom-energ}

In the definition of wrapped Floer homology using linear Hamiltonians, Floer complexes are defined to have generators given by \emph{Hamiltonian chords}:
\begin{dfn}
Let $(X,\la)$ be a Liouville domain and $H:X\to[0,\infty)$ a $C^\iy$ function. Denote by $v_H$ its Hamiltonian vector field on $X$ and define an {\it $H$-chord} in $X$ to be a smooth map $x=x(t):[0,1]\to X$ with $\frac{dx}{dt}=v_H\cm x,$ $x(0)\in L_0$ and $x(1)\in L_1$ where $L_0,L_1$ are some Lagrangians in $X.$
\end{dfn}
The flow of $H$ gives a bijection between $H$-chords with endpoitns $(L_0,L_1)$ and intersection points between the image of $L_0$ under the time-$1$ flow of $H$ and $L_1$. It is possible to then recast all definitions in terms of such intersection points, but we shall not do so to keep consistency with  \cite{AS}.

The differential in Floer complexes in this setting is defined to be the count of solutions to the Floer equation
\begin{equation} 
  (du - v_H\circ u \otimes dt)^{0,1} = 0 ,
\end{equation}
on a strip $\mathbb{R} \times [0,1]$, with interval coordinates $t$, boundary conditions $L_0$ and $L_1$, and which have finite energy in the sense that
\begin{equation}
  \int |du - v_H \otimes dt|^2 < \infty.  
\end{equation}
An application of Stokes's theorem shows that this expression depends only on the homotopy class of the strip, and a second application shows that, in the case of exact Lagrangian it in fact depends only on the input and output. 

To define operations on Floer complexes generated by Hamiltonian chords, one removes boundary marked points of Riemann surfaces (which would be asked to map to intersection points between Lagrangians in the framework of Section \ref{dfn:Liouv}), and one counts solutions to a specific class of inhomogeneous Cauchy-Riemann equations, which in a neighbourhood of the punctures agree with the Floer equation. To be more precise, we need some terminology:   for $k\in\{0,1,2,\dots\}$ a {\it punctured disc with $k+1$ ends} is a non-compact bordered Riemann surface $\Si$ isomorphic to the closed unit disc in $\C$ minus $k+1$ distinct boundary points indexed by $\{0,\dots,k+1\}$ in the counter-clockwise order. In particular, $\Si$ is given a complex structure and hence orientation; and we index the $k+1$ boundary arcs also by $\{0,\dots,k+1\}$ in a counter-clockwise order, which we fix by saying that the $0^{\rm th}$ boundary arc is between the $0^{\rm th}$ and $1^{\rm st}$ boundary points. 

Note that there exist a holomorphic embedding $(-\iy,0]\times[0,1]\to \Si$ which parametrizes the $0^{\rm th}$ end of $\Si,$ and for $a=1,\dots,k,$ a holomorphic embedding $[0,\iy)\times[0,1]\to \Si$ which parametrizes the $a^{\rm th}$ end of $\Si.$ We make a universal choice of these that is compatible with isomorphisms, deformations and gluings of punctured discs \cite[\S2.2]{AS}. Denote by $s:(-\iy,0]\times[0,1]\to(-\iy,0]$ the first co-ordinate and $t:(-\iy,0]\times[0,1]\to [0,1]$ the second co-ordinate, which we call the $(s,t)$ co-ordinates. Define in the same way the $(s,t)$ co-ordinates of $[0,\iy)\times[0,1].$ 

We now arrive at the key additional datum which we must choose:
\begin{dfn}
 
Let $\Si$ be a punctured disc with $k+1$ ends. We say that a $1$-form $\chi$ on $\Si$ is {\it sub-closed} if $d\chi\le0$ with respect to the orientation of $\Si;$ that is, $d\chi=f\om_\Si$ where $f$ is some $0$-form $\Si\to(-\iy,0]$ and $\om_\Si$ an orientation $2$-form of $\Si.$  We assume as well that $\chi$ agrees,  along the strip-like end labelled by an integer $a \in \{0, \ldots, k\} $ , with a non-negative multiple of the $1$-form $dt$.
\end{dfn}
Most of the discussion in this Appendix can be performed using closed forms. The generality of subclosed forms is required for continuation elements (quasi-units) that enter in the choice of localisation.

The definition of the $A_\infty$ operation on the wrapped Floer complexes in this setting is given by  maps
\begin{equation}
  u :   \Si \to X,
\end{equation}
mapping the boundary conditions on a sequence $L_0,\dots,L_k\subset X$ of Lagrangians, which solve the equation
\begin{equation} \label{eq:equation_for_subclosed_form}
  (du - v_H \otimes \chi)^{0,1} = 0  
\end{equation}
on punctured discs with $k+1$ ends, for a specific family of subclosed $1$-forms and a family of  compatible almost complex structures on $X$ parametrised by this domain, and which have finite energy in the sense that the geometric energy
\begin{equation} \label{eq:geometric_energy}
    \int_\Si\frac12 |du-(v_H\cm u)\otimes\chi|^2
\end{equation}
is finite. We now arrive at the main notion which we shall later use as the Novikov exponent of operations on wrapped Floer complexes:
\begin{dfn}\l{dfn: top en}
The topological energy $\cE(u)$ of a solution to Equation \eqref{eq:equation_for_subclosed_form} is
\e \cE(u):= \int_\Si (u^*d\la- d[(u^*H)\chi]). 
\e
\end{dfn}
By Stokes's theorem, the topological energy depends only on the homotopy class of the map $u$.  A manipulation using the definition of the Hamiltonian flow further shows that
\begin{equation}
     \cE(u) =  \int_\Si\frac12 |du-(v_H\cm u)\otimes\chi|^2 -\int_\Si (u^*H)d\chi. 
\end{equation}
Our assumptions that $H$ is non-negative, and the fact that $d\chi\le0$, together  imply that the topological energy is positive.

\subsection{The homotopy method}
\label{sec:homotopy-method}

Consider a pair $(L^{\rho}_0, L^{\rho}_1)$ of paths of exact Lagrangians parametrised by an interval $[0,\infty)$, as well as a $1$-parameter family $H^{\rho}$ of Hamiltonians on the same domain, which satisfy the following stringent assumption:
\begin{equation}
  \parbox{33em}{the space of Hamiltonian $H^\rho$-chords with endpoints on $(L^{\rho}_0, L^{\rho}_1)$ is regular for all values of $\rho$, and is proper over the interval $[0,\infty)$.}  
\end{equation}
This means in particular that we have a canonical identification $x^{\rho} \mapsto x^{\rho'}$ between the $H^\rho$-chords connecting  $(L^{\rho}_0, L^{\rho}_1)$ and the $H^{\rho'}$-chords connecting  $(L^{\rho'}_0, L^{\rho'}_1)$ for different points in the interval.

In this setting, one would like to define an associated map of Floer complexes
\begin{equation} \label{eq:homotopy_map}
    CF^*(L_0^{\rho}, L_1^{\rho}; H^\rho) \to CF^*(L_0^{\rho'}, L_1^{\rho'}; H^{\rho'})
\end{equation}
for different parameter values. Even in the most basic case, this map is not given by the naive identification of generators, but requires a correcting scaling factor:
\begin{dfn}
Given a chord $x$ of $H$, the relative action difference $\cA(x^\rho,x^{\rho'})$ for parameters $\rho$ and $\rho'$ in $[1,\infty)$ is the sum of the area of the rectangle swept by the $1$-parameter family of chords interpolating between $x^{\rho}$ and $x^\rho$ and the difference
\begin{equation}
 \int_0^1H^\rho\cm x^\rho(t) - H^{\rho'} \cm x^{\rho'}(t) dt.
\end{equation}
\end{dfn}

  If $L_0,L_1$ are exact with primitive $h_0:L_0\to\R$ and $h_1:L_1\to\R$ respectively, there is a standard notion of $\cA_H(x)\in\R$ given by 
\e\l{A}
\cA_H(x):=h_1\cm x(1)-h_0\cm x(0)-\int_0^1x^*\la+\int_0^1H\cm x\,dt.  
\e
The relative action difference agrees in this case with the difference $\cA_{H^{\rho}}(x^{\rho}) -\cA_{H^{\rho'}}(x^{\rho'}) $. 

Given the above notion, we have:
\begin{lem}
  If the moduli spaces of solutions to Floer's equation submerse over the interval $[\rho,\rho']$, there is a chain-isomorphism of Floer complexes given by
  \begin{equation} \label{eq:action-change-naive-homotopy}
       x^{\rho} \mapsto T^{\cA(x_{\rho'},x_{\rho})} x^{\rho'}.
  \end{equation} \qed
\end{lem}

The homotopy method (going back to \cite[Section 19]{FOOO}) allows one to drop the condition that all solutions over the $[\rho,\rho']$ are regular: in this case, one considers a sequence $\rho_1 < \rho_2 < \cdots < \rho_k $ of intermediate times, and a sequence $(u_1, \cdots, u_k)$ of solutions to Floer's equations at these times, with asymptotic conditions which successively match in the sense that if the output of $u_i$ is $x_i^{\rho_i} $, then the input of $u_{i+1}$ is $x_{i}^{\rho_{i+1}}$. In this case, the expression for Equation \eqref{eq:homotopy_map} is a sum of terms, one for each such configuration, with the property that the curves $u_i$ have virtual dimension $0$. The $T$-adic valuation of the contribution of such a configuration is the sum of the topological energies of the strips $u_i$, and the contribution of Equation \eqref{eq:action-change-naive-homotopy} for each subinterval. Using the positivity of topological energy, we conclude:
\begin{lem}
  If a sequence $(u_1, \cdots, u_k)$ as above contributes to the homotopy map, then the valuation of its contribution is bounded below by the sum of the relative action differences $\cA(x_{i}^{\rho_{i+1}},x_{i}^{\rho_{i}})$, over all intermediate orbits. \qed
\end{lem}
This result, and its generalisations, will be used to show that certain operations must be $T$-adically convergent, as a consequence of geometric choices that ensure that the relative action differences are bounded below by a positive multiple of $\rho_{i} - \rho_{i-1}$. The purpose of the next sections is to indicate which configurations of curves are used to define the operations of interest, at which point a version of the above Lemma will yield the desired convergence.

\subsection{Viterbo restriction: the linear part}
\label{sec:viterbo-restriction-1}

In Section \ref{sec: ex}, we considered a wrapped Fukaya category $\cW(X)$, obtained by localisation. Here, we give a slightly different account, replacing Lagrangian isotopies with Hamiltonian continuation maps. For the later discussion, it is also convenient to allow $X$ to be either a Liouville manifold (with complete Liouville flow) or a compact Liouville domain (with boundary along which the flow is inward pointing). We use the term \emph{end of $X$} to refer either to an infinite end or to a neighbourhood of the boundary.

We begin by fixing a finite set of Lagrangians in $X$, which along the end are invariant under the Liouville flow. We then choose a Hamiltonian $H$ on $X$ so that, for all non-negative integers $w$, all Hamiltonian $w$-chords with endpoints on a pair of Lagrangians in this set are non-degenerate and are disjoint from the end. We can then define a variant of the wrapped category:  objects are pairs $(L,w)$ with $w$ a natural number, and, before localisation, there are no morphisms from $(L,w)$ to $(L',w')$ unless either (i)  $w < w'$ or (ii) $w=w'$ and $L=L'$, in which case it is given by the Floer group $CF^*(L,L'; (w'-w) \cdot H)$ (we are omitting the choices of bounding cochains and local systems for the notation). We then obtain a category by localising at continuation elements in $HF^*(L,L; H) $, thought of as the morphisms form $(L,w)$ to $(L,w+1)$.

Let us now consider the inclusion of a subdomain $X \subset E$ in a complete Liouville domain. We fix now a finite set of Lagrangians in $E$ whose intersection with $X$ are objects of $\cW(X)$, so that if an element $L$ of this set has non-empty boundary, it is exact outside of $X$, and admits a primitive which is constant near the boundary. For the choices required to define the Viterbo restriction functor, we start by restating a basic result from \cite{AS}, which ensures positivity of actions:
\begin{lem}[Abouzaid--Seidel {\cite[Lemma 7.3]{AS}}]\l{lem: AS7.3}
There is a  $C^\iy$ function  $H:E\to(0,\iy)$ be a $C^\iy$ function which on the collar agrees with the radius function $r$ of $(E,\la)$ and near $\partial X$ agrees with the radius function $r^{\rm in}$ of $(X,\la|_X)$, and such that $H-dH(v_\la)>0$ at every point of $E\-X,$ where $v_\la$ is the vector field on $E$ defined in Definition \ref{dfn:Liouv}. 
\end{lem}
\begin{cor}
Given any pair $(L_0,L_1)$ of Lagrangians in $E$ which are exact in the complement of $X$, there exists a constant $\nu>0$ such that for every $w\ge\nu$ and for every $wH$-chord $x:[0,1]\to X\-E$ with $x(0),x(1)\in L$ we have $A_{wH}(x)\ge\al.$ \qed
\end{cor}
Up to rescaling, we may assume that $\nu=1$ in the above result.
%


The key idea in the construction of the Viterbo functor in \cite[\S4]{AS} is a family of Hamiltonian parametrised by $\rho \in (0,1]$, for which the actions of chords in $X$ are scaled by $\rho$, while those outside of $X$ remain constant. This is achieved by rescaling $H$ along the Liouville flow inside $X$. We shall instead consider a family parameterised by  $\rho \in [1,\infty)$,  $(H_\rho)_{\rho\in[1,\iy)}$ of Hamiltonians which are constant in $X$ and have the effect of rescaling outside $X$. For the next definition, denote by $\ps^r:E\to E$ the flow of $v_\la$ at time $\log r.$
\begin{dfn}\l{dfn: H^rho}
   Given a function $H$ on $E$ which is linear near $\partial X$ and on the collar of $E$, and a real number $\rho\in[1,\infty)$ define a smooth function $H^\rho:E\to(0,\iy)$ by $H^\rho:=H$ on $X$ and $H^\rho = \rho(H\cm \ps^{1/\rho})$ on 
   $E\-X.$

   Given a Lagrangian $L$ which is stable under the Liouville flow on a neighbourhood of $\partial X$ and on the end of $E$, we similarly define $L^{\rho}$ to agree with $L$ in the region $L \cap X$, and to be given by the image of $L$ under $\ps^{\rho}$ in its complement.
\end{dfn}
Note that, if $L$ lies in $X$, this means that $L^\rho$ does not depend on $\rho$. We have the following basic computations:
\begin{lem}
  Given a pair of Lagrangians $(L_0,L_1)$ in $E$ satisfying the properties in Definition \ref{dfn: H^rho}, and $H$ a function so that there is no $H$-chord with endpoints $L_0$ and $L_1$ in a neighbourhood of $\partial X$ and in the cylindrical end of $E$, the flow $\psi^\rho$ maps $H$-chords with endpoints on $(L_0,L_1)$ bijectively to  $H^\rho$-chords with endpoints on $(L^\rho_0,L^\rho_1)$. \qed
\end{lem}
In this setting, every $H$ chord from $L_0$ to $L_1$ lies either in $X\-\partial X$ or in $(E\-\partial E)\- X.$ We call this an {\it inner} chord in the former case and an {\it outer} chord in the latter case.

We may thus use the notation of Section \ref{sec:homotopy-method}, and write $x^{\rho}$ for the image of an $H$-chord $x$, with endpoints $L_0$ and $L_1$. Given how the definitions of $H^\rho$ and $L^\rho$ involve the Liouville flow, we have:
\begin{lem} \label{lem:action_rescale}
  For any inner $H$-chord $x$ with endpoints on $(L_0,L_1)$, the relative action difference associated to an interval $[\rho,\rho']$ vanishe. There is a constant $C$, so that the relative action difference is bounded below by $C \cdot (\rho'- \rho)$ if $x$ is an outer chord.
  \qed
\end{lem}
Assuming that all such chords are non-degenerate, we may now setup the homotopy map as a map of chain complexes
\begin{equation}
 CF^*(L_0,L_1; H) \to  CF^*(L^\rho_0,L^\rho_1; H^\rho).
\end{equation}
While there is some flexibility in the definition of such a map, we assume that our the homotopy restricts to constant data on $M$; this is a statement about the choice of almost complex structure not varying with respect to $\rho$ in this region..
\begin{lem}
  Assuming that $H$ satisfies the conditions of Definition \ref{lem: AS7.3}, there is a value $\rho_0$ so that the outer chords generate a subcomplex of $CF^*(L^\rho_0,L^\rho_1; H^\rho)$ whenever $\rho_0 \leq \rho $.  The quotient complex is isomorphic to $CF^*(L_0 \cap X, L_1 \cap X ; H|_X)$, and the map
  \begin{equation}
      CF^*(L_0,L_1; H)  \to    CF^*(L_0\cap X , L_1\cap X ; H|_X)
    \end{equation}
    obtained by composing the homotopy map with the projection does not depend on $\rho$ in this range.
\end{lem}
\begin{proof}
  Lemma \ref{lem:action_rescale} implies that there is a choice of $\rho_0$ past which all outer chords have actions larger than inner chords. In that case, there can be no holomorphic strip with outer input and inner output, so we conclude both that outer chords form a subcomplex and that the components of the continuation map taking value in inner chords do not depend on $\rho$ past this value. The last statement is an application of the integrated maximum principle.
\end{proof}

\subsection{Viterbo restriction: compatibility with products}
\label{sec:viterbo-restr-cohom}

Let us now consider a compact, not necessarily exact, Lagrangian $K \subset X$, equipped with brane data and a bounding cochain which we supress from the notation, as well as a exact Lagrangian $L$ in $E$ for which the Viterbo restriction map to the wrapped Fukaya category of $X$ is defined. Consider a Hamiltonian $H$ as in Definition \ref{dfn: H^rho}, and pick an almost complex structure on $E$, of contact type along the end, which extends an almost complex structure on $X$ of contact type near $\partial X$. The integrated maximumm principle of \cite{AS} implies:
\begin{lem} \label{lem:integrated_maximum_gives_identifications}
The homotopy method isomorphism $ CF^*(L,K; H)  \cong CF^*(L^{\rho},K; H^{\rho}) $ is also given by the canonical identification of generators. This identification also induces an isomorphism $CF^* (L^{\rho},K; H^{\rho}) \cong CF^*(L \cap X, K; H|_{X})$ for all $\rho$.  \qed
\end{lem}

We next assume that we have a pair of exact Lagrangians $(L_0, L_1)$, for which Viterbo restriction is defined, and consider the diagram
\begin{equation} \label{eq:restriction_preserves_product}
  \begin{tikzcd}
    CF^*(K, L_0; w H) \otimes  CF^*(L_0,L_1; w' H)  \ar[r] \ar[d] &  CF^*(K, L_1; (w+w') H)  \ar[d]  \\
    CF^*(K, L^{\rho}_0; w H) \otimes  CF^*(L_0,L_1; w' H)  \ar[r] &  CF^*(K, L_1; (w+w') H) .
  \end{tikzcd}
\end{equation}
The method of Section \ref{sec:homotopy-method} extends to give a homotopy in this diagram. The key moduli spaces to consider are moduli spaces of pseudo-holomorphic triangles, with boundary conditions $(K, L^{r}_0,  L^{r}_1)$ for $r < \rho$, and equipped with the subclosed $1$-form which agrees, in strip-like ends, with $w dt$ and $w' dt$ near the punctures with adjacent edges labelled $(K,L^{r}_0)$ and $(L^{r}_0, L^{r}_1)$. We choose this $1$-form to be independent of the parameter $r$ in the interval $[0,\infty)$, i.e. to be the one used to define the Floer product. The definition requires as well a choice of a family of almost complex structure, whose restriction to $X$ we assume not to depend on $r$. With this, we obtain a moduli space $\Mbar^{[0,\rho]}(y_0, x; y_1)$ for each triple of Hamiltonian chords with $y_i$ starting on $K$ and ending on $L_i$, and  $x$ starting on $L_0$ and ending on $L_1$. The following result then follows from the integrated maximum principle:
\begin{lem}
 If $x$ lies in the subdomain $X$, then  $\Mbar^{[0,\rho)}(y_0, x; y_1)$ is the product of the interval $[0,\rho]$ with the space $\Mbar(y_0, x; y_1)$ of pseudo-holomorphic triangles. \qed
\end{lem}

The homotopy in diagram \eqref{eq:restriction_preserves_product} is given by considering a sequence of real numbers $1 < r_1 < \cdots < r_k < \rho$, a sequence of chords $x_1, \ldots, x_k$ with endpoints $L_0$ and $L_1$, and a collection $(u_1, \ldots, u_k)$ of maps from Riemann surfaces to $E$ so that the following properties hold:
\begin{enumerate}
  \item For $1 \leq i < k$, the map $u_i$ is a solution of Floer's equation, with boundary conditions $(L^{r_i}_0, L^{r_i}_1)$ and asymptotic conditions $x_i^{r_i}$ and $x_{i+1}^{r_i}$.  
  \item The map $u_k$ is an element of the moduli space $ \Mbar^{[0,\rho)}(y_0, x; y_1)$ a pseudo-holomorphic triangle with boundary conditions $(K, L^{r_k}_0,  L^{r_k}_1)$, and asymptotic conditions $(y_0, x_k^{r_k}; y_1)$.
  \end{enumerate}
  The $T$-adic valuation of the contribution of each such configuration is given by the sum of the topological energy of each disc with the sum of the relative action differences $\cA(x_i^{r_{i+1}},x_i^{r_{i}}) $  for each intermediate chord $x_i$. Using the fact that the topological energy is positive, as well as Lemma \ref{lem: AS7.3}, we now arrive at the main estimate:
  \begin{lem}
    The $T$-adic valuation of a configuration contributing to the homotopy is bounded below by a constant multiple of $r_k$ (the element of $[0,\rho])$ corresponding to the last disc. \qed
  \end{lem}
    The above result implies that the homotopies in Diagram \eqref{eq:restriction_preserves_product} satisfy the following properties: if $\rho < \rho'$, then the difference between the homotopies associated to $\rho$ and $\rho'$ has valuation bounded below by a constant multiple of $\rho$. This means that we obtain a well-defined map by letting $\rho$ go to infinity:
  \begin{cor}
    The following diagram, in which the vertical map on the right is the Viterbo restriction map, commutes up to homotopy:
\begin{equation} 
  \begin{tikzcd}
    CF^*(K, L_0; w H) \otimes  CF^*(L_0,L_1; w' H)  \ar[r] \ar[d] &  CF^*(K, L_1; (w+w') H)  \ar[d]  \\
    CF^*(K, L_0 \cap X ; w H|_{X}) \otimes  CF^*(L_0 \cap X ,L_1 \cap X; w' H|_{X})  \ar[r] &  CF^*(K, L_1; (w+w') H|_{X}).
  \end{tikzcd}
\end{equation}    
  \end{cor}
  The above result says that, at the cohomological level, the Viterbo restriction map is compatible with the Floer product for a triple $(K,L_0,L_1)$ with $K$ compact and $L_i$ exact for $i = 0,1$, extending the result of \cite{AS} for triples of exact Lagrangians. There is a completely analogous argument showing compatibility of restriction with the Floer product for triples $(L_0,L_1,K)$. The remaining cases, in which two of the three Lagrangians are compact, or in which the triple is $(L_0,K,L_1)$, are even easier to address, as the restriction map is just given by the identity on all Floer groups. We conclude:
  \begin{cor} \label{cor:cohomology_restriction}
The cohomology-level restriction functor from the exact Fukaya category of $E$ to that of $X$  constructed in \cite{AS} extends to the cohomological wrapped Fukaya category with objects which are either (i) exact Lagrangians equipped with primitive which is constant on $\partial X$, or (ii) compact Lagrangians in $X$. \qed
  \end{cor}

\subsection{Wrapped categories using linear Hamiltonians}
\label{sec:wrapp-categ-using}

In order to extend Corollary \ref{cor:cohomology_restriction} to a statement about $A_\infty$ categories, need an auxiliary definition of wrapped Fukaya categories. To this end, we introduce variants of the category $\cC(X)$ from Hypothesis \ref{hp: branes}:
\begin{dfn}
The curved $A_\iy$ categories $\cC'(X)$ and $\cC''(X)$ have objects pairs consisting of an object of $\cC(X)$ and  a non-negative integer.  Let $L_0,L_1\subset X$ be underlying Lagrangians of $\cC(X)$ and regard these as branes, leaving local systems out of the notation. For $(L_0,k_0),(L_1,k_1)\in \obj\cC(X)\times\{0,1,2,\dots\}$ define 
\e
\hom_{\cC'(X)}((L_0,k_0),(L_1,k_1)):=
\begin{cases} CF^*(L_0,L_1;(k_0-k_1)H)\text{ for }k_0>k_1\\ 
CF^*(L_0,L_1) \text{ for } k_0=k_1   \text{ and}\\
0 \text{ for }k_0<k_1,
\end{cases}
\e
Define $\hom_{\cC''(X)}((L_0,k_0),(L_1,k_1))$ to be the same for $k_0\ne k_1$ and to be $\La$ for $k_0=k_1.$ Denote by $\cW'(X)$ the localization of $\cC'(X)$ with respect to continuation element from Definition \ref{localisation}, as well as the elements  $\{\ka_{L,k}\in \hom_{\cC'(X)}((L,(k+1)),(L,k))\}$ associated to the homotopy from $kH$ to $(k+1) H$. We denote by $\cW''(X)$ the localization of $\cC''(X)$ with respect to the same continuation morphisms.

We define on $\cC'(X)$ and $\cC''(X)$ the $A_\iy$ structures in the same way as in \S\ref{sec: Fuk2}, choosing domain-dependent almost complex structures and using \eq{eq:equation_for_subclosed_form} in place of the plain Cauchy--Riemann equation. 

There is then an $A_\iy$ inclusion $\cC(X)\to \cC'(X)$ mapping each underlying Lagrangians $L$ to $(L,0)\in\obj\cC'(X)$, and which is the identity on morphism spaces. Theh construction of homotopy units \cite[Chapter 8]{FOOO} defines an $A_\iy$ functor $\cC''(X)\to \cC'(X)$ acting as the identity between the same object sets and whose $\fm^k$ operations factor through the identity map $\hom_{\cC''(X)}(\xi,\et)\to \hom_{\cC'(X)}(\xi,\et)$ for each pair of distinct objects $\xi,\et\in \obj\cC'(X)=\obj\cC''(X)$, and whose linear term maps $1$ to a representative of the unit. 

\end{dfn}

Using Proposition \ref{prop: loc} we have:
\begin{lem}
The $A_\iy$ functors $\cC(X)\to \cC'(X) \leftarrow  \cC''(X) $ induce, after adding bounding cochains and passing to localizations, homotopy equivalences $\cW(X)\to \cW'(X) \leftarrow  \cW''(X)$. \qed 
\end{lem}


Composing the functor $\cW(X)\to \cW'(X)$ and a homotopy inverse $\cW'(X)\to \cW''(X)$ we get a homotopy equivalence from $\cW(X)$ to the subcategory of $\cW''(X)$ with object set $[\obj\cC(X)]\times\{0\}.$ The latter will in practice be useful in the next subsection. Denote by $\cF_0(X)\sb\cF''(X)$ and $\cW_0(X)\sb\cW''(X)$ the subcategories with object set  $[\obj\cC(X)]\times\{0\},$ which we  identify with $\obj\cC(X).$

\begin{rem}
Here we use $A_\iy$ localizations repeatedly, and for the definition of $\cC'(X)$ or $\cC''(X)$ we only use sub-closed $1$-forms for the definition of the continuation elements, so that we do not require popsicles as in the original paper \cite{AS}.  This is one of the main advantages of the localization approach.
\end{rem}

\subsection{Proof of Proposition \ref{prop: restriction compatible}}
\label{sec:proof-proposition}

We now outline how the results of Appendix \ref{sec:viterbo-restr-cohom} can be extended to obtain an $A_\infty$ restriction functor:
\begin{proof}[Proof of Proposition \ref{prop: restriction compatible}]

  The restriction functor $\cW''(E)\to\cW''(X)$ is given by a parametrized count of pseudo-holomorphic curves, defining maps
 \begin{equation}\l{Vit1}
    V^k:  CF^*(L_0,L_1;w_1H)\times \dots \times CF^*(L_{k-1},L_k;w_kH)   \to    CF^*(L_0\cap X, L_k\cap X;w_0H|_X).
    \end{equation}
for each sequence of Lagrangians defining objects of our category.

The geometric construction is essentially the same as that of Abouzaid--Seidel \cite[\S4]{AS} except that we use virtual counts and that we change the  Hamiltonians in \cite[\S4.2]{AS} as discussed in Appendix \ref{sec:viterbo-restriction-1}, where we defined the linear term $V^1$ of the functor. The quadratic term $V^2$ is the homotopy introduced in Appendix \ref{sec:viterbo-restr-cohom}.

The homotopy is produced from counts of cascades, which are configuration of pseudo-holomorphic curves associated to the vertices of a directed tree, so that the curve labelled by a vertex $v$ has boundary on a subsequence $\{L_i^{\rho_v}\}$, for some parameter $\rho_v \in [1,\infty)$. The key point to have in mind from the construction in \cite{AS} is that, if $v < v'$ with respect to the ordering associated to the tree, then there is chord $x$ so that the label of the output at $v$ is $x^{\rho_v}$, and the corresponding input at $v'$ is $x^{\rho_{v'}}$.

Exactly as in \cite{AS}, the integrated maximum principle together with a choice of data that is $\rho$-independent in $X$  ensures that the only non-trivial contributions are associated to vertices where at least one input is an outer orbit.

The proof that $V^k$ is $T$-adically convergent then amounts to showing that the valuation of the contribution of any cascade is bounded below by a constant multiple of $\max_{v} \rho_v$. This is exactly the same argument as in Appendix \ref{sec:viterbo-restr-cohom}: each vertex contributes a non-negative amount to the valuation, because of the positivity of topological energy, while each edge contributes the action differences of the corresponding homotopy map. Since every vertex must have at least one incoming outer edge, the result then follows from the action estimate in Appendix \ref{sec:viterbo-restriction-1}.




  Using the $A_\iy$ functors $\cW(E)\to\cW'(E) \leftarrow \cW''(E)$ and choosing homotopy inverses $\cW(X) \leftarrow \cW'(X)\to\cW''(X)$ we get an $A_\iy$ functor $\cW(E)\to\cW(X)$. 
\end{proof}


\end{document}